\documentclass[a4paper,11pt]{article}
\usepackage{amsmath}
\usepackage{amsthm}
\usepackage{amsfonts}
\newcommand{\BigO}[1]{\ensuremath{\operatorname{O}\bigl(#1\bigr)}}
\usepackage{algorithm}
\usepackage{algpseudocode}
\usepackage{comment}
\usepackage{amsthm}
\usepackage{algpseudocode}
\usepackage{amssymb}
\usepackage{algorithm}
\usepackage{graphicx}
\usepackage{epstopdf}
\usepackage[export]{adjustbox}
\usepackage{color}
\usepackage{float}
\usepackage{pgfplots}
\usepackage[normalem]{ulem}
\usepackage{fullpage}
\usepackage{algpseudocode}
\usepackage{authblk}

\newcounter{dummy} \numberwithin{dummy}{section}
\newtheorem{theorem}[dummy]{Theorem}

\newtheorem{definition}[dummy]{Definition}
\newtheorem{proposition}[dummy]{Proposition}
\newtheorem{corollary}[dummy]{Corollary}
\newtheorem{conjecture}[dummy]{Conjecture}
\newtheorem{question}[dummy]{Question}
\newtheorem{remark}[dummy]{Remark}
\newtheorem{lemma}[dummy]{Lemma}

\date{}
\makeatletter
\newcommand\footnoteref[1]{\protected@xdef\@thefnmark{\ref{#1}}\@footnotemark}
\makeatother
\newlength\figureheight \newlength\figurewidth
\newcommand{\mytilde}{\raise.17ex\hbox{$\scriptstyle\mathtt{\sim}$}}
\title{Interference Queueing Networks on  Grids}

\author{%

	Abishek Sankararaman \footnote{Department of ECE, The University of Texas at Austin. Email - abishek@utexas.edu.},\hspace{4mm}
	Fran\c cois Baccelli, \footnote{Department of Mathematics and ECE, The University of Texas at Austin. Email - baccelli@math.utexas.edu.}
	\hspace{4mm}
	Sergey Foss \footnote{School of Mathematical Sciences, Heriot-Watt University,
		Edinburgh and Novosibirsk State University and
		Sobolev Institute of Mathematics, Novosibirsk, Russia. Email - S.foss@hw.ac.uk}

}

\begin{document}
	\maketitle

\begin{abstract}

Consider a countably infinite collection of interacting queues, with a queue located at each point of the $d$-dimensional integer grid, having independent Poisson arrivals, but dependent service rates.
The service discipline is of the processor sharing type,
with the service rate in each queue slowed down, when the neighboring queues have a larger workload. The interactions are {\color{black} translation invariant in space and is} neither of the Jackson Networks type, nor of the mean-field type.
Coupling and percolation techniques are first used to show that this dynamics has well defined trajectories.
Coupling from the past techniques are then proposed to build its minimal stationary regime. The rate conservation principle of Palm calculus is then used to identify the stability condition of this system, where the notion of stability is appropriately defined for an infinite dimensional process. We show that the identified condition is also necessary in certain special cases and conjecture it to be true in all cases.
Remarkably, the rate conservation principle also provides a closed form expression for the mean queue size. When the stability condition holds, this minimal solution is the unique translation invariant stationary regime. In addition, there exists a range of small initial conditions for which the dynamics is attracted to the minimal regime. Nevertheless, there exists another range of larger though finite initial conditions for which the dynamics diverges, even though stability criterion holds.



\end{abstract}

%



\section{Introduction}

In this paper, we consider a spatial queueing network consisting of an infinite collection of processor sharing queues {
interacting with each other in a translation invariant way}. In our model, there is a queue located at each grid
point of $\mathbb{Z}^d$, for some $d \geq 1$. The queues evolve in continuous time and serve the customers according
to a generalized processor-sharing discipline. The arrivals to the queues form a collection of i.i.d. Poisson Point
Processes of rate $\lambda >0$. Thus, the total arrival rate to the network is infinite since there is an infinite
number of queues. The different queues interact through their departure rates. We model the interactions through
an \emph{interference sequence} that we denote by $\{a_i\}_{i \in \mathbb{Z}^d}$. It is such that $a_i \geq 0$
and $a_i = a_{-i}$ for all $i \in \mathbb{Z}^d$. We also assume that this sequence is finitely supported,
i.e., $L:= \max \{||i||_{\infty}: a_i > 0\} < \infty$. For ease of exposition, 
we also assume that $a_0 = 1$ in certain sections of the paper, although our model and its analysis can be carried out for any non-zero value
of $a_0$. For any $t\in \mathbb{R}$, let $\{x_i(t)\}_{i \in\mathbb{Z}^d} \in \mathbb{N}^{\mathbb{Z}^d}$
denote the queue lengths at time $t$ in the network, i.e., the state of the system at time $t$.
Then \emph{the interference experienced by a customer located in queue $i$ at time $t$} is defined as 
$\sum_{j \in \mathbb{Z}^d}a_j x_{i-j}(t)$, i.e., some weighted sum of queue lengths of the
\emph{neighbors} of queue $i$. Observe that the neighborhood definition is translation invariant.
Conditional on the queue lengths $\{x_i(t)\}_{i \in \mathbb{Z}^d}$ at time $t$, the instantaneous
departure rate from any queue $i$ at time $t$ is given by $\frac{x_i(t)}{\sum_{j \in \mathbb{Z}^d} a_j x_{i-j}(t)}$,
with $0/0$ interpreted as being equal to $0$. Note that since the interference sequence $\{a_i\}_{i \in \mathbb{Z}^d}$
is non-negative, and $a_0 = 1$, for all $t \in \mathbb{R}$ and all $i \in \mathbb{Z}^d$,
the instantaneous departure rate from queue $i$ at time $t$ is $\frac{x_i(t)}{\sum_{j \in \mathbb{Z}^d} a_j x_{i-j}(t)} \in [0,1]$
and is hence bounded. Since $\{a_i\}_{i \in \mathbb{Z}^d}$ is non-negative, the rate of service at a queue is reduced
if its `neighbors' have larger queue lengths. This is meant to capture the fundamental spatiotemporal dynamics
in wireless networks where the instantaneous rate of a link is reduced if there are a lot of other links
accessing the spectrum nearby, due to an increase of interference. {\color{black} In the rest of the paper, we shall always assume that there exists at least one $i \in \mathbb{Z}^d \setminus \{0\}$ such that $a_i > 0$. For otherwise, the system is `trivial', as the queues evolve independent of each other without any interaction amongst them, according to a standard $M/M/1$ dynamics with unit service rate. Observe that the Markovian dynamics of our model is non-reversible and does not fall under the class of generalized Jackson networks. This model is also not of the mean-field interacting queues type such as the supermarket model \cite{dobrushin}, which admit a form of `asymptotic independence' across queues, as the system sizes get large.}
\\

This model is motivated by fundamental design questions in wireless networks.
The motivation for this particular model comes from certain mathematical questions about
such wireless dynamics left open in \cite{sbd_tit}. In our model, we view the queues {\color{black}as representing}
`regions of space' and the customers in each queue to be the wireless links in that region of space.
One can interpret a link or customer to be a transmitter--receiver pair, with the transmitter
transmitting a file to its intended receiver. For simplicity, we assume that the links are very tiny,
i.e., a single customer represents both the transmitter and receiver. The links share the wireless spectrum
in space and hence they impact each other's performance due to interference.
We assume that links arrive `uniformly' in space, and each transmitter has a file
whose length is exponentially distributed to transmit to its receiver.
A link departs and leaves the network once the transmitter has finished sending
the file to its receiver. We model the instantaneous rate of communication any 
transmitter can send to its own receiver as being inversely proportional to the interference
seen at the receiver, i.e., as $\frac{1}{\sum_{j \in \mathbb{Z}^d}a_jx_{i-j}(t)}$.
This can be viewed as the low `Signal-to-Noise-and-Interference-Ratio (SINR)' channel capacity
of a point-to-point Gaussian channel (see \cite{cover_book}). 
Since there are $x_i(t)$ links simultaneously transmitting, and each of them has an independent
unit mean exponentially distributed file, the rate at which a link departs is then 
$\frac{x_i(t)}{\sum_{j \in \mathbb{Z}^d}a_jx_{i-j}(t)}$. The instantaneous rate of transmission of a link is lowered
if it is in a `crowded' area of space, due to interference, and hence it takes longer for this
link to complete the transmission of its file. In the meantime, it is more likely that a new link
will arrive at some point nearby before it finishes transmitting, further reducing the rate of transmission.
Understanding how the network evolves due to such spatio temporal interference dynamics is crucial
in designing and provisioning of wireless systems (see discussions in \cite{sbd_tit}).
\\

The central thrust of this paper is to understand when the above described model is \emph{stable}. 
By {\it stability}, we mean {\it stabilization in time} of the distribution of the infinite-dimensional queue-length vector. Traditionally, this means that the distribution of any finite-dimensional restriction of the vector converges weakly to the limiting one. In fact, in this paper, we introduce an appropriate coupling construction to investigate a stronger version of the sample-path stability (or boundedness). We show the coupling-convergence of finite-dimensional vectors (that imply convergence in the total variation norm), using the so-called Loynes backward representation of the system dynamics (see, e.g., \cite{loynes}). The latter means that we fix initial (non-random) values of the queue-length process, start with this values at time $-t$ and observe the queue lengths $\{x_{i;t}(0)\}_{i\in {\mathbb{Z}}^d}$ at time $0$. Then we let $t$ tend to infinity. We begin with all-zero initial values. We establish certain monotonicity properties to conclude that, in the case of zero initial values, $x_{i;t}(0)$ increases a.s. with $t$, for any $i$. Therefore, the limit $x_i\equiv x_{i;\infty}(0) =\lim_{t\to\infty} x_{i;t}(0)$ exists {\color{black}a.s..} It may be either finite or infinite, where each occurs with probability either zero or one (See Lemma \ref{lem_01_law} in Section \ref{sec:math_framework}).  This is the {\it minimal} stationary regime:   any other stationary  regime, say $\{y_i\}$ must satisfy $x_i\le y_i$, for all $i$.  Then we identify a {\it sufficient} condition for stability, i.e., for the finiteness of the minimal stationary regime. Remarkably, we are able to provide an {\it exact} formula for the mean queue length of the minimal stationary solution. 

\begin{theorem}
	If $\lambda < \frac{1}{\sum_{j \in \mathbb{Z}^d}a_j}$, then the system $\{x_i(\cdot)\}_{i \in \mathbb{Z}^d}$ is stable. Furthermore, for all $i \in \mathbb{Z}^d$ and $s \in \mathbb{R}$,  the minimal stationary solution $\{x_{i;\infty}(s)\}_{i \in \mathbb{Z}^d}$ satisfies
		\begin{align*}
		\mathbb{E}[x_{i;\infty}(s)] &= \frac{\lambda a_0}{1-\lambda \sum_{j \in \mathbb{Z}^d} a_j}.  
		\end{align*}
		\label{thm:main_stability}
\end{theorem}

The proof of this theorem is carried out in Section \ref{sec:sids}, with some accompanying calculations in Section \ref{sec:rcl}. In the rest of the paper, the condition $\lambda < \frac{1}{ \sum_{j \in \mathbb{Z}^d}a_j}$ will be referred to as the stability criterion for the system. In this theorem, we only considered whether there exists a   stationary solution to the dynamics. However, as our network consists of infinitely many queues, uniqueness of stationary solutions is not guaranteed. {\color{black} In this paper, we are mainly concerned with stationary solutions of queue lengths that are translation invariant in space. Formally, a stationary solution $\{y_i\}_{i \in \mathbb{Z}^d}$ is said to be translation invariant in space if, for all $x \in \mathbb{Z}^d$, the law $\{y_{i-x}\}_{i \in \mathbb{Z}^d}$ is identical to that of $\{y_i\}_{i \in \mathbb{Z}^d}$. Observe that the minimal stationary solution $\{x_{i;\infty}(s)\}_{i \in \mathbb{Z}^d}$ are translation invariant for every $s \in \mathbb{R}$. This follows, as for every finite $-t \leq s$, $\{x_{i;t}(s)\}_{i \in \mathbb{Z}^d}$ is translation invariant, as the initial conditions (of all queues being empty) and the driving sequences in the finite time interval $[-t,s]$ are both translation invariant. Thus, the almost-sure limit $\{x_{i;\infty}(s)\}_{i \in \mathbb{Z}^d}$ is also translation invariant.} The following Proposition sheds light on the question of unique translation invariant stationary solutions.

\begin{proposition}
	If $\mathbb{E}[x_{0;\infty}(0)^2] < \infty$, then $\{x_{i;\infty}(0)\}_{i \in \mathbb{Z}^d}$ is the unique {\color{black} translation invariant } stationary solution with finite second moment.
	\label{prop:uniq_bdd_convergence}
\end{proposition}

This proposition is proved in Section \ref{sec:proof_bounded_converge}. This result relies on the finiteness of second moment of the stationary queue length, which does not follow immediately from the conclusions of Theorem \ref{thm:main_stability}. In this regard, we have the following proposition, that establishes finiteness of second moment under further restrictive conditions than stability.

\begin{proposition}
	If $\lambda < \frac{2}{3} \frac{1+c}{\sum_{j \in \mathbb{Z}^d}a_j}$, where  $c = \frac{\sqrt{a_0^2 + a_0 \sum_{j \in \mathbb{Z}^d \setminus \{0\}} a_j } \text{	}- \text{	}a_0}{\sum_{j \in \mathbb{Z}^d \setminus \{0\}} a_j }$, then we have $\mathbb{E}[x_{0;\infty}(0)^2] < \infty$.
	\label{prop:finite_second_moment}
\end{proposition}

The proof of this proposition is carried out in  Section \ref{sec:sids}, with some accompanying calculations in Section \ref{sec:rcl}. Note that under our assumption of $a_0 = 1$, the value of the constant $c$ can be simplified as $c = \frac{\sqrt{ \sum_{j \in \mathbb{Z}^d}a_i} -1}{\sum_{j \in \mathbb{Z}^d}a_i - 1}$. {\color{black}Observe that if $c = \frac{1}{2}$, then the above proposition will cover the full range of stability. However, for any valid interference sequence $\{a_i\}_{i \in \mathbb{Z}^d}$, we have $c \in \left( 0, \frac{1}{2} \right)$, with $c \nearrow \frac{1}{2}$ as $\sum_{j \in \mathbb{Z}^d \setminus \{0\}} a_j \searrow 0$. Thus, this proposition does not cover the full stability region}. For the simplest non-trivial case of one dimensions and the interference sequence being $a_i = 1$ if $|i| \leq 1$ and $a_i = 0$ if $|i| > 1$, the second moment is finite for $\lambda \leq 0.91 \frac{1}{\sum_{j \in \mathbb{Z}}a_j}$. From Propositions \ref{prop:uniq_bdd_convergence} and \ref{prop:finite_second_moment}, we have the following immediate corollary.

\begin{corollary}
		If $\lambda < \frac{2}{3} \frac{1+c}{\sum_{j \in \mathbb{Z}^d}a_j}$, where $c$ is given in Proposition \ref{prop:finite_second_moment}, then $\{x_{i;\infty}(0)\}_{i \in \mathbb{Z}^d}$ is the unique {\color{black} translation invariant } stationary solution with finite second moment.
\end{corollary}

Our next set of results assesses whether queue length process converges to any stationary solution when started from different starting states. Observe that we deemed the system stable if when started with all queues empty,
the queue lengths converge to a proper random variable. Thus, stability alone does not imply convergence from other initial conditions.  In this regard, our main results are stated in Theorems \ref{thm:initial_cond_bdd_converge} and \ref{thm:bad_initial_conditions} which show the sensitivity of the dynamics to the starting conditions. In particular, we show in Theorem \ref{thm:initial_cond_bdd_converge}, that if $\lambda$ is sufficiently small and the initial conditions are uniformly bounded, then the queue lengths converge to the minimal stationary solution. Surprisingly, in Theorem \ref{thm:bad_initial_conditions} below, we exhibit both deterministic and random initial conditions \emph{for all} $\lambda >0$, such that the queue lengths diverge, even though the stability criterion $\lambda < \frac{1}{\sum_{j \in \mathbb{Z}^d}a_j}$ is met. This is a new type of result which holds primarily since the network consists of an infinite collection of queues.

\begin{theorem}
	Let $\lambda$ be such that the minimal stationary solution satisfies $\mathbb{E}[x_{0;\infty}(0)^2] < \infty$. Then if the initial condition satisfies $\sup_{i \in \mathbb{Z}^d} x_i(0) < \infty$, the queue length process $\{x_i(\cdot)\}_{i \in \mathbb{Z}^d}$ converges weakly to the minimal stationary solution as $t \rightarrow \infty$.
	\label{thm:initial_cond_bdd_converge}
\end{theorem}

This theorem is proved in Section \ref{sec:proof_thm_initial_bdd}. As the queue lengths are positive integer valued, and the dynamics admits a form of monotonicity, every fixed finite collection of coordinates  also converges to the minimal stationary solution in the total variation norm in the above Theorem, which is stronger than just weak convergence. Notice from Proposition \ref{prop:finite_second_moment}, that if $\lambda < \frac{2}{3} \frac{1+c}{\sum_{j \in \mathbb{Z}^d}a_j}$, where $c$ is given in Proposition \ref{prop:finite_second_moment}, then the conclusion of the above Proposition holds. 
\\


We further examine sensitivity to initial conditions in Theorem \ref{thm:bad_initial_conditions} by constructing examples where the queue lengths diverge, even though the stability criterion is met. To state the result, we need a natural `irreducibility' condition on the interference sequence $\{a_i\}_{i \in \mathbb{Z}^d}$.

\begin{definition}
	The interference sequence $\{a_i\}_{i \in \mathbb{Z}^d}$ is said to be \emph{irreducible} if, for all $z \in \mathbb{Z}^d$,
	there exists $k \in \mathbb{N}$ and $i_1,\cdots,i_k \in \mathbb{Z}^d$, not necessarily distinct,
	such that $i_1 + i_2 \cdots +i_k = z$ and $a_{i_j} > 0$ for all $j \in [1,k]$.
\end{definition}
This is a natural condition which ensures that we cannot `decompose' the grid into many sets of queues, each of which does not interact with the queues in the other group. In the extreme case, this disallows the case when $a_i = 0$ for all $i \neq 0$, in which case the network can be decomposed into an infinite collection of independent $M/M/1$ Processor Sharing queues.

\begin{theorem}
	For all $\lambda >0$, $d \in \mathbb{N}$, and irreducible interference sequences $\{a_i\}_{i \in \mathbb{Z}^d}$, and even when the stability criterion holds, there exists
	\begin{enumerate}
		\item A deterministic sequence $(\alpha_i)_{i \in \mathbb{N}}$ such that if the initial condition satisfied $x_i(0) \geq \alpha_i$ for all $i \in \mathbb{Z}^d$, then the queue length of $0$ satisfies $\lim_{t \rightarrow \infty}x_0(t) = \infty$ almost surely.
		\item A distribution $\xi$ on $\mathbb{N}$ such that if the initial condition $\{x_i(0)\}_{i \in \mathbb{Z}^d}$ is an i.i.d. sequence with each $x_i(0)$, $i \in \mathbb{Z}^d$ being distributed as $\xi$ independent of everything else, then the queue length of $0$ (or any finite collection of queues) satisfies $\lim_{t \rightarrow \infty}x_i(t) = \infty$ almost surely.
	\end{enumerate}
\label{thm:bad_initial_conditions}
\end{theorem}

This theorem is proved in Section \ref{sec:examples_diverge}. Based on the proof of this theorem, we make the following remark.

\begin{remark}
	For all $\lambda$, the support of $(\alpha_i)_{i \in \mathbb{Z}^d}$ in statement $1$ above can be made arbitrarily sparse, i.e. for any sequence $(b_n)_{n \in \mathbb{N}}$ such that $b_n \rightarrow \infty$, the initial conditions $(\alpha_i)_{i \in \mathbb{Z}^d}$ can be chosen, such that  $\lim_{n \rightarrow \infty} \frac{\sum_{i \in \mathbb{Z}^d: ||i||_{\infty} \leq n}  \mathbf{1}_{\alpha_i > 0}}{b_n} = 0$, yet the queue lengths converge almost surely to infinity. 
\end{remark}

The above theorem is qualitative in nature, as it only establishes the existence of bad initial conditions, but does not provide estimates for how large this initial condition must be. In this regard, we include Proposition \ref{prop:quant_diver_full}, which pertains to the deterministic starting state in the simplest non-trivial system, namely the case of $d=1$, and the interference sequence being $(a_i)_{i \in \mathbb{Z}}$ such that $a_i = 1$ for $|i| \leq 1$ and $a_i = 0$ otherwise. This simplest non-trivial example already contains the key ideas and hence we present the computations involved explicitly here. In principle, one can provide a quantitative version of the above theorem in full generality. However, we do not pursue this here as they involve heavy calculations without additional insight into the system.

\begin{proposition}
Consider the  system with $d=1$ and the interference sequence $a_i = 1$ if $|i| \leq 1$ and $a_i = 0$ otherwise. Let $(b_n)_{n \in \mathbb{N}}$ be arbitrary deterministic non-negative integer valued sequence such that  $b_n \rightarrow \infty$. If the initial condition, $\alpha_i := i^{i2^{i+2}+8}$ for $i \in \{b_n: n\in \mathbb{N}\}$, and $\alpha_i = 0$ otherwise, then  for every $\lambda >0$, $\lim_{t \rightarrow \infty}x_0(t) = \infty$ almost surely.
\label{prop:quant_diver_full}
\end{proposition}

This proposition is proved in Appendix $D$. Regarding the converse to stability, we prove the following result in Theorem \ref{thm:transience}, which establishes that the phase-transition at the critical $\lambda$ is sharp, at least in certain cases, and we conjecture it to be sharp for all cases. In order to state the result about transience,  we require the following definition about the monotonicity of the interference sequence.
\begin{definition}
	The interference sequence $(a_i)_{i \in \mathbb{Z}}$ for the dynamics on the one dimensional grid is said to be \emph{monotone} if for all $ i \in \mathbb{Z}_{+}$, $a_i \geq a_{i+1}$ holds true.
	\label{def:monotone_ai}
\end{definition}
The following theorem is the main result regarding instability.
\begin{theorem}
	For the system with $d=1$ and monotone interference sequence, if $\lambda > \frac{1}{\sum_{j \in \mathbb{Z}} a_i}$, then the system is unstable.
	\label{thm:transience}
\end{theorem}

This theorem is proved in Section \ref{sec:transience_proof}. We provide a more quantitative version of this result in Theorem \ref{thm:transience_quant} stated in Section \ref{sec:transience_proof}, which is applied to large finite spatial truncation of the dynamics.

\subsection{Open Questions and Conjectures}
\label{sec:open_questions}

We now list some conjectures and questions that are left open by the present paper. The first one concerns the moments of the minimal stationary solution $x_{0;\infty}(0)$. We have established an exact formula for the mean that holds in the entire stability region and finiteness of the second moment in a fraction of the stability region. In an earlier version of this paper that we posted online, we had put forth the following conjecture.
\begin{conjecture}
	If $\lambda < \frac{1}{\sum_{j \in \mathbb{Z}^d}a_j}$, then $\mathbb{E}[x_{0;\infty}(0)^2] < \infty$.\label{conjecture_2nd_mom}
\end{conjecture}

Subsequently, this has been proven to be correct by \cite{stolyar_conjecture} using rate-conservation techniques, similar to those presented in the present paper. This fact, along with Proposition \ref{prop:uniq_bdd_convergence} implies, that the minimal stationary solution is indeed the unique {\color{black}translation invariant} stationary solution to the dynamics that admits finite second moments. Furthermore, the conclusion of Theorem \ref{thm:initial_cond_bdd_converge} also hold for all $\lambda < \frac{1}{\sum_{j \in \mathbb{Z}^d}a_j}$, namely from all bounded initial conditions, the queue length process will converge to this unique {\color{black} translation invariant} stationary solution if $\lambda < \frac{1}{\sum_{j \in \mathbb{Z}^d}a_j}$. In this regard, three natural interesting questions arise - one concerning what other moments of stationary queue lengths are finite, one regarding correlation decay and another on existence of other stationary solutions.

\begin{question}
	For each $\lambda \in \left( 0,\frac{1}{\sum_{j \in \mathbb{Z}^d}a_j} \right)$, what moments of $x_{0;\infty}(0)$ are finite ?
\end{question}

\begin{question}
	How does the correlation $k \rightarrow \mathbb{E}[x_{0;
	\infty}(0) x_{k;\infty}(0)] - \mathbb{E}[x_{0;\infty}(0^2)]$ decay as $|k| \rightarrow \infty$ ?
\end{question}


\begin{question}
	Does the dynamics admit stationary solutions other than the minimal one ? If so, do there exist initial conditions such that the law of the queue lengths converge to them ?
\end{question}

{\color{black}We know from Proposition \ref{prop:uniq_bdd_convergence} that the minimal  stationary solution is the unique translation invariant stationary solution with  finite second moment. This then raises the following question.
	
\begin{question}
	Does there exist a translation invariant stationary solution that has an infinite first moment ? Does there exist one with finite first moment, but infinite second moment ? 
\end{question}
}
 In regard to establishing transience, a natural open question in light of Theorem \ref{thm:transience} is to extend this result to higher dimensions and non monotone interference sequence. We make the following conjecture.

\begin{conjecture}
	For all $d \geq 1$ and interference sequence $(a_i)_{i \in \mathbb{Z}}$, if $\lambda > \frac{1}{\sum_{j \in \mathbb{Z}^d} a_j}$, then the system is unstable.
\end{conjecture}

\subsection{Main Ideas in the Analysis}
 
 The key technical challenge in analyzing our model is the positive correlation between queue
lengths, which persist even in the model with infinitely many queues 
(see also Figure \ref{fig:queue_corr}). As mentioned, our system of queues is neither reversible,
nor falls under the category of generalized Jackson networks. Thus, our model does not admit a product
form stationary distribution, even when there are finitely many queues.  
In particular, the model has no asymptotic independence properties as those encountered in 
``mean-field models" (such as the supermarket model \cite{dobrushin}).
The correlations across queues is intuitive, since if a queue has a large number of customers,
then its neighboring queues will receive lower rates, and thus they will in turn build up.
Therefore in steady state, if a particular queue is large, most likely, its neighboring queues 
are also large (see also Figure \ref{fig:queue_corr}). 
 \\
 
 To prove the sufficient condition for stability,
we first study finite space-truncated torus systems in Section \ref{sec:torus_system}. In words, we restrict
the dynamics to a large finite set $B_n \subset \mathbb{Z}^d$, and study its stability by employing fluid-like
and Lyapunov arguments. For this model, we write down rate conservation equations in Section \ref{sec:rcl} and solve for the mean
queue-length of this dynamics. This section contains the key technical innovations in this paper. The rate conservation equations turn out to be surprisingly fruitful, as we are able to obtain an
exact formula for the mean queue length. This formula also gives as a corollary, that the queue length 
distributions are tight, as the size of the truncation $B_n$ increases to $\mathbb{Z}^d$.
In Section \ref{sec:sids}, we then show that we can take a limit as $B_n$ increases to all of $\mathbb{Z}^d$
and consider the stationary solution $\{x_{i;\infty}(0)\}_{i \in \mathbb{Z}^d}$ as an appropriate 
limit of the stationary solutions of the space-truncated system. The central argument in this section
is to exploit the many symmetries, the monotonicity of the dynamics and the aforementioned tightness 
to arrive at the desired conclusion. We furthermore, apply a similar rate conservation equation
for the infinite system, which along with monotonicity
arguments, establishes the uniqueness of stationary solutions with finite second moments.
\\
 
To study the convergence from different initial conditions, we employ different arguments,
again exploiting the symmetry and monotonicity in the model. To show that stability implies
convergence from bounded initial conditions, we define a modified $K$-shifted system
in Section \ref{sec:model_estensions_K_shifted}. It is a model having the same dynamics as our
original model, except that the queue lengths do not go below $K$, for some $K \in \mathbb{N}$. 
We carry out the same program of identifying a bound on the first moment on the minimal
stationary solution to the shifted dynamics by analyzing similar rate conservation equations
as for the original system. We then exploit the monotonicity and the fact that a stationary
solution with finite mean is unique, to conclude that stability implies convergence to the minimal stationary
solution from bounded initial conditions. In order to identify initial conditions from where
the queue length can diverge even though the stability condition holds, we first consider a simple 
idea of `freezing' a boundary of queues at a large distance $n$, to a `large value' $\alpha_n$ around
a typical queue, say $0$, and then consider its effect on the queue length at the origin.
By freezing, we mean, there are no arrivals and departures in those queues, but a constant number
$\alpha_n$ of customers that cause interference. We see that by choosing $\alpha_n$ sufficiently large,
this \emph{wall} can influence the stationary distribution at queue $0$. We leverage this observation, 
along with monotonicity, to construct both deterministic and random translation invariant
initial conditions such that queue lengths diverge to $+\infty$ even though the stability condition holds. This proof technique is inspired by similar ideas developed to establish non-uniqueness of Gibbs measures in the case when the state space of a particle is finite, while our methods and results bear on the case when the state space is countable.

\subsection{Organization of the Paper}

The rest of the paper is organized as follows. In 
Section \ref{sec:related_work}, we survey related work on infinite queueing dynamics and place our model in context.
We then start the technical part of the paper by providing the complete mathematical framework in Section \ref{sec:math_framework},
where we formalize the model and the questions studied. 
We also state the monotonicity properties satisfied by the model, which are crucial throughout.
We discuss certain generalizations of the model in Section \ref{sec:model_extensions}. We subsequently proceed to state and prove the main results in this paper.
In Section \ref{sec:torus_system}, we introduce
the space truncated finite system version of our model and analyze it using fluid-like arguments. The space truncated system can be viewed as a certain  finite dimensional approximation of our infinite dimensional dynamics.
The key technical part in that section is in writing and analyzing certain
rate-conservation equations in Section \ref{sec:rcl}, which give an explicit formula 
for the mean queue length in steady state. Based on the results in this section, we complete the proof of Theorem \ref{thm:main_stability} in Section \ref{sec:sids}, where we establish that the minimal stationary solution of our dynamics is a limit of the stationary solutions of the finite approximations in an appropriate sense. Subsequently in Section \ref{sec:proof_bounded_converge}, we prove Proposition \ref{prop:uniq_bdd_convergence}. In Section \ref{sec:proof_bounded_converge}, we prove Theorem \ref{thm:initial_cond_bdd_converge}. The proof of Theorem \ref{thm:bad_initial_conditions} which establishes the presence of bad initial conditions is then done in Section \ref{sec:examples_diverge}. The proof of Theorem \ref{thm:transience} establishing the converse to stability is carried out in Section \ref{sec:transience_proof}. For ease of exposition, we delegate many details of the proof to the Appendix while
outlining the key ideas in the body of the paper. For instance, the details on construction of the process are forwarded to the Appendix.

\section{Related Work}
\label{sec:related_work}

Our study is motivated by the performance analysis of wireless networks which has a large and rich literature
(see for ex. \cite{bonald_wireless} \cite{baccelli_sg_book}, \cite{srikant_book} and the references therein).
Our model is an adaptation of the \emph{Spatial Birth-Death} model proposed in \cite{sbd_tit}, where a dynamics 
of this type was introduced on a compact subset of the Euclidean space. Although that paper has a phase-transition
result similar to ours for stability, the analysis sheds no light on whether the result holds true
for an infinite network. In this paper, we answer in the affirmative in Theorem \ref{thm:main_stability}, that the same result indeed holds in the infinite
discrete network case. From a mathematical point of view, the tools and techniques of \cite{sbd_tit},
which rely on fluid limits, are very different from those discussed in the present paper.
The results are quite different too, with new quantitative results (like the closed
form for the mean queue size) and new qualitative phenomena such as the existence of
multiple stationary solutions being reachable depending on the initial conditions.
\\

Since some of the new properties are directly linked to the fact that there are infinitely
many queues, we thought it appropriate to briefly survey the mathematical literature on queueing models
consisting of infinitely many queues interacting through some translation invariant dynamics. 
A model related to ours is the so called \emph{Poisson Hail} model which has been 
studied in a series of papers \cite{poisson_hail1},\cite{poisson_hail2},\cite{poisson_hail3}. 
The discrete version of this model consists of a collection of queues on $\mathbb{Z}^d$, where
the queues interact through their service mechanism in a translation invariant manner. In this model,
the customer at a queue occupies a `footprint' and when being served, no other customer in
the queues belonging to its footprint is served. In contrast, in our model a customer slows down
the customers in neighboring queues, but does not block them. Another set of papers close to ours is 
\cite{baccelli_infinite_q}, \cite{martin_queuing}, and \cite{Mairesse-Prabakhar}.
These papers analyze an infinite collection of queues in series. The main results are connections
with last passage percolation on grids. A similar model to this is studied by \cite{ferrari}, where 
analogues of Burke's theorem are established for a network of infinite collection of queues on the integers. 
There is also a series of papers on infinite polling systems.
The paper \cite{foss_polling} considers a polling model with an infinite collection of stations,
and addresses questions about ergodicity and positive recurrence of such models.
In a similar spirit, \cite{borovkov} considers infinite polling models and establishes
the presence of many stationary solutions leveraging the fact that the Markov process is not finite-dimensional.
The dynamics in these polling systems are however very different from ours.
The paper of \cite{malyshev} also introduced a nice problem with translation invariant dynamics,
but only analyzed the setting with finitely many queues. The paper of \cite{jackson_graph} introduced an elegant problem on Jackson queueing networks on infinite graphs. However, the stationary distribution there admits a product-form representation, which is very different from our model in the present paper. The paper \cite{hajek_balanced} 
studies translation-invariant dynamics on infinite graphs arising from combinatorial optimization,
which again falls broadly in the same theme, but for a fundamentally different class of problems. Queueing like dynamics on an infinite number of nodes are also studied, though under different names, in the
interacting particle system literature in the sense of \cite{liggett}. The most well known instance of interacting particle system
connected to queueing is probably the TASEP. Another fundamental class of interacting particle system exhibiting a
positive correlation between nodes (like our model) is the ferromagnetic Ising model. The first difference
is that the state-space of a node is not compact (i.e., $\mathbb{N}$, since the
state is the number of customers in the queue) in our model, whereas it is finite in these models.
Another fundamental difference between our model and these is the lack of reversibility.
The common aspects are the infinite dimensional Markovian representation of the dynamics, the 
non uniqueness of stationary solutions, and the sensitivity to initial conditions.
Infinite queueing models are also central in mean-field limits.
In the literature on mean-field queueing systems (\cite{dobrushin,Graham,Shlosman})
the finite case exhibits correlations among the queue lengths thereby making them difficult to analyze. 
However, in the large number of node limit, one typically shows that there is  `propagation of chaos'.
This then gives that the queue lengths become independent in the limit. This independence can then be leveraged
to write evolution equations for the limiting dynamics which can be analyzed. Such mean-field analysis have recently become very popular in the applied literature (for ex. \cite{srikant_ying},\cite{ramanan_hydrodynamic}).
Our model differs fundamentally from the above models in many aspects. First, unlike the mean-field models described above,
we can directly define the limiting infinite object, i.e., a model with infinitely many queues.
Secondly and more crucially, our infinite model does {\em not} exhibit any independence properties in the limit,
i.e., queue lengths are positively correlated even in the infinite model.
This is why we need different techniques to study this model.
Our main technical achievement in this context is to introduce coupling and
rate conservation techniques not relying on any independence properties.

\section{Problem Setup}
\label{sec:math_framework}
In this section, we give a precise description of our model in subsection \ref{subsec:model_framework} and demonstrate certain useful monotonicity properties it satisfies in Subsection \ref{sec:mono}. We then precisely state the definition of stability in Section \ref{sec:stable} and the notion of stationary solutions to the dynamics in Section \ref{subsec:stationary_solutions}. 

\subsection{Framework}
\label{subsec:model_framework}

 Our model is parametrized by $\lambda \in \mathbb{R}$ and an \emph{interference sequence} $\{a_i\}_{i \in \mathbb{Z}^d}$ which is a non-negative sequence. This sequence satisfies $a_0 = 1$, $a_i = a_{-i}$ for all $i \in \mathbb{Z}^d$ and $L:= \sup \{||i||_{\infty}: a_i > 0\} < \infty$, i.e., is finitely supported. We also impose the sequence $\{a_i\}_{i \in \mathbb{Z}^d}$ to be \emph{irreducible}, which gives that for all $z \in \mathbb{Z}^d$, there exists $k \in \mathbb{N}$ and $i_1,\cdots,i_k \in \mathbb{Z}^d$ not necessarily distinct, such that $i_1+i_2\cdots+i_k = z$ and $a_{i_j} > 0$ for all $j \in \{1,\cdots,k\}$. To describe the probabilistic setup, we assume there exists a probability space $(\Omega, \mathcal{F},\mathbb{P})$ that
contains the stationary and ergodic driving sequences $(\mathcal{A}_i,\mathcal{D}_i)_{\i \in \mathbb{Z}^d}$.
For each $i \in \mathbb{Z}^d$, $\mathcal{A}_i$ is a Poisson Point Process (PPP) of intensity
$\lambda$ on $\mathbb{R}$, independent of everything else and $\mathcal{D}_i$ is a PPP of intensity $1$ on $\mathbb{R} \times [0,1]$, independent of everything else. 
Our stochastic process denoting the queue lengths $t \rightarrow \{x_i(t)\}_{i \in \mathbb{Z}^d}$
will be constructed as a factor of the process $(\mathcal{A}_i,\mathcal{D}_i)_{\i \in \mathbb{Z}^d}$. 
The process $\mathcal{A}_i := \sum_{q \in \mathbb{Z}} \delta_{A_{q}^{(i)}}$ encodes the fact that,
at times $\{A_{q}^{(i)}\}_{q \in \mathbb{Z}}$, there is an arrival of a customer in queue $i$. Thus the arrivals to queues form  PPPs of intensity $\lambda$ and are independent of everything else.
The process $\mathcal{D}_i := \sum_{q \in \mathbb{Z}}\delta_{(D_{q}^{(i)}, U_{q}^{(i)})}$ encodes 
 that there is a possible departure from queue $i$ at time $D_{q}^{(i)}$, with an additional independent $U[0,1]$
 random variable provided by $U_{q}^{(i)}$. To precisely describe the departures, we define  \emph{the interference at a customer in queue $i$ at time $t$} as equal to $\sum_{j \in \mathbb{Z}^d} a_{j}x_{i-j}(t)$.  A customer, if any, is removed from  queue $i$ at times  $D_{q}^{(i)}$ if and only if $U_{q}^{(i)} \leq \frac{x_i(D_{q}^{(i)})}{\sum_{j \in \mathbb{Z}^d} a_{j-i}x_j(D_{q}^{(i)})  }$.
In other words, conditionally on the state of the network $\{x_j(D_{q}^{(i)})\}_{j \in \mathbb{Z}^d}$ at time $D_q^{(i)}$, we remove a customer from queue $i$ at time $D_{q}^{(i)}$ with probability
$\frac{x_i(D_{q}^{(i)})}{\sum_{j \in \mathbb{Z}^d} a_{j-i}x_j(D_{q}^{(i)})}$, independently of everything else.
Thus we see that conditionally on the network state $\{x_j(t)\}_{j \in \mathbb{Z}^d}$ at time $t$, the instantaneous rate of departure from any queue $i \in \mathbb{Z}^d$ at time $t \in \mathbb{R}$ is $\frac{x_i(t)}{\sum_{j \in \mathbb{Z}^d}a_jx_{i-j}(t)}$, independently of everything else. Observe that since $a_0 = 1$, if $x_i(t) > 0$, then necessarily, $\frac{x_i(t)}{\sum_{j \in \mathbb{Z}^d}a_jx_{i-j}(t)} \in (0,1]$. 
\\

We further assume (without loss of generality) that the probability space $(\Omega,\mathcal{F},\mathbb{P})$ equipped with a group 
$(\theta_u)_{u \in \mathbb{R}}$ of measure preserving functions from $\Omega$ to itself where $\theta_u$ denotes
the `time shift operator' by  $u \in \mathbb{R}$. More precisely
$(\mathcal{A}_i,\mathcal{D}_i)_{\i \in \mathbb{Z}^d} \circ \theta_u$ is the same driving sequence
where each of the arrivals and departures are shifted by time $u$ in all queues, i.e., if 
$\mathcal{A}_i := \sum_{q \in \mathbb{Z}} \delta_{A_{q}^{(i)}}$ and
$\mathcal{D}_i := \sum_{q \in \mathbb{Z}}\delta_{(D_{q}^{(i)}, U_{q}^{(i)})}$, then
$\mathcal{A}_i \circ \theta_u := \sum_{q \in \mathbb{Z}} \delta_{A_{q}^{(i)} - u}$ and 
$\mathcal{D}_i \circ \theta_u := \sum_{q \in \mathbb{Z}}\delta_{(D_{q}^{(i)}-u, U_{q}^{(i)})}$, for all
$i \in \mathbb{Z}^d$. We also assume that the system $(\mathbb{P},(\theta_u)_{u \in \mathbb{R}})$
is ergodic, i.e. if for some event $A \in \mathcal{F}$, if
$\mathbb{P}[A \bigtriangleup A\circ \theta_u ] = 0$ for all $u \in \mathbb{R}$,
then  $\mathbb{P}[A] \in \{0,1\}$.

%

\subsection{Construction of the Process}

Before  analyze the above model, one needs to ensure that it is `well-defined'. We mean that our model is well defined if given the
initial network state $\{x_i(0)\}_{i \in \mathbb{Z}^d}$, any time $T \geq 0$ and any index $k \in \mathbb{Z}^d$, we are able to construct the queue length $x_k(T)$ unambiguously and exactly. In the case of finite networks (i.e., networks with finitely many queues),
the construction is trivial: almost surely, one can order all possible  events in the network with increasing time,
and then update the network state sequentially using the evolution dynamics described above. Such a scheme works unambiguously since, almost surely, all event times
will be distinct and in any interval $[0,T]$, there will be finitely many events. 
The main difficulty in the case of infinite networks is that there is no \emph{first-event} in the network.
In other words, in any arbitrarily small interval of time, infinitely many events will occur almost surely
and hence we cannot construct by ordering all the events in the network. However we show  in Appendix \ref{appendix_construction} that in order to determine the value of any  arbitrary queue $k \in \mathbb{Z}^d$ at any time $T \geq 0$, we can effectively restrict our
attention to an almost surely finite subset $X_{k,T} \subset \mathbb{Z}^d$ and determine $x_k(T)$
by restricting the dynamics to $X_{k,T}$ to the interval $[0,T]$.
This is then easy to construct as it is a finite system. Thereby we can determine $x_k(T)$ unambiguously. Such construction procedures are common in Interacting Particle systems setup (for example, the book of \cite{liggett}). Nevertheless, we present the entire details of construction in Appendix  \ref{appendix_construction} for completeness. 

\subsection{Monotonicity}

\label{sec:mono}
We establish an obvious but an extremely useful property of path-wise monotonicity satisfied by the dynamics. Note that our model is not monotone separable in the sense of \cite{monotone_sep}
since the dynamics does not satisfy the external monotonicity condition. Nonetheless, the model still enjoys certain restricted forms of monotonicity, which we state below. We only highlight the key idea for the proof and defer the details to  Appendix \ref{appendix_monotonicity}.

\begin{lemma}
	If we have two initial conditions $\{x^{'}_{i}(0)\}_{i \in \mathbb{Z}^d}$ and
	$\{x_{i}(0)\}_{i \in \mathbb{Z}^d}$ such that for all $i \in \mathbb{Z}^d$, $x^{'}_{i}(0) \geq x_i(0)$,
	then there exists a coupling such that   $x^{'}_{i}(T) \geq x_i(T)$ for all $i \in \mathbb{Z}^d$ and all $T \geq 0$ almost surely.
	\label{lem:mono1}
\end{lemma}

The proof is by a path-wise coupling argument, where the two different initial conditions are driven by the same arrival and potential departures. The key idea  the following. At arrival times, the ordering will trivially be maintained. Consider some queue $i$ and time $t$ where there is a potential departure. If $x_{i}^{'}(t) \geq x_i(t) + 1$, then, since at most one departure occurs, the ordering will be maintained. But if $x_{i}^{'}(t) = x_i(t)$, then the rates $\frac{x_{i}^{'}(t)}{\sum_{j \in \mathbb{Z}^d}a_j x_{i-j}^{'}(t)} \leq \frac{x_{i}(t)}{\sum_{j \in \mathbb{Z}^d}a_j x_{i-j}(t)}$ and hence the ordering will again be maintained. This observation can be leveraged again to have the following  form of monotonicity.

\begin{lemma}
	For all initial conditions $\{x_i(0)\}_{i \in \mathbb{Z}^d}$,
	for all $0 \leq s \leq t \leq \infty$, all $X \subset \mathbb{Z}^d$, and all $T >0$,
	$\{x_i(T)\}_{i \in \mathbb{Z}^d}$ is coordinate-wise larger in the true dynamics than in the dynamics
	constructed by setting $\mathcal{A}_j([s,t])=0$ for all $j \in X$.
	\label{lem:mono2}
\end{lemma}

\subsection{Stochastic Stability}
\label{sec:stable}

We establish a $0-1$ law stating that either all queues are transient or all queues are recurrent (made precise in Lemma \ref{lem:joint_01} in the sequel). Thus, we can then claim that the entire network is stable if and only if any (say queue indexed $0$ without loss of generality) is stable (made precise in Definition \ref{defn:stable} in the sequel). To state the lemmas, we set some notation. Let $T \geq 0$ and $s > -T$ be arbitrary and finite. Denote by
$\{x_{i;T}(s)\}_{i \in \mathbb{Z}^d}$ the value of the process seen at time $s$ when started with the empty initial
state at time $-T$, i.e., with the initial condition of $x_{i;T}(-T) = 0$ for all $i \in \mathbb{Z}^d$.
Lemma \ref{lem:mono1} implies that for every queue $i \in \mathbb{Z}^d$, and for $\mathbb{P}$ almost-every $\omega \in \Omega$,
we have $T\to x_{i;T}(s)$ is non-decreasing for every fixed $s$. Thus, for every $i$, and every $s \in \mathbb{R}$,
there exists an almost sure limit  $\lim_{T \rightarrow \infty}x_{i:T}(s) := x_{i;\infty}(s)$. From the definition, this limit is shift-invariant, i.e., almost surely, for all $x \in \mathbb{R}$, we have $x_{i;\infty}(s) \circ \theta_x = x_{i;\infty}(s+x)$.

\begin{lemma} \label{lem:joint_01}
	We have either $\mathbb{P}[\cap_{i \in \mathbb{Z}^d} \{x_{i;\infty}(0) = \infty \}] = 1$
	or $\mathbb{P}[\cap_{i \in \mathbb{Z}^d} \{ x_{i;\infty}(0) < \infty \} ] = 1$.
	\label{lem_01_law}
\end{lemma}
The proof follows from standard shift-invariance arguments which we present here for completeness. Since for all $x \in \mathbb{R}$ and all $j \in \mathbb{Z}^d$, $x_{i;\infty}(0)\circ \theta_x  = x_{i;\infty}(x)$, we have that this lemma implies for all $s \in \mathbb{R}$, either $\mathbb{P}[\cap_{i \in \mathbb{Z}^d} \{x_{i;\infty}(s) = \infty \}] = 1$ or $\mathbb{P}[\cap_{i \in \mathbb{Z}^d} \{ x_{i;\infty}(s) < \infty \} ] = 1$.
\begin{proof}
	
	It suffices to first show that for any fixed $i \in \mathbb{Z}^d$, we have
	$\mathbb{P}[x_{i;\infty}(0) < \infty] \in \{0,1\}$. Assume that we have established for some $i$ 
	(say $0$ without loss of generality that) $\mathbb{P}[x_{0;\infty}(0) < \infty] \in \{0,1\}$.
	From the translation invariance of the dynamics, it follows that, for all $i \in \mathbb{Z}^d$,
	we have $\mathbb{P}[x_{i;\infty}(0) < \infty] = \mathbb{P}[x_{0;\infty}(0) < \infty]$. Thus,
	if $\mathbb{P}[x_{0;\infty}(0) < \infty] = 1$, then $\mathbb{P}[\cap_{i \in \mathbb{Z}^d}x_{i;\infty}(0) < \infty] =1$.
	Similarly, if $\mathbb{P}[x_{0;\infty}(0) = \infty] = 1$, then $\mathbb{P}[\cap_{i \in \mathbb{Z}^d}x_{i;\infty}(0) = \infty] =1$.
	Thus to prove the lemma, it suffices to prove that  $\mathbb{P}[x_{0;\infty}(0) < \infty] \in \{0,1\}$.
	\\
	
	The key observation is, the event $A := \{ \omega \in \Omega : x_{0;\infty}(0) < \infty \}$ is such that for all $x \in \mathbb{R}$, $\mathbb{P}[A \bigtriangleup A\circ \theta_u] = 0$.
	To show this, first notice that from elementary properties of PPP, we have that for every $i \in \mathbb{Z}^d$ and
	every compact set $B \subset \mathbb{R}$, $\mathcal{A}_i(B) < \infty$ a.s.. Now for any $x \geq 0$, we have 
	$x_{0;\infty}(0) \circ \theta_x \leq x_{0;\infty}(0) + A_{0}([0,x])$, which is finite almost surely if
	$x_{0;\infty}(0) < \infty$ almost surely. Similarly for every $x < 0$, 
	$x_{0;\infty}(0)= x_{0;\infty}(0) \circ \theta_x + A_{0}([x,0])$, which again implies that
	$x_{0;\infty}(0) \circ \theta_x$ is almost surely finite if $x_{0;\infty}(0) < \infty$.
	Thus, for all $x \in \mathbb{R}$, we have $\mathbb{P}[A \bigtriangleup A \circ \theta_x] = 0$, which from ergodicity of $(\mathbb{P},(\theta_u)_{u \in \mathbb{R}})$ implies $\mathbb{P}[A] \in \{0,1\}$ and thus the lemma is proved. 	
\end{proof}

The following definition of stability follows naturally.

\begin{definition} \label{defn:stable}
	The system is \textbf{stable} if $x_{0;\infty}(0) < \infty$ almost surely.
	Conversely, we say the system is \textbf{unstable} if $x_{0;\infty}(0) = \infty$ almost surely.
\end{definition}

Observe that the definition of stability  does not require $\mathbb{E}[x_{0;\infty}(0)]$ to be finite. In words, we say that our model is stable if when starting with all queues being empty at  time $-t$ in the past,  the queue length of any queue stays bounded  at time $0$ when letting $t$ go to infinity. This definition of stability is similar to the definition introduced for example by \cite{loynes} in the single server queue case. A nice account of such backward coupling methods can be found in \cite{baccelli_bremaud}.
\\

The main result in this paper is to prove that if $\lambda \sum_{j \in \mathbb{Z}^d} a_j < 1$, then the system is stable (Theorem \ref{thm:main_stability}). Moreover, in this case, we compute exactly the mean queue length in steady state, i.e., an explicit formula for $\mathbb{E}[x_{0;\infty}(0)]$ (Theorem \ref{thm:main_stability}) and by shift-invariance it is equal to $\mathbb{E}[x_{i;\infty}(s)]$. We also conjecture this condition to be necessary, i.e., if $\lambda \sum_{j \in \mathbb{Z}^d}a_j > 1$, then $x_{0;\infty} = \infty$ almost surely. We are unable to prove this conjecture  yet, but prove it for the special case of $d=1$ in Theorem \ref{thm:transience}.

\subsection{{\color{black} Translation Invariant} Stationary Solutions}
\label{subsec:stationary_solutions}


\begin{definition}
	A probability measure $\boldsymbol{\pi}$ on $(\mathbb{Z}^d)^{\mathbb{N}}$ is said to be
	{\color{black} translation invariant, if $(y_i)_{i \in \mathbb{Z}^d} \sim \boldsymbol{\pi}$ implies, for all $x \in \mathbb{Z}^d$, $(y_{i-x})_{i \in \mathbb{Z}^d} \sim \boldsymbol{\pi}$}. A probability measure $\boldsymbol{\pi}$ on $(\mathbb{Z}^d)^{\mathbb{N}}$ is said to be stationary
	for the dynamics $\{x_i(t)\}_{i \in \mathbb{Z}^d}$ if, whenever $\{x_{i}(0)\}_{i \in \mathbb{Z}^d}$ is
	distributed according to $\boldsymbol{\pi}$ independently of everything else, then, for all $t \geq 0$,
	the random variables $\{x_i(t)\}_{i \in \mathbb{Z}^d}$ are also distributed as $\boldsymbol{\pi}$.
\end{definition}

	{\color{black} In this paper, we restrict ourselves to studying stationary solutions to the dynamics that are translation invariant in space. Observe that the driving sequence $(\mathcal{A}_i,\mathcal{D}_i)_{i \in \mathbb{Z}^d}$ is translation invariant on $\mathbb{Z}^d$, i.e., for all $v \in \mathbb{Z}^d$,  $(\mathcal{A}_{i-v},\mathcal{D}_{i-v})_{i \in \mathbb{Z}^d}$ is equal in distribution to $(\mathcal{A}_i,\mathcal{D}_i)_{i \in \mathbb{Z}^d}$. Furthermore, the interactions among the queues are also translation invariant, since the definition of interference seen at a queue is translation invariant. However, it is not immediately clear that all stationary solutions must necessarily be translation invariant. It is known, for instance in the literature on Ising Models (see the book \cite{georgii_gibbs}), that certain stationary measures for translation invariant Glauber dynamics need not necessarily be translation invariant. We leave the question of existence and construction of non-translation invariant stationary measures for our model to future work.}
	\\


Moreover, as our network is not  finite-dimensional, stability in the sense of Definition \ref{defn:stable} does not imply ergodicity in the usual Markov chain sense. In particular, it does not imply that stationary distributions are unique, and starting from any initial condition on $\mathbb{N}^{\mathbb{Z}^d}$, the queue lengths converge in some sense to the minimal stationary distribution considered in Definition \ref{defn:stable}. Stability only implies the \emph{existence} of  a stationary solution, namely the law of $\{x_{i\infty}(0)\}_{i \in \mathbb{Z}^d}$. However, uniqueness is not granted and one of our main results in Proposition \ref{prop:uniq_bdd_convergence} bears on this. Moreover, convergence to stationary solutions from different starting states is more delicate as evidenced in Theorems \ref{thm:initial_cond_bdd_converge} and \ref{thm:bad_initial_conditions}.

\section{Model Extensions}
\label{sec:model_extensions}

In this section, we introduce two natural extensions to the model not considered in Section \ref{sec:math_framework}. We show that similar results as for our original model hold, albeit with a little bit more notation. Hence we separate this discussion from the main body of the paper with proofs deferred to the Appendix as the  key ideas are the same as for the model described earlier.

\subsection{Infinite Support for the Interference Sequence}
\label{subsec:infinite_interference}

We consider here a system where $\{a_i\}_{i \in \mathbb{Z}^d}$ is such that $a_i \geq 0 $ for all $i \in \mathbb{Z}^d$ with  $\{i: a_i > 0\}$ having infinite cardinality but being summable, i.e., $\sum_{j \in \mathbb{Z}^d}a_j < \infty$. In this case as well, we can uniquely construct the system in a sense as a limit of finite systems with finite truncation. The following proposition encapsulates the main results
\begin{proposition}
	Consider $\{a_i\}_{i \in \mathbb{Z}^d}$ such that $\{i: a_i > 0\}$ has infinite cardinality, and $\sum_{j \in \mathbb{Z}^d}a_j < \infty$. Then the dynamics is well defined. 
	\label{prop:infinite_support_existence}
\end{proposition}
\begin{proof}
	To show the existence of the dynamics, we introduce a sequence of systems, with the $k$ th system evolving according to the dynamics described in Section \ref{sec:math_framework} with the interference sequence being $\{a_i \mathbf{1}_{||i||_{\infty} \leq K}\}$. This interference sequence satisfies all the conditions specified in Section \ref{sec:math_framework} and hence the dynamics can be constructed. We now construct the infinite dynamics sequentially as follows. Consider any arbitrary initial conditions $\{x_i(0)\}_{i \in \mathbb{Z}^d}$. For this system, for every $K \in \mathbb{N}$, we can define the process $\{\tilde{x}^{(K)}_i(t)\}_{i \in \mathbb{Z}^d}$, $t \geq 0$, which is the process corresponding to the truncated interference sequence $\{a_i \mathbf{1}_{||i||_{\infty} \leq K}\}$. Now, it suffices to assert that at each arrival and potential departure event at queue $i$, we can unambiguously decide on how the system with infinite interference support evolves. The evolution due to an arrival event is easy, we just add an customer to queue $i$. At the first potential departure event at queue $i$, with the independent mark given by $u \in [0,1]$, we have to decide whether to remove a customer or not. Now, at this time, we can do this unambiguously by deciding whether $u \leq \lim_{K \rightarrow \infty} \frac{x_{i}(t)}{\sum_{j \in \mathbb{Z}^d} a_j \mathbf{1}(||j||_{\infty} \leq K)\tilde{x}_{i-j}^{(K)}(t)  }$ or not. The existence of the almost sure limit is guaranteed by monotonicity. In words, $a_j^{(K)}$ is non-decreasing in $K$ and the queue lengths $\tilde{x}_i^{(K)}(t)$ are non-decreasing in $K$ for each $i \in \mathbb{Z}^d$ and $t \geq 0$. The numerator $x_{i}(t)$ can be deduced without resorting to limits as this is the first potential departure after time $0$ in queue $i$. Hence we can unambiguously decide on the outcome of the first potential departure event at queue $i$ after time $0$.  Now, by induction, we can construct the sample path of any queue $i$ over any finite time interval, thereby establishing that the dynamics is well defined. 
\end{proof}

Based on the construction described above, it is not immediately clear that a stability region even exists for the case with infinitely supported interference sequence. The following proposition gives an alternative representation of the dynamics as a point-wise limit of dynamics with truncated interference sequence. 

\begin{proposition}
	Consider an initial condition $\{x_i(0)\}_{i \in \mathbb{Z}^d}$ and interference sequence $\{a_i\}_{i \in \mathbb{Z}^d}$ with $\{i: a_i > 0\}$ being infinite and such that $\sum_{j \in \mathbb{Z}^d}a_j < \infty$. Consider the sequence of processes $\{\tilde{x}^{(K)}_i(t)\}_{i \in \mathbb{Z}^d}$ each driven by the $K$-truncated interference sequence dynamics. Then for each $i \in \mathbb{Z}^d$ and $t \geq 0$ finite, we have $x_i(t) = \lim_{K \rightarrow \infty}\tilde{x}_i^{(K)}(t)$ almost surely.
	\label{prop:infinite_point_wise_conv}
\end{proposition}
\begin{proof}
	For every queue $i$ and finite time $t$, there are only finitely many potential departure events almost surely in the interval $[0,t]$. From Proposition \ref{prop:infinite_support_existence}, we know that at each instance of a potential departure at queue $i$, we take a limit in $K$, the truncation length to determine whether or not to remove a customer. However, since there are only finitely many events in the time interval $[0,t]$, one can make the limit uniform to conclude that  for all $t \geq 0$ and all $i \in \mathbb{Z}^d$, $ \lim_{K \rightarrow \infty}\sup_{0 \leq u \leq t}|\tilde{x}_i^{(K)}(u) - x_i(u)| = 0$ almost surely. 
\end{proof}

Based on the construction outlined above, one can extend the existence of a stationary solution to the case when the interference sequence has an infinite support. Indeed, it is not a corollary, as, by the construction of the infinite support dynamics as a point-wise limit of the  $K$ truncated interference systems' dynamics, the existence of a stability region is not granted. 

\begin{proposition}
	Suppose that the interference sequence $\{a_i\}_{i \in \mathbb{Z}^d}$ is such that $\{i: a_i > 0\}$ is infinite with $\sum_{j \in \mathbb{Z}^d}a_j < \infty$. Under this conditions, if $\lambda < \frac{1}{\sum_{j \in \mathbb{Z}^d}a_j}$, then there exists a minimal stationary solution $\{\tilde{x}_i\}_{i \in \mathbb{Z}^d}$ with $\mathbb{E}[\tilde{x}_0] = \frac{\lambda}{1 - \lambda \sum_{j \in \mathbb{Z}^d}a_j}$.
	\label{prop:infinite_seq_stablity}
\end{proposition}

The proof is deferred to Appendix \ref{appendix:infinite_support_proof}. However, establishing uniqueness of stationary regime in this case is slightly more delicate and we leave it to future work. The main difficulty being that writing down rate-conservation equations as done in Section \ref{sec:rcl} when the interference support is infinite is not obvious.

\subsection{$K$-Shifted System}
\label{sec:model_estensions_K_shifted}

In this subsection, we introduce a model of queues which `reflect' at  level $K$. In other words, we consider a dynamic which will forbid any departures from a queue if it has $K$ or more customers at any point of time. Note that the original model we describe is the $0$ shifted, or the model reflected at $0$. Thus, if $\{x_{i}^{(K)}(t)\}_{i \in \mathbb{Z}^d}$ is the stochastic process corresponding to the $K$ shifted dynamics for some $K \in \mathbb{N}$, then the instantaneous rate of departure from any queue $i \in \mathbb{Z}^d$ at time $t$ is then given by 
\begin{align}
\hat{R}^{(K)}_i(t) = \frac{x_i^{(K)}(t)}{\sum_{j \in \mathbb{Z}^d}a_j x_{i-j}^{(K)}(t)} \mathbf{1}(x_i^{(K)}(t) > K).
\label{eqn:K_shifted_rate_func}
\end{align}
For the purposes of this section, we assume that $L := \sup\{||i||_{\infty}: a_i > 0\} < \infty$, although one could extend this definition to include the case of $L = \infty$ as well by the  ideas introduced in Section \ref{subsec:infinite_interference}.  In this case of finitely supported interference sequence, the process can be formally defined  through a Poisson clock similar to that used in Section \ref{sec:math_framework}.
 The main result for the general $K$ shifted system is the following.
\begin{proposition}
	If $\lambda < \frac{1}{\sum_{j \in \mathbb{Z}^d}a_j}$, then for all $K \in \mathbb{N}$, the $K$-shifted dynamics is stable. Moreover, the minimal stationary solution $\{\tilde{x}_i^{(K)}\}_{i \in \mathbb{Z}^d}$ satisfies
	\begin{align}
	\mathbb{E}[\tilde{x}_0^{(K)}] \leq \frac{\lambda+K}{1 - \lambda \sum_{j \in \mathbb{Z}^d}a_j} < \infty.
	\label{eqn:K_shifted_mean}
	\end{align} 
	\label{prop:K_shifted_stability_result}
\end{proposition}

\begin{proposition}
		If $\lambda < \frac{2}{3} \frac{1+c}{\sum_{j \in \mathbb{Z}^d}a_j}$, where the constant  $c = \frac{\sqrt{a_0^2 + a_0 \sum_{j \in \mathbb{Z}^d \setminus \{0\}} a_j } \text{  } - \text{  } a_0}{\sum_{j \in \mathbb{Z}^d \setminus \{0\}} a_j }$, then we have $\mathbb{E}[(\tilde{x}_0^{(K)})^2] < \infty$.
		\label{prop:K_shifted_2ndmom}
\end{proposition}

The proofs are deferred to Appendix \ref{appendix:K_shifted_proof}. The $K$-shifted dynamics is introduced as it will later be used to show convergence from bounded initial conditions to the stationary regime of the original initial dynamics, i.e., it is used as a tool to prove Theorem \ref{thm:initial_cond_bdd_converge} in Section \ref{sec:proof_bounded_converge}. One can also naturally extend the $K$-shifted dynamics to accommodate the case when the interference sequence $\{a_i\}_{i \in \mathbb{Z}^d}$ has infinite support satisfying $\sum_{j \in \mathbb{Z}^d}a_j < \infty$, but we do not do so here.

\section{Space Truncated  Finite Systems}
\label{sec:torus_system}
In this section, we discuss a finite version of the aforementioned infinite queueing network.
For any $n \in \mathbb{Z}_{+}$, we consider two $n$-truncated systems, both of which are obtained by restricting the dynamics to the set $B_n(0)$, the $l_{\infty}$ ball of radius $n$ centered at $0$. For notational convenience, we shall drop $0$ and denote by $B_n := B_n(0)$  the $l_{\infty}$ ball of radius $n$ centered at $0$. For every $n \in \mathbb{N}$, we define two truncated dynamics, $\{y_i^{(n)}(\cdot)\}_{i \in B_n}$ and $\{z_i^{(n)}(\cdot)\}_{i \in B_n}$. The process $\{y_i^{(n)}(\cdot)\}_{i \in B_n}$ evolves with the set $B_n$ `wrapped around' to form a torus. More precisely, the process $\{y_i^{(n)}(\cdot)\}_{i \in B_n}$ is driven by $(\mathcal{A}_i,\mathcal{D}_i)_{i \in B_n}$. The arrival dynamics is the same as for the infinite system described in Section \ref{sec:math_framework} wherein, for all $i \in B_n$, at each epoch of $\mathcal{A}_i$, a customer is added to queue $i$. The departure dynamics is driven by $\mathcal{D}_i$ as before, but we treat the set $B_n$ as a torus. More precisely, given any  
$i,j \in B_n$, define $d_n(i,j) := (i-j) \mod n$, where the modulo operation is coordinate-wise.
Thus, at any time $t$, and any $i \in B_n$, the rate at which a departure occurs from queue $i$
at time $t$ in the process $\{y^{(n)}_i(t)\}_{i \in B_n(0)}$ is 
$\frac{y_{i}^{(n)}(t)}{\sum_{j \in B_n} a_{d_n(i,j)}y^{(n)}_j(t)}$. 
Since $n$ is finite, the stochastic process $\mathbf{y}^{(n)}(t)$ is a continuous time Markov process
on a countable state-space, i.e., on $\mathbb{N}^{(2n+1)^d}$. Moreover, since the jumps are triggered by a finite number of Poisson processes, this chain has almost surely no-explosions.
\\

 Similarly, the process $\{z_i^{(n)}(t)\}_{i \in B_n}$ is driven by the arrival data $(\mathcal{A}_i,\mathcal{D}_i)_{i \in B_n}$ as before, but this time the set $B_n$ is viewed as a subset of $\mathbb{Z}^d$ and in particular the `edge effects' are retained. {\color{black} The arrival rate to any queue $i \in B_n$ in the system $\{z_i^{(n)}(t)\}_{i \in B_n}$  is $\lambda$, while there are no arrivals to queues in $B_n^{\complement}$, i.e., an arrival rate of $0$. Moreover, the queue lengths of queues in $B_n^{\complement}$ is set to $0$, i.e., for all $t \geq 0$ and all $i \in B_n^{\complement}$, we have $z_i^{(n)}(t) = 0$. The departure process for any queue $i \in B_n$, is identical to the original infinite system described in Section \ref{sec:math_framework}. At any time $t \geq 0$, and any $i \in B_n$, the rate of departure from queue $i$ at time $t$ is given by $\frac{z_i^{(n)}(t)}{\sum_{j \in \mathbb{Z}^d} a_{i-j}z_j^{(n)}(t)}$}. From the monotonicity in the dynamics, we have the following proposition.

\begin{proposition}
 For all $n > L$, there exists a coupling of the processes $\{x_i(\cdot)\}_{i \in \mathbb{Z}^d}$, $\{z^{(n)}_i(\cdot)\}_{i \in B_n}$ and $\{y^{(n)}_i(\cdot)\}_{i \in B_n}$ such that for all $t \in \mathbb{R}$, and all $i \in \mathbb{Z}^d$, we have
	$x_{i}(t) \geq z^{(n)}_{i}(t)$ and $y^{(n)}_{i}(t) \geq z^{(n)}_{i}(t)$ almost surely.
	\label{prop:truncation_mono}
\end{proposition}

The following property of the truncated systems will be used in the analysis of the infinite system.

\begin{theorem}
	For all $n > L$ and  $\lambda < \frac{1}{\sum_{j \in \mathbb{Z}^d}a_j}$,
	the Markov process $\{y^{(n)}_i(t)\}_{i \in B_n(0)}$ is positive recurrent. 
	Let $\pi^{(n)}$ denote the stationary queue length distribution on $\mathbb{N}$ of any queue $i \in B_n(0)$ and
	let $Z$ be distributed as $\pi^{(n)}$. Then there exists a $c > 0$ possibly depending on $n$ such that $\mathbb{E}[e^{cZ}] < \infty$. 
	\label{thm:finite_PR}
\end{theorem}

\begin{remark}
	The symmetry in the torus implies that the marginal stationary queue length distribution of any queue
	$i$, $\pi^{(n)}$, is the same for all $i$. 
	\label{remark_rotation_inv}
\end{remark}

\begin{remark}
	The existence of an exponential moment yields that all power moments of $\pi^{(n)}$ are finite.
	\label{remark:finite_moment}
\end{remark}
\begin{remark}
	In view of Proposition \ref{prop:truncation_mono}, if $\lambda < \frac{1}{\sum_{j \in \mathbb{Z}^d}a_j}$, then for all $n \in \mathbb{N}$, the process $\{z_i^{(n)}(\cdot)\}_{i \in B_n}$ is positive recurrent. Moreover, for all $i \in B_n$, the stationary distribution of $\{z_i^{(n)}(\cdot)\}_{i \in B_n}$, denoted by $\{\tilde{\pi}^{(n)}_i\}_{i \in B_n}$, is such that there exists a $c > 0$ possibly depending on $n$ satisfying $\mathbb{E}[e^{cZ_i}] < \infty$, where $Z_i$ is distributed according to $\tilde{\pi}^{(n)}_i$.
\end{remark}

 \subsection*{Proof of Theorem \ref{thm:finite_PR}}
 
 In order to carry out the proof,  we  define a modified dynamics $\{\tilde{y}_{i}^{(n)}(t)\}_{t \geq 0, i \in B_n(0)}$
 which is coupled with the evolution of $\{{y}_{i}^{(n)}(t)\}_{t \geq 0, i \in B_n(0)}$.
 For notational brevity, in this section, we will drop the superscript $n$ since all systems of interest 
 are on a fixed torus $B_n(0)$. In particular, we write $\{y_i(t)\}_{i \in B_n(0)} := \{{y}_{i}^{(n)}(t)\}_{ i \in B_n(0)} $ and 
 $\{\tilde{y}_{i}(t)\}_{ i \in B_n(0)} := \{\tilde{y}_{i}^{(n)}(t)\}_{ i \in B_n(0)}$.
 \\
 
 We construct the modified dynamics such that it satisfies $\tilde{y}_{i}(t) \geq y_i(t)$ a.s.
 for all $i \in B_n(0)$ and $t \geq 0$. We will then conclude that this modified dynamics
 is positive recurrent, which, by standard coupling arguments, will imply that $\{y^{(n)}_i(t)\}_{i \in B_n(0)}$
 is positive recurrent. In order to construct the process $\{\tilde{y}_i(t)\}_{ i \in B_n(0)}$,
 we  need some constants that depend on $n$, which we  omit for notational brevity.
 Given $(\sum_{i \in \mathbb{Z}^d }a_i)^{-1} - \lambda := 2 \epsilon > 0$,
 we choose $h$ small enough so that $\mathbb{P}[J \leq h]  - \lambda h \geq \epsilon$, where $J$
 is an exponential random variable of mean $ \sum_{i \in \mathbb{Z}^d }a_i$.
 We discretize time into `slots'  of duration $h$, with the slot boundaries at times
 $-2h,-h,0,h,2h,\cdots$. From the driving process $(\mathcal{A}_i,\mathcal{D}_i)_{i \in B_n(0)}$,
 we  construct a (deterministically) modified process $(\widehat{\mathcal{A}}_i,\widehat{\mathcal{D}}_i)_{i \in B_n(0)}$
 such that $\mathcal{A}_i \subseteq \widehat{\mathcal{A}}_i$ and $\widehat{\mathcal{D}}_i \subseteq \mathcal{D}_i$ for all $i \in B_n(0)$.
 The modification $\{\widehat{\mathcal{A}}_i\}_{i \in B_n(0)}$  form the driving process to
 $\{\tilde{y}_i(t)\}_{i \in B_n(0)}$ and $\{\widehat{\mathcal{D}}_i\}_{i \in B_n(0)}$,
  the potential departure process to $\{\tilde{y}_i(t)\}_{i \in B_n(0)}$,
 with an additional re-normalization step which we will describe in Subsection \ref{subsec:modified_dynamics}.

 \subsubsection*{Modification of the Arrival and Departure Events}
 \label{appendix_torus_modification}
 We implement the modification to the arrival process as follows. In any time slot $[mh,(m+1)h)$, 
 $m \in \mathbb{Z}$, if $\max_{i \in B_{n}(0)}  \mathcal{A}_i([mh,(m+1)h)) < K$, 
 then the process $\{\widehat{\mathcal{A}}_i\}_{i \in B_n(0)}$ is identical to $\{\mathcal{A}_i\}_{i \in B_n(0)}$
 in the interval $[mh,(m+1)h)$. On the other hand, if $\max_{i \in B_{n}(0)} \mathcal{A}_i([mh,(m+1)h)) \geq K$,
 then there must exist a $p \in [mh,(m+1)h)$ and an $l \in B_n(0)$ such that $\mathcal{A}_l([mh,p]) = K$
 and $\mathcal{A}_{l^{'}}([mh,p]) < K$ for all $l^{'} \in B_n(0) \setminus \{l\}$ almost surely.
 In other words, there will be a unique first time when a certain queue will receive more than $K$ customers
 in that time slot. In this case, the process $\{\widehat{\mathcal{A}}_i\}_{i \in B_n(0)}$ is identical to
 $\{{\mathcal{A}}_i\}_{i \in B_n(0)}$ in the interval $[mh,p]$. In the interval $(p,(m+1)h)$, for all $j \in B_n(0)$,
 the process $\widehat{\mathcal{A}}_j$,  is the superposition of all the processes
 $\{\mathcal{A}_i\}_{i \in B_n(0)}$ in the time interval $(p,(m+1)h)$. In other words,
 we keep the arrivals the same until one of the queues receives at-least $K$ customers
 in that particular slot, and, subsequently, we replicate any arrival to any queue in that slot to all other queues.
 This construction ensures that, for any slot $[mh,(m+1)h)$, 
 $\max_{i \in B_n(0)} \widehat{\mathcal{A}}_i([mh,(m+1)h)) - \min_{i \in B_n(0)} \widehat{\mathcal{A}}_i([mh,(m+1)h)) \leq K$,
 i.e., the total possible discrepancy in any time slot and any queue is at most $K$.
 We will chose $K$ sufficiently large so that for any queue $i \in B_n(0)$ and any time slot $m$,
 the expected number $\mathbb{E}[\widehat{\mathcal{A}}_i([mh,(m+1)h))] = h\lambda  + \epsilon/10$. 
 \\

 We construct the departure process $\{\widehat{\mathcal{D}}_i\}_{i \in B_n(0)}$ as follows,
 using a parameter $\delta$. We choose $\delta$ sufficiently small so that
 $\mathbb{P}[\hat{J} \leq h]  - \lambda h \geq 9\epsilon/10$, where $\hat{J}$ is an exponential
 random variable with mean $(\sum_{i \in \mathbb{Z}^d }a_i(1 + \delta))$. To construct
 the process $\{\widehat{\mathcal{D}}_i\}_{i \in B_n(0)}$,  we do a selection by marks of the
 process $\{{\mathcal{D}}_i\}_{i \in B_n(0)}$ and retain those atoms whose marks are less than 
 or equal to $(\sum_{i \in \mathbb{Z}^d}a_i(1+\delta))^{-1}$ and delete the other atoms.
 We then perform a further deletion of points by retaining in each time slot and each queue, 
 at most one (the first one if there are many) atom whose mark $U^{(i)}$ is less than or equal to
 $(\sum_{i \in \mathbb{Z}^d}a_i(1+\delta))^{-1}$ and delete the rest to obtain the process
 $\{\widehat{\mathcal{D}}_i\}_{i \in B_n(0)}$. Thus the process $\{\widehat{\mathcal{D}}_i\}_{i \in B_n(0)}$
 is a subset of the process $\{{\mathcal{D}}_i\}_{i \in B_n(0)}$ and is such that, in any queue
 and any time-slot, $\widehat{\mathcal{D}}_i$ has at most one atom and the mark of this atom (if any)
 is less than or equal to $\sum_{i \in \mathbb{Z}^d}a_i(1+\delta))^{-1}$.
 Since $\{\mathcal{D}_i\}_{i \in B_n(0)}$ is a collection of independent PPPs,
 the construction ensures that $\{\widehat{\mathcal{D}}_i([mh,(m+1)h)\}_{i \in B_n(0),m \in \mathbb{Z}}$
 is an i.i.d. collection of $\{0,1\}$ valued random variables. Moreover, it is immediate to verify that the 
 probability that $\widehat{\mathcal{D}}_0([0,h)) = 1$ is equal to the probability that an
 exponential random variable $J$ of mean $\sum_{i \in \mathbb{Z}^d }a_i(1+\delta)$ is less than or equal to $h$. Thus note that the modified arrival process has more points and the modified departure process has less points compared to the original process.
 
 \subsubsection*{Modified Dynamics}
 \label{subsec:modified_dynamics}
 
 Before defining the dynamics of the process $\{\tilde{y}_i(t)\}_{i \in B_n(0)}$,
 we need to define two large integer constants $r_0$ and $y_0$.
 We first pick $r_1 \in \mathbb{N}$ such that, for all $r \geq r_1$, we have
 \begin{align}
 \frac{1}{r} \mathbb{E} \max_{j \in B_n(0)} \sum_{i=1}^{r} \widehat{\mathcal{A}}_{j}([(i-1)h,ih)) \leq \lambda h + \epsilon/5.
 \label{eqn:arrival_major}
 \end{align} 
 Note that such a choice of $r_1$ is possible since we know that the following limit
 \begin{align*}
 \lim_{r \rightarrow \infty}\frac{1}{r}  \max_{j \in B_n(0)} \sum_{i=1}^{r} \widehat{\mathcal{A}}_{j}([(i-1)h,ih))  = \mathbb{E}[\widehat{\mathcal{A}}_{0}([0,h))] = \lambda h + \epsilon/10,
 \end{align*}
 exists almost surely and in $L^{1}$. The almost sure limit is a consequence of the Strong Law of Large Numbers
 and the $L^{1}$ limit follows since  the random variables
 $\left\{ \frac{1}{r}  \sum_{i=1}^{r} \widehat{\mathcal{A}}_{0}([rh,(r+1)h)) \right\}_{j \in B_{n}(0), r \in \mathbb{N}}$
 are uniformly integrable. Uniform integrability follows as the second moment of
 $\frac{1}{r}  \sum_{i=1}^{r} \widehat{\mathcal{A}}_{0}([rh,(r+1)h))$ is uniformly bounded in $r$ and $j$.
 Similarly we pick $r_0 \geq r_1$ such that, for all $r \geq r_0$, we have
 \begin{align}
 \frac{1}{r} \mathbb{E} \min_{j \in B_n(0)} \sum_{i=0}^{r} \widehat{\mathcal{D}}_j([(i-1)h,ih)) \geq \lambda h + 4\epsilon/5.
 \label{eqn:dep_minor}
 \end{align}
 Such a choice for $r_0$ is possible since the following limit
 \begin{align*}
 \lim_{r \rightarrow \infty}\frac{1}{r}  \min_{j \in B_n(0)} \sum_{i=0}^{r} \widehat{\mathcal{D}}_j([(i-1)h,ih)) = \lambda h + 9\epsilon/10,
 \end{align*}
 exists a.s. Moreover since for all $r \in \mathbb{N}$ and $j \in B_n(0)$, 
 $\frac{1}{r}   \sum_{i=0}^{r} \widehat{\mathcal{D}}_j([(i-1)h,ih)) $
 is bounded above and below by $1$ and $0$ respectively, the limit also exist in $L^{1}$ and hence we can choose such a $r_0$. Now, we  pick a positive integer $u_0$ large enough so that 
 \begin{align}
 1 \leq \frac{u_0 + (K+1) r_0}{u_0 -  r_0} \leq (1 + \delta).
 \label{eqn:dep_minor_queue}
 \end{align}

 The modified process $\{\tilde{y}_i(t)\}_{i \in B_n(0)}$ evolves with the same update rules as 
 $\{{y}_i(t)\}_{i \in B_n(0)}$, except that it uses the modified arrival and departure process
 $(\widehat{\mathcal{A}}_i,\widehat{\mathcal{D}}_i)_{i \in B_n(0)}$ along with a key \emph{renormalization} step.
 We will use the constants $r_0,u_0$ to {re-normalize} so that, for all $m \in \mathbb{N}$, we 
 have $\max_{i \in B_{n}(0)}  \tilde{y}_i(mr_0h) = \min_{i \in B_{n}(0)} \tilde{y}_i(mr_0h) \geq u_0$.
 In the rest of this proof, we will assume (without loss of generality) that we start the modified
 system at time $0$ with all queues having the same number $y \geq u_0$ of customers, i.e.,
 for all $i \in B_n(0)$, we assume $\tilde{y}_i(0)  = y \geq y_i(0)$.
 The modified system $\{\tilde{y}_i(t)\}_{i \in B_n(0)}$ is driven by the modified driving
 process $(\widehat{\mathcal{A}}_i,\widehat{\mathcal{D}}_i)_{i \in B_n(0)}$, with the same rules of evolution
 as those of $\{{y}_i(t)\}_{i \in B_n(0)}$, but with a {re-normalization} step.
 More precisely, at times $mr_0h$, for $m \in \mathbb{N}$ (i.e., at the end of every $r_0$ time slots),
 we add customers and re-normalize as follows:
 \begin{align}
 \tilde{y}_i(mr_0h) = \max(u_0, \max_{j \in B_{n}} \tilde{y}_j(mr_0h^{-})),
 \label{eqn:re_norm}
 \end{align}
 where $mr_0h^{-}$ is the time instant just before $mr_0h$.
 In words, we add more customers so that all queues in the modified dynamics have the same number
 of customers and this number is at least $u_0$. Thanks to the monotonicity lemmas \ref{lem:mono1} and \ref{lem:mono2},
 we still have that, for all $t \geq 0$ and all $j \in B_n(0)$, $\tilde{y}_j(t) \geq y_j(t)$ almost surely.
 \\
 
 To conclude about the positive recurrence of $\{\tilde{y}_i(t)\}_{i \in B_n(0)}$, we consider the stochastic
 process $\{Y_m\}_{m \in \mathbb{N}}$ such that $Y_m := \tilde{y}_0(m r_0h)$, i.e., the modified dynamics
 sampled after every $r_0h$ units of time. Note that due to the re-normalization, it does not matter which
 queue we sample, since all of them  have the same number of customers.
 It is straightforward to observe that the process $(Y_m)_{m \in \mathbb{N}}$ is a Markov process on 
 the set of integers $\{u_0,\cdots\}$ with respect to the filtration $\{\mathcal{F}_m\}_{m \in \mathbb{N}}$,
 where $\mathcal{F}_m$ is the sigma-algebra generated by $(\mathcal{A}_i,\mathcal{D}_i)_{i \in B_n(0)}$ up to time $mr_0h$.
 This follows from the fact that the process $(\mathcal{A}_i,\mathcal{D}_i)_{i \in B_n(0)}$ is Poisson and
 has independent increments, and the modification to obtain $(\widehat{\mathcal{A}}_i,\widehat{\mathcal{D}}_i)_{i \in B_n(0)}$
 is \emph{local}, i.e., only depends on the realization of $(\mathcal{A}_i,\mathcal{D}_i)_{i \in B_n(0)}$
 in that particular slot. The following simple lemma shows that $(Y_m)_{m \in \mathbb{N}}$ essentially
 behaves as a $GI/GI/1$ queue.
 \begin{lemma}
 	The Markov Chain $(Y_m)_{m \in \mathbb{N}}$ satisfies the $GI/GI/1$ recursion
 	\begin{align*}
 	Y_{m+1} = \max  \bigg(Y_m + \max_{j \in B_n(0)}  \chi_{j,m}^{(n)}, u_0 \bigg),
 	\end{align*}
 	where 
 	\begin{align*}
 	\chi_{j,m}^{(n)} = \widehat{\mathcal{A}}_{j}[ mr_0h ,(m+1)r_0h) -  \widehat{\mathcal{D}}_{j}[ mr_0h ,(m+1)r_0h).
 	\end{align*}
 	\label{lem:yn_struct}
 \end{lemma}

 \begin{proof}
 	First notice that the total number of cumulative arrivals in queue $j \in B_n(0)$ from time $[mr_0h, (m+1)r_0h]$
 	is precisely $ \widehat{\mathcal{A}}_{j}([mr_0h , (m+1)r_0h]) $. To prove the lemma,
 	we will show that for all queues $j \in B_n(0)$ and all slots $l \in \mathbb{N}$,
 	exactly one departure occurs from queue $j$ in that time slot \emph{if and only if} $\widehat{\mathcal{D}}_j([lh,(l+1)h]) > 0$.
 	If we establish this fact, then, at time $(m+1)r_0h^{-}$, the number of customers in any queue $j \in B_n(0)$ 
 	will precisely be equal to 
 	$ Y_m + \widehat{\mathcal{A}}_{j}([mr_0h , (m+1)r_0h]) -  \widehat{\mathcal{D}}_{j}([mr_0h , (m+1)r_0h])$.
 	The lemma will then follow from the definition of the re-normalization step in Equation (\ref{eqn:re_norm}). 
 	\\

 	The claim above is established by a sequence of elementary observations.
 	It follows from the construction of  $\{\widehat{\mathcal{A}}_i\}_{i \in B_n(0)}$ that,
 	for all slots $l \in \mathbb{N}$, 
 	$\max_{i \in B_{n}(0)} \widehat{\mathcal{A}}_i([lh,(l+1)h]) -  \min_{i \in B_{n}(0)} \\ \widehat{\mathcal{A}}_i([lh,(l+1)h]) \leq K$.
 	Since we have at most one departure event in the process $\{\widehat{\mathcal{D}}_i\}_{i \in B_n(0)}$ per time slot, per queue,
 	we have $\tilde{y}_i(t) \geq u_0 - r_0$, for all $t \geq 0$ and all $i \in B_n(0)$.
 	There is only a loss of at most $r_0$, since we re-normalize to at-least $u_0$ customers per queue after $r_0$ slots. 
 	These facts yield that, for every instant $t \geq 0$, we have 
 	$\max_{i \in B_{n}(0)} \tilde{y}_i(t) - \min_{i \in B_{n}(0)} \tilde{y}_i(t) \leq (K+1) r_0$, since we 
 	equalize all queues after $r_0$ slots. Thus the rate of departure in the modified dynamics at any instant 
 	$t \geq 0$ and any queue $i \in B_{n}(0)$ is at least  
 	$\frac{\tilde{y}_i(t) }{(\tilde{y}_i(t) + (K+1)r_0) \sum_{i \in \mathbb{Z}^d}a_i}$.
 	Now since $\tilde{y}_i(t) \geq u_0 - r_0$, we have from Equation (\ref{eqn:dep_minor_queue}) that
 	the departure rate at any time $t \geq 0$ and any queue $i \in B_n(0)$ is at least
 	$(\sum_{i \in \mathbb{Z}^d }a_i(1+\delta))^{-1}$. By construction of the process $\{\widehat{\mathcal{D}}_i\}_{i \in B_n(0)}$,
 	we have at most one departure event per slot and per queue, and its mark is no larger than
 	$(\sum_{i \in \mathbb{Z}^d }a_i(1+\delta))^{-1}$. This concludes the claim that, in each slot
 	$m\in \mathbb{N}$ and queue $i \in B_n(0)$, there is a departure from that queue during the
 	interval $[mh,(m+1)h)$ if and only if $\widehat{\mathcal{D}}_i([mh,(m+1)h)) > 0$. 
 \end{proof}
 
 \subsubsection*{Concluding the Proof of Theorem \ref{thm:finite_PR}}
 \begin{proof} 
 	Lemma \ref{lem:yn_struct} gives that $(Y_m)_{m \in \mathbb{N}}$ is of the form of $Y_{m+1} = \max(Y_m + \xi_m,{u}_0)$ where $\{\xi_m\}_{m \in \mathbb{N}}$ is an i.i.d. sequence of random variables. Moreover we have
 	\begin{align*}
 	\mathbb{E}[\xi_0] = \mathbb{E}\left[ \max_{j \in B_n} \left(  \widehat{\mathcal{A}}_{j}([0,r_0h]) -  \widehat{\mathcal{D}}_{j}([0,r_0h]) \right)  \right]  \leq -r_0 \frac{3 \epsilon}{5},
 	\end{align*}
 	where the second equality follows from Equations (\ref{eqn:arrival_major}) and (\ref{eqn:dep_minor}).
 	Moreover $\{\xi_m\}_{m \in \mathbb{N}}$ has exponential moments, i.e., for all $c > 0$, $\mathbb{E}[e^{c \xi_0}] < \infty$,
 	since Poisson random variables have exponential moments. Thus standard results on $GI/GI/1$ queues immediately
 	imply that the Markov chain $(Y_m)_{m \in \mathbb{N}}$ is ergodic and its stationary distribution $\zeta$
 	on the set of integers $\{u_0,u_0+1,\cdots\}$ is such that there exists $c^{'}>0$ such that $\mathbb{E}[e^{c^{'}Z}] < \infty$,
 	where $Z \sim \zeta$.
 	\\

 	Now assume we start the true system $\{y_{i}^{(n)}(t)\}_{i \in B_n(0)}$ with some arbitrary
 	initial condition $\mathbf{a} \in \mathbb{N}^{|B_n(0)|}$. Then we start the modified system
 	$\{\tilde{y}_i(t)\}_{i \in B_n(0)}$ with all queues having the same number of customers equal 
 	to $\max(u_0,||\mathbf{a}||_{\infty})$. The monotonicity properties from Lemma \ref{lem:mono1} and \ref{lem:mono2} 
 	imply that, for all $j \in B_n(0)$ and all $t \geq 0$, we have that $y_j^{(n)}(t) \leq \tilde{y}_j(t)$.
 	Furthermore, we have $Y_m := \tilde{y}_0(mr_0h)$, $m \in \mathbb{Z}$ is an ergodic Markov chain with stationary
 	distribution $\zeta$ on the integers $\{u_0,u_0+1,\cdots\}$. From standard results on the $GI/GI/1$ queue,
 	the sequence of random variables $Y_m$ converges in total variation to the distribution $\zeta$ as $m$ goes
 	to infinity. Now since, for all $t \geq mr_0h$ and all $j \in B_n(0)$, we have the trivial inequality
 	\begin{align*}
 	\tilde{y}_j(t) \leq Y_{m} + \sum_{i \in B_n(0) } (\widehat{\mathcal{A}}_i([mr_0h,t]) + \mathcal{D}_i([mr_0h,t])),
 	\end{align*}
 	where $Y_{m}$ is independent of  $\sum_{i \in \mathbb{Z}^d } (\widehat{\mathcal{A}}_i([mr_0h,t]) + \mathcal{D}_i([mr_0h,t]))$.
 	Hence as $t \rightarrow \infty$, for all $j \in B_n(0)$, $y_j(t)$ is stochastically dominated by a random variable
 	that converges in total variation to a non-degenerate distribution on $\{u_0,u_0+1,\cdots\}$, which is the sum of
 	independent random variables distributed as  $\zeta$ and 
 	$\sum_{i \in B_n(0) } (\widehat{\mathcal{A}}_i([0,r_0h]) + \mathcal{D}_i([0,r_0h]))$ respectively.
 	Since $\{y_i(t)\}_{i \in B_n(0)}$ is a Markov process on a finite-dimensional space,
 	the stochastic domination by a stationary process implies that $\{y_i(t)\}_{i \in B_n(0)}$ is ergodic.
 	Moreover since $Y_m$ has exponential moments, it follows that the stationary distribution of 
 	$\{y_i(t)\}_{i \in B_n(0)}$  also has exponential moments.
 	
 \end{proof}

\section{Rate Conservation Arguments}
\label{sec:rcl}


This section forms the central tool used in our analysis. We shall consider  the space truncated systems introduced before to explicitly write down differential equations for certain functionals of the dynamics. The key result we will establish in this section is a closed form formula for the mean of the steady state queue length of the space truncated torus system. To do so, we will use the general rate conservation principle of Palm calculus  \cite{baccelli_bremaud} to derive certain relations between the system parameters in steady state. We shall assume $\lambda < \frac{1}{\sum_{j \in \mathbb{Z}^d}a_j}$ throughout in this section. We shall let $n > L$ be arbitrary and fixed for the rest of this section. In this section, we shall again consider the two space truncated stochastic processes $\{y_i^{(n)}(\cdot)\}_{i \in B_n}$ and $\{z_i^{(n)}(\cdot)\}_{i \in B_n}$ to be in steady-state. Recall that the process $\{y_i^{(n)}(\cdot)\}_{i \in B_n}$ is one wherein the set $B_n$ is viewed as a torus with its end points identified and the process $\{z_i^{(n)}(\cdot)\}_{i \in B_n}$  is one with the end effects, i.e., with the set $B_n$ is viewed as a subset of $\mathbb{Z}^d$. Furthermore, we denote by $\pi^{(n)}$,  the steady state distribution of $y_0^{(n)}(0)$ {\color{black} and by the translation invariance on the torus, the steady state distribution of $y_i^{(n)}(0)$, for all $i \in B_n$. Similarly, for all $i \in B_n$, we shall denote by
$\tilde{\pi}^{(n)}_i$,  the steady state distribution of the marginal $z_i^{(n)}(t)$. Notice that the marginal distributions in the process ${z_i^{(n)}(\cdot)}_{i \in B_n}$ depend on the coordinate, unlike in the process ${y_i^{(n)}(\cdot)}_{i \in B_n}$}. For notational simplicity, we shall denote by $\mu^{(n)}$, the mean of $\pi^{(n)}$ and for all $i \in B_n$, by $\nu_i^{(n)}$,  the mean of $\tilde{\pi}^{(n)}_i$. In this section, we shall study two stochastic processes - $\{\mathbb{I}_t\}_{t \in \mathbb{R}}$ and $\{\tilde{\mathbb{I}}_t\}_{t \in \mathbb{R}}$, with $\mathbb{I}_t := y_0^{(n)}(t)(\sum_{j \in \mathbb{Z}^d} a_j y_j^{(n)}(t) )$ and  $\tilde{\mathbb{I}}(t) := \sum_{i \in B_n} z_i^{(n)}(t) (\sum_{j \in \mathbb{Z}^d} a_j z_{i-j}^{(n)}(t)  )$. If one were to be more precise, one should  use the notation $\mathbb{I}^{(n)}_t$ and $\tilde{\mathbb{I}}^{(n)}_t$, but we drop the superscript $n$ to simplify the notation. In words, the process $(\{\mathbb{I}_t\})_{t \in \mathbb{R}}$ considers the interference seen at a typical queue in $\{y_i^{(n)}(\cdot)\}_{i \in B_n}$, the system where the set $B_n$ is viewed as a torus and $(\{\tilde{\mathbb{I}}_t\})_{t \in \mathbb{R}}$ considers the total interference in the process $\{z_i^{(n)}(\cdot)\}_{i \in B_n}$,  the system with boundary effects, where the set $B_n$ is viewed as a subset of $\mathbb{Z}^d$. Observe that since $B_n$ is a torus, the marginals of the process $\{y_i^{(n)}(\cdot)\}_{i \in B_n}$ are identical and hence, we can consider a typical queue, but as the marginals of $\{z_i^{(n)}(\cdot)\}_{i \in B_n}$ are different due to the edge effects, we study the total interference at all queues instead of the interference seen at a typical queue. Since $\lambda < \frac{1}{\sum_{j \in \mathbb{Z}^d}a_j}$, and the systems $\{y_i^{(n)}(t)\}_{i \in B_n}$  and $\{z_i^{(n)}(t)\}_{i \in B_n}$ are in steady state, and the queue lengths in both systems posses exponential moments, it is the case that for all $t \in \mathbb{R}$, $\mathbb{E}[\mathbb{I}_t] < \infty$ and $\mathbb{E}[\tilde{\mathbb{I}}(t)] < \infty$. 
\\

The main technical results of this section are  Propositions \ref{prop:rcl_diff_eqn},  \ref{prop:diff_eqn_edge_effects} and Lemma \ref{lem:2ndmom_torus}. These will then help us to derive closed form expressions for the mean queue length and a bound on the second moment for the original infinite system.

	\begin{proposition}

	\begin{multline}
	\frac{d}{dt} \mathbb{E}[\mathbb{I}_t] = 0 =  \lambda a_0 + 2 \lambda {\left( \sum_{j \in B_n} a_j \right) }
	\mu^{(n)}  - \\ \mathbb{E} \bigg[ R(0) \bigg(a_0 (2 y_0^{(n)}(0) -1)
	+  \sum_{i \in {B_n} \setminus \{0\}} a_i y_i^{(n)}(0) \bigg) 
	+ \sum_{i \in {B_n} \setminus \{0\}}  R(i)a_i y_0^{(n)}(0)  \bigg] ,
	\label{eqn:rcl_diff_eqn}
	\end{multline}
	where for any $i \in B_n$,
	\begin{align*}
	R(i) := \frac{y^{(n)}_i(0)}{\sum_{j \in {B_n}} {\color{black}a_{d_n(i,j)} y^{(n)}_j(0) }}.
	\end{align*}
	\label{prop:rcl_diff_eqn}
\end{proposition}

\begin{proof}
	 We provide a heuristic derivation of the differential equation using the PASTA property of the arrival and departure process and skip all the technical details  as it is standard. For example see the Appendix of \cite{baccelli_p2p_journal} for an account of the derivation. In a small interval of time $\Delta t$, in every queue, there will be exactly one arrival  with probability roughly $\lambda \Delta t$. The chance that two or more arrivals occur in a time interval $\Delta t$ in the entire network is $\BigO{(\Delta t)^2}$, where the $\BigO{\cdot}$ hides all system parameters (for ex. $\lambda,n$) other than $\Delta t$ as they are fixed.
	 On an arrival at queue $0$, the increase in the quantity $\mathbb{I}_0$ is 
	$\mathbb{E}[(y_0^{(n)}+1)(a_0(y_0^{(n)}+1) + \sum_{j \in { B_n} \setminus \{0\}} a_j y_j^{(n)} )
	- y_0^{(n)} (\sum_{j \in{ B_n}} a_j y_j^{(n)})]$,
	which is equal to $\mathbb{E}[a_0 + \sum_{j \in { B_n}} a_j y_j^{(n)} ]$.
	Similarly, the average increase in $\mathbb{I}_0$ due to an arrival in the neighboring queues of $0$ is
	$\mathbb{E}[(y_0^{(n)})(a_i(y_i^{(n)}+1) + \sum_{j \in {B_n} \setminus \{i\}} a_j y_j^{(n)} )
	- y_0^{(n)} (\sum_{j \in { B_n}} a_j y_j^{(n)})]$, which is equal to {$\mathbb{E}[a_i y_0^{(n)}]$}.
	The chance that there are two or more arrivals is $\BigO{(\Delta t)^2}$, which is small.
	Thus, the average increase due to arrivals is  
	{$\lambda \Delta t \mathbb{E}[a_0(y_0^{(n)} + 1) 
		+ \sum_{j \in B_n} a_j y_j^{(n)} + \sum_{j \in B_n \setminus \{0\}}a_j y_0^{(n)}  ] + \BigO{(\Delta t)^2}$.}
	After simplification, and using the fact that the variables $y_j^{(n)}$ all have the same mean,
	we get that the average increase in time $\Delta t$ 
	is \begin{align} {\lambda \Delta t
		\left(a_0 +2 \mu^{(n)} \left(\sum_{j \in B_n}a_j\right) \right)  + \BigO{(\Delta t)^2}.} \label{eqn:rcl1_a} \end{align}

	Likewise, with probability $R(i)\Delta t$, there will be a departure from queue $i$. When a customer leaves from queue $0$, which occurs with probability $R(0)\Delta t$
	the average decrease in $\mathbb{I}_0$ is then 
	$\mathbb{E}[(a_0((y_0^{(n)})^2 - a_0(y_0^{(n)} - 1)^2 + \sum_{i \in {B_n} \setminus \{0\}} a_i y_i^{(n)} )]$. 
	Similarly, a departure from queue $i$, which occurs with probability $R(i)\Delta t$ results in
	an average decrease in $\mathbb{I}_0$ of $ \mathbb{E} [a_i y_0^{(n)}]$. The chance that two or more possible departures occur in time $\Delta t$ is $\BigO{(\Delta t^2}$, which is small. Thus, the total average decrease in $\mathbb{I}_0$ due to 
	departures is 
	\begin{align}
	\Delta t \mathbb{E} \bigg[ R(0) \left(a_0 (2 y_0^{(n)} -1) + \sum_{i \in {B_n} \setminus \{0\}} a_iy_i^{(n)} \right)
	+  \sum_{i \in {B_n} \setminus \{0\}}  R(i)a_i y_0^{(n)}  \bigg] + \BigO{(\Delta t)^2}.
	\label{eqn:rcl1_d}
	\end{align}
	
	Hence, we see from Equations (\ref{eqn:rcl1_a}) and (\ref{eqn:rcl1_d}), that 
	\begin{multline*}
	\frac{1}{\Delta t} \mathbb{E}[ \mathbb{I}(t+ \Delta t) - \mathbb{I}(t)]  \\ = \lambda 
	\left(a_0 +2 \mu^{(n)} \left(\sum_{j \in B_n}a_j\right) \right) -  \mathbb{E} \bigg[ R(0) \left(a_0 (2 y_0^{(n)} -1) + \sum_{i \in {B_n} \setminus \{0\}} a_iy_i^{(n)} \right)
	+  \sum_{i \in {B_n} \setminus \{0\}}  R(i)a_i y_0^{(n)}  \bigg] \\ + \BigO{\Delta t}
	\end{multline*}
	
The proposition is concluded by letting $\Delta t$ go to $0$.

\end{proof}

Now, we compute the differential equation for the space truncated system, by carefully taking into consideration the `edge effects' introduced by the truncation to the set $B_n$. {\color{black}Denote by the set $B_n^{(I)} \subset B_n$, where $B_n^{(I)} := \{z \in B_n : \forall y \text{ s.t. } ||y-z||_{\infty} \leq L, y \in B_n\}$. In words, the set $B_n^{(I)}$ is the set of all points $z \in B_n$ such that the $l_{\infty}$ ball of radius $L$ is completely contained in $B_n$}.

	\begin{proposition}
	\begin{align*}
	\frac{d}{d t} 	\mathbb{E}[ \tilde{ \mathbb{I}}(t)] = 0  \geq  -2(1 - \lambda \sum_{j \in \mathbb{Z}^d} a_j) \sum_{i \in B_n^{(I)}}\nu_i^{(n)} + 2 \lambda a_0 |B_n| - 2\sum_{i \in B_n \setminus B_n^{(I)}} \nu_i^{(n)}.
	\end{align*}
	\label{prop:diff_eqn_edge_effects}
\end{proposition}

\begin{proof}
		A rigorous proof of this is standard and we skip it. For example see \cite{baccelli_p2p_journal}. Instead we outline the computations required in establishing this proposition. Furthermore to lighten the notation in the proof, we  drop the superscript $n$, as $n$ is fixed and does not change in the course of the proof. Thus, we shall denote  $\{z_i^{(n)}(t)\}_{i \in B_n}$ by $\{z_i(t)\}_{i \in B_n}$ and  the steady-state means by $(\nu_i)_{i \in B_n}$, instead of $(\nu_i^{(n)})_{i \in B_n}$. As in the proof of Proposition \ref{prop:rcl_diff_eqn}, we shall consider a small interval $\Delta t$ of time such that at most one event of either an arrival or departure occurs anywhere in the network in the set $B_n$. Roughly speaking, with probability $\lambda \Delta t$, there will be an arrival in some queue $i \in B_n$. In the rest of the proof, we shall partition the set $B_n$ into $B_n^{(I)}$ and $B_n \setminus B_n^{(I)}$. 
		\\
		
		From similar computations as in the previous proposition, if there is an arrival in queue $i \in B_n^{(I)}$, the increase in $\tilde{\mathbb{I}}(t)$ will be 
	\begin{multline*}
	\mathbb{E}\bigg[(z_i(t)+1)(a_0(z_i(t)+1) + \sum_{j \in \mathbb{Z}^d \setminus \{0\} } a_j z_{i-j}(t)  ) - z_i(t)(z_i(t)+  \sum_{j \in \mathbb{Z}^d \setminus \{0\} } a_j z_{i-j}(t)  ) + \\ \sum_{l \in B_n \setminus \{i\}} ( z_l(t)(a_{i-l} (z_i(t) +  1)+ \sum_{j \in \mathbb{Z}^d \setminus \{i\}} a_{j-l}z_j(t ) ) - z_l(t)(a_{i-l} z_i(t)  \sum_{j \in \mathbb{Z}^d \setminus \{i\}} a_{j-l}z_j(t ) )  )\bigg].
	\end{multline*}
	This follows since if there is an extra customer in queue $i$, then the total interference is increased both at queue $i$ and any other queue $j$ such that $a_{i-j} > 0$. From the PASTA property, we know that at the moment of arrival, $\{z_i(t)\}_{I \in B_n}$ is in steady-state and in particular, $\mathbb{E}[z_i(t)] = \nu_i$. Thus, the above expression can be simplified as 
	\begin{align*}
	\sum_{j \in \mathbb{Z}^d}a_j \nu_{i-j} + a_0 + a_0\nu_i + \sum_{l \in B_n \setminus \{i\}} \nu_l a_{i-l}. 
	\nonumber
	\end{align*}
	If $i \in B_n^{(I)}$, the above expression is equal to 
	\begin{align*}
	2 \sum_{j\in \mathbb{Z}^d} a_{j} \nu_{i-j} + a_0, 
	\end{align*}
	while if $i \in B_n^{(I)} \setminus B_n$, we use the trivial inequality
	\begin{align*}
	\sum_{j \in \mathbb{Z}^d}a_j \nu_{i-j} + a_0 +a_0 \nu_i + \sum_{l \in B_n \setminus \{i\}} \nu_l a_{i-l} \geq a_0.
	\end{align*}
	Thus, the average increase in the time interval $\Delta t$ in the interference $\mathbb{I}(t)$  due to an arrival event is at least
	\begin{align*}
	\lambda \Delta t \left( \sum_{i \in B_n^{(I)}} (2 \sum_{j\in \mathbb{Z}^d} a_{j} \nu_{i-j} + a_0) + \sum_{i \in B_n \setminus B_{n}^{(I)}} a_0 \right)+ \BigO{\Delta t^2}.
	\end{align*}
 Since for all $i \in B_n^{(I)}$, the $l_{\infty}$ ball of radius $L$ is contained within the set $B_n$, we can further simplify the above expression as 
	\begin{align*}
	\lambda \Delta t \left(\sum_{i \in B_n^{(I)}} 2\nu_i \sum_{j \in \mathbb{Z}^d}a_j +a_0 \right) + \lambda a_0 \Delta t |B_n \setminus B_n^{(I)}| + \BigO{\Delta t^2}.
	\end{align*}

	Similarly, we can compute the average decrease in $\tilde{\mathbb{I}}(t)$ due to a departure event. Roughly, the probability of a departure from any queue $i \in B_n$ is given by $R_i(t)$ where $R_i(t) = \frac{z_i(t)}{\sum_{j \in \mathbb{Z}^d} a_{i-j}z_j(t)}$. If there is a departure from queue $i$, the average decrease can be computed as 
	\begin{multline*}
	\mathbb{E}\bigg[(z_i(t))(a_0z_i(t) + \sum_{j \in \mathbb{Z}^d \setminus \{0\} } a_j z_{i-j}(t)  ) - (z_i(t)-1)(a_0(z_i(t)-1)+  \sum_{j \in \mathbb{Z}^d \setminus \{0\} } a_j z_{i-j}(t)  ) \\ + \sum_{l \in B_n \setminus \{i\}} ( z_l(t)(a_{i-l} z_i(t) + \sum_{j \in \mathbb{Z}^d \setminus \{i\}} a_{j-l}z_j(t ) ) - z_l(t)(a_{i-l} (z_i(t)-1)  \sum_{j \in \mathbb{Z}^d \setminus \{i\}} a_{j-l}z_j(t ) )  )\bigg].
	\end{multline*}
	We do not need to worry about the fact that $z_i(t) - 1$ can be negative since, in this case, the rate of departure $R_i(t)$ will be $0$. Thus the average rate of decrease in the interference due to a departure can be written  as 
	\begin{align*}
	\Delta t \left( \sum_{i \in B_n} \mathbb{E}[ R_i(t) ( a_0z_i(t) + \sum_{j \in \mathbb{Z}^d}a_j z_{i-j}(t)    -a_0) +      \sum_{l \in B_n \setminus \{i\}} a_{i-l}z_l(t)     ] \right)+ \BigO{\Delta t^2}.
	\end{align*}
	Using the fact that for all $i \in B_n$ and all $t \in \mathbb{R}$, we have $R_i(t)(\sum_{j \in \mathbb{Z}^d}  a_j z_{i-j}(t) ) = z_i(t)$, we can simplify the average rate of decrease as 
	\begin{align*}
	\Delta t (\sum_{i \in B_n}2 \nu_i- a_0\sum_{i \in B_n}\mathbb{E}[R_i(t)]) + \BigO{\Delta t^2}.
	\end{align*}
	However, as $\{z_i(\cdot)\}_{i \in B_n}$ is a stationary process, $ \sum_{i \in B_n}\mathbb{E}[R_i(t)] = \lambda |B_n|$. This then gives that average rate of change in $\mathbb{E}[\mathbb{I}(t)]$ is 
	\begin{multline*}
	\frac{1}{\Delta t}	\mathbb{E}[\mathbb{I}(t \Delta t) - \mathbb{I}(t)] \geq 	\lambda (\sum_{i \in B_n^{(I)}} 2 \nu_i ( \sum_{j\in \mathbb{Z}^d} a_{j}) + a_0) + \lambda a_0 |B_n \setminus B_n^{(I)}| -  \sum_{i \in B_n}2 \nu_i+ a_0\sum_{i \in B_n}\lambda + \BigO{\Delta t}.
	\end{multline*}

	Letting $\Delta t$ go to $0$, we obtain the bound in Proposition \ref{prop:diff_eqn_edge_effects}.
\end{proof}

 We now  state  Lemma \ref{lem:rcl_formula_finite} which holds as a consequence of the rate conservation argument. This establishes a closed form expression for the mean queue length in steady state in the space truncated torus system $\{y_i^{(n)}(t)\}_{i \in B_n}$.
 Recall that the system $\{y_i^{(n)}(t)\}_{i \in B_n}$ is in steady state. Thus, the stochastic process $(\mathbb{I}_t)_{t \in \mathbb{R}}$ is stationary. In particular, $\frac{d}{dt}\mathbb{E}[\mathbb{I}_t]$ is equal to $0$. Thus, from Proposition \ref{prop:rcl_diff_eqn}, we have the following key lemma 
\begin{lemma}
	For all $\lambda < \frac{1}{\sum_{j \in \mathbb{Z}^d} a_j}$ and all $n > L$, 
	\begin{align}
	\mu^{(n)} = \frac{\lambda a_0}{1 - (\sum_{j \in \mathbb{Z}^d} a_j) \lambda}.
	\end{align}
	\label{lem:rcl_formula_finite}
\end{lemma}
\begin{remark}
Note that we assumed $a_0=1$ in the model. For completeness, we give the derivation for any general $a_0 > 0$.
\end{remark}

This lemma in particular yields that the mean number of customers in the steady state of
the space truncated torus  is  independent of $n$, provided $n$ is large enough. {\color{black} This in particular gives $\sup_{n} \mu^{(n)} < \infty$}.
\begin{proof}
	
From Equation (\ref{eqn:rcl_diff_eqn}), we get
	\begin{align*}
	\lambda
	{\left(a_0 +2\mu^{(n)} \left(\sum_{j \in \mathbb{Z}^d}a_j\right) \right)}
    =  \mathbb{E} \bigg[ R(0)\left((a_0(2x_0^{(n)} -1) + \sum_{i \in {B_n} \setminus \{0\}} a_ix_i^{(n)} \right)
	+  \sum_{i \in {B_n} \setminus \{0\} }R(i)a_i x_0^{(n)}  \bigg].
	\end{align*} 
	
	Now, we use the following version of {the Mass Transport Principle for unimodular random graphs (see also \cite{lyons_peres_book})}:
	
	\begin{proposition}
		The following formula holds.
		\begin{align*}
		\mathbb{E}\left[\sum_{i \in {B_n}\setminus \{0\} }R(i)a_i y_0^{(n)} \right] 
		= \mathbb{E}\left[\sum_{i \in  {B_n} \setminus \{0\} }R(0)a_{i} y_i^{(n)}\right].
		\end{align*} 
		\label{prop:unimod_mtp}
	\end{proposition}
	\begin{proof}
	The proof follows from the standard argument of Mass Transport  involving swapping double sums.
	Observe from the definition of the dynamics, the queue lengths $\{y_{k}^{(n)}\}_{k \in B_n}$ is translation invariant on the torus $B_n$.
	Hence, for all $j \in B_n$, the variables $y_{j}^{(n)} \sum_{i \in B_n \setminus \{j\} } R(i)a_{i-j}$ are identically distributed, and in particular have the same means.  The proposition is now proved thanks to the following calculations.
	\begin{align*}
	\mathbb{E}\left[\sum_{i \in {B_n}\setminus \{0\} }R(i)a_i y_0^{(n)} \right] &= \frac{1}{|B_n|} \mathbb{E} \left[ \sum_{j \in B_n} y_{j}^{(n)} \sum_{i \in B_n \setminus \{j\}} R(i)a_{i-j} \right]\\
	& \stackrel{(a)}{=} \frac{1}{|B_n|} \mathbb{E} \left[ \sum_{i \in B_n} R(i) \sum_{j \in B_n \setminus \{i\}} a_{i-j}y_{j}^{(n)} \right] \\
	&\stackrel{(b)}{=} \frac{1}{|B_n|} \mathbb{E} \left[ \sum_{i \in B_n} R(i) \sum_{j \in B_n \setminus \{i\}} a_{j-i}y_{j}^{(n)} \right] \\
	& \stackrel{(c)}{=} \mathbb{E} \left[ \sum_{i \in  {B_n} \setminus \{0\} }R(0)a_{i} y_i^{(n)} \right].
	\end{align*} 
	Equality $(a)$ follows by swapping the order of summations, which is licit since they each contain finitely many terms. Equality $(b)$ follows since $a_k = a_{-k}$ for all $k \in \mathbb{Z}^d$. Equality $(c)$ again follows from the fact that for all $i \in B_n$, $R(i) \sum_{j \in B_n \setminus \{i\}} a_{j-i}y_{j}^{(n)}$ are identically distributed. This is a consequence of the queue lengths $\{y_{k}^{(n)}\}_{k \in B_n}$ being translation invariant on the torus.
\end{proof}

{\color{black} We now show how to conclude the proof of Lemma \ref{lem:rcl_formula_finite}, using the conclusions of Proposition \ref{prop:unimod_mtp}}.
 Intuitively, Proposition \ref{prop:unimod_mtp}
	can be interpreted by considering the finite graph with vertices on the torus $B_n$
	with a directed edge from $i$ to $j$ in $B_n$ with weight $R(i)a_{d_n(i-j)}y_j^{(n)}$. This random graph, when
	rooted in 0, is unimodular and hence the Mass Transport Principle holds (\cite{lyons_peres_book}). 
	Since $a_i = a_{-i}$, we get that the average decrease is
	$\mathbb{E}[- a_0 R(0) + 2R(0) \sum_{i \in {B_n}} a_i y_i^{(n)}  ]$.
	Now, $\mathbb{E}[R(0)] = \lambda$, and since, for all $i \in B_n$, $\mathbb{E}[y_i^{(n)}] = \mu^{(n)}$,
	\begin{align}
	2 \lambda (\sum_{i \in  {B_n}} a_i) \mu^{(n)} + 2 \lambda a_0
	= \mathbb{E}[2 R(0) ( \sum_{i \in  {B_n}} a_i y_i^{(n)})].
	\label{eqn:rate_conserv}
	\end{align}
	
	But since $R(0)( \sum_{i \in  {B_n}} a_i y_i^{(n)}) = y_0^{(n)}$, we get 
	\begin{align*}
	\mu^{(n)} = \frac{\lambda a_0}{1 - (\sum_{i \in \mathbb{Z}^d}a_i) \lambda}.
	\end{align*}
\end{proof}

\begin{corollary}
	If $\lambda < \frac{1}{\sum_{j \in \mathbb{Z}^d} a_j}$, then the sequence of probability measures $\{\pi^{(n)}\}_{n > L}$ is tight.
	\label{cor:tightness}
\end{corollary}
\begin{proof}
From Markov's inequality, we have
	\begin{align*}
	\mathbb{P}[X > Q] \leq \frac{\lambda a_0}{Q(1 - (\sum_{i \in \mathbb{Z}^d}a_i) \lambda)} ,
	\end{align*}
	where $X$ is distributed according to $\pi^{(n)}$. Thus, for every $\varepsilon >0$, we can find $Q$ large such that $\sup_{n > L} \mathbb{P}_{\pi^{(n)}}[X > Q] < \varepsilon$.
\end{proof}

\subsection{Finiteness of Second Moments }
\label{subsec:2ndmom}
In this section, we  establish that under the conditions stated in Proposition \ref{prop:finite_second_moment}, the second moments of the marginals of the queue lengths of $\{y_i^{(n)}(\cdot)\}_{i \in B_n}$ are uniformly bounded in $n$. In order to show this, we need the following auxiliary lemma. For completeness, we provide expressions without assuming that the value of $a_0$ of the interference sequence $\{a_i\}_{i \in \mathbb{Z}^d}$ to be $1$.

\begin{lemma}
	For all $\lambda > 0$, $\{a_i\}_{i \in \mathbb{Z}^d}$, $d \in \mathbb{N}$ and $n > L$, we have $\mathbb{E}[y_0^2 \sum_{i \in B_n} R_ia_i] \leq 2c \mathbb{E}[y_0^2]$, where the constant $c$ equals $ \frac{\sqrt{a_0^2 + a_0 \sum_{j \in \mathbb{Z}^d \setminus \{0\}} a_j } \text{	}-\text{	} a_0}{\sum_{j \in \mathbb{Z}^d \setminus \{0\}} a_j }$.
	\label{lem:2ndmom_auxilary}
\end{lemma}
\begin{proof}
	From symmetry, i.e., $a_i = a_{-i}$ for all $i \in B_n$ and translation invariance on the torus, we have
	\begin{align}
	\mathbb{E}[y_0^2 \sum_{i \in B_n} R_ia_i] = \mathbb{E}[R_0  \sum_{i \in B_n} y_i^2 a_i].
	\label{eqn:2ndmom_1}
	\end{align}
	
Let $c > 0$ be such that $2c a_0 = a_0 -  (\sum_{j \in \mathbb{Z}^d \setminus \{0\}} a_i )c^2$. The only positive solution to this equation is $c = \frac{\sqrt{a_0^2 + a_0 \sum_{j \in \mathbb{Z}^d \setminus \{0\}} a_j }\text{	} -\text{	} a_0}{\sum_{j \in \mathbb{Z}^d \setminus \{0\}} a_j }$. Thus, we have the following chain of equations - 
	\begin{align}
	\sum_{i \in B_n} a_i y_i^2 &= a_0 y_0^2  + \sum_{i \in \mathbb{Z}^d \setminus \{0\}} a_i y_i^2  = 2c a_0 y_0^2 +  \sum_{i \in \mathbb{Z}^d \setminus \{0\}}  a_i (y_i^2 + c^2 y_0^2) \geq 2c y_0 \sum_{ i \in B_n} a_i y_i \label{eqn:2ndmom_2},
	\end{align}
	where the last inequality follows from  $y_i^2 + c^2y_0^2 \geq 2cy_0 y_i$. Thus, from Equations (\ref{eqn:2ndmom_1}) and (\ref{eqn:2ndmom_2}), we have
	\begin{align*}
	\mathbb{E}[R_0  \sum_{i \in B_n} y_i^2 a_i] &\geq 2c \mathbb{E}[R_0 y_0 \sum_{ i \in B_n} a_i y_i  ] = 2c \mathbb{E}[y_0^2].
	\end{align*}
\end{proof}

\begin{lemma}
	For all $\lambda < \frac{2}{3}\frac{1+c}{\sum_{j \in \mathbb{Z}^d}a_j}$, we have $\mathbb{E}[(y_0^{(n)})^2] \leq \frac{ 2\mu( \lambda + \lambda \sum_{j \in \mathbb{Z}^d}a_j +1) }{2(1+c) - 3 \lambda \sum_{j \in \mathbb{Z}^d}a_j}$, where the constant $c$ is  $\frac{\sqrt{a_0^2 + a_0 \sum_{j \in \mathbb{Z}^d \setminus \{0\}} a_j } \text{	}-\text{	} a_0}{\sum_{j \in \mathbb{Z}^d \setminus \{0\}} a_j }$.
	\label{lem:2ndmom_torus}
\end{lemma}
\begin{proof}

The proof of this lemma is an application of the rate conservation equation to the process $(y_0^{(n)})^2 I_0^{(n)}$, where $I_0^{(n)}$ is the interference given by $I_0^{(n)} = \sum_{j \in \mathbb{Z}^d}a_jy_j^{(n)}$. For brevity of notation, we remove the superscript $n$ in the calculations. The average increase in the process $y_0^2 I_0$ due to an arrival is given by 
\begin{align}
&\mathbb{E} \left[\lambda \left( ((y_0+1)^2 (I_0+1) - y_0^2I_0) + \sum_{j \in \mathbb{Z}^d \setminus \{0\}} (y_0^2(I_0 +a_j) - y_0^2I_0 ) \right) \right] \nonumber \\
=&\mathbb{E} \left[\lambda \left( ((y_0^2 + 2y_0 +1)(I_0+1) - y_0^2I_0) + \sum_{j \in \mathbb{Z}^d \setminus \{0\}} y_0^2a_j \right) \right]  \nonumber \\
=&\mathbb{E} \left[\lambda \left( y_0^2 + 2y_0 I_0 + 2y_0 + I_0+1 +\sum_{j \in \mathbb{Z}^d \setminus \{0\}} y_0^2a_j \right) \right] \nonumber \\
=&\lambda \sum_{j \in \mathbb{Z}^d}a_j \mathbb{E}[y_0^2] + 2\lambda \mathbb{E}[y_0I_0] + 2\lambda \mu +\lambda \mu \sum_{j \in \mathbb{Z}^d}a_j  + \lambda \nonumber \\
\leq& 3 \lambda \sum_{j \in \mathbb{Z}^d}a_j \mathbb{E}[y_0^2] + 2 \lambda \mu +\lambda \mu \sum_{j \in \mathbb{Z}^d}a_j  + \lambda \label{eqn:avg_inc_2nd_moment}.
\end{align}
In the last simplification, we use the bound that $y_0y_j \leq \frac{1}{2}(y_0^2 + y_j^2)$ and the fact that $\mathbb{E}[y_0^2] = \mathbb{E}[y_j^2]$ for all $j \in B_n$. Similarly, the average decrease in the process $(y_0^{(n)})^2 I_0^{(n)}$ due to a departure is then given by
\begin{align}
&\mathbb{E} \left[ R_0 \left( y_0^2 I_0 - (y_0-1)^2(I_0-1) \right) + \sum_{j \in \mathbb{Z}^d \setminus \{0\}} R_j ( y_0^2I_0 - y_0^2(I_0 - a_j)  ) \right] \nonumber \\
=&\mathbb{E} \left[ R_0 \left( y_0^2 I_0 - (y_0^2 - 2y_0 +1)(I_0 -1) \right) + \sum_{j \in \mathbb{Z}^d \setminus \{0\}} R_j y_0^2 a_j \right] \nonumber \\
=&\mathbb{E} \left[ R_0 \left( y_0^2 + 2y_0 I_0 - 2y_0 - I_0 +1\right)  \sum_{j \in \mathbb{Z}^d \setminus \{0\}} R_j y_0^2 a_j \right] \nonumber \\
=&\mathbb{E}\left[ y_0^2 \sum_{j \in \mathbb{Z}^d}a_jR_j   \right] + 2\mathbb{E}[R_0y_0 I_0] - 2 \mathbb{E}[R_0y_0] - \mathbb{E}[R_0 I_0] + \mathbb{E}[R_0] \label{eqn:avg_dec_2nd_mom}.
\end{align}
Since the process $\{y_i^{(n)}\}_{i \in B_n}$ is stationary, the average change due to arrivals and departures is $0$, i.e., the difference between the left hand sides of  Equations (\ref{eqn:avg_inc_2nd_moment}) and (\ref{eqn:avg_dec_2nd_mom}) equals $0$ . Further, using the simplifications that $\mathbb{E}[R_0] = \lambda$, $R_0I_0 = y_0$ and $R_0 \leq 1$ almost surely, we have by taking a difference of Equations (\ref{eqn:avg_inc_2nd_moment}) and (\ref{eqn:avg_dec_2nd_mom}), that
\begin{multline*}
0 \leq 3 \lambda \sum_{j \in \mathbb{Z}^d}a_j \mathbb{E}[y_0^2] + 2 \lambda \mu +\lambda \mu \sum_{j \in \mathbb{Z}^d}a_j  + \lambda  - \\ \left( \mathbb{E}\left[ y_0^2 \sum_{j \in \mathbb{Z}^d}a_jR_j   \right] + 2\mathbb{E}[R_0y_0 I_0] - 2 \mathbb{E}[R_0y_0] - \mathbb{E}[R_0 I_0] + \mathbb{E}[R_0] \right).
\end{multline*}
The above equation can be simplified by employing the result of Lemma \ref{lem:2ndmom_auxilary} as follows:

\begin{align}
0&\stackrel{(a)}{\leq} 3 \lambda \sum_{j \in \mathbb{Z}^d}a_j \mathbb{E}[y_0^2] + 2 \lambda \mu +\lambda \mu \sum_{j \in \mathbb{Z}^d}a_j - 2c \mathbb{E}[y_0^2] -  2 \mathbb{E}[y_0^2] + 2 \mu + \lambda \mu  \sum_{j \in \mathbb{Z}^d}a_j, \label{eqn:2ndmom_internal} \nonumber \\
& \leq 2\mu( \lambda + \lambda \sum_{j \in \mathbb{Z}^d}a_j +1) - (2(1+c) - 3 \lambda \sum_{j \in \mathbb{Z}^d}a_j) \mathbb{E}[y_0^2] .\nonumber
\end{align}
The inequality $(a)$ follows from Lemma \ref{lem:2ndmom_auxilary}. By rewriting the last display, it is clear that if $\lambda < \frac{2(1+c)}{3}\frac{1}{\sum_{j \in \mathbb{Z}^d}a_j}$, then $\mathbb{E}[y_0^2] \leq \frac{ 2\mu( \lambda + \lambda \sum_{j \in \mathbb{Z}^d}a_j +1) }{2(1+c) - 3 \lambda \sum_{j \in \mathbb{Z}^d}a_j}$.

\end{proof}

The above proposition in particular gives us the following corollary.
\begin{corollary}
	For all $n > L$, if  $\lambda < \frac{2(1+c)}{3} \frac{1}{\sum_{j \in \mathbb{Z}^d} a_j}$, then $\sup_{n \geq L} \mathbb{E}[(y_0^{(n)})^2] < \infty$.
	\label{cor:2nd_mom_finite}
\end{corollary}

{\color{black} 
Based on discrete event simulations, we  conjectured in the initial version of this paper posted online, that  the second moment is uniformly bounded in $n$ for the entire stability region. This was subsequently established by \cite{stolyar_conjecture}.  See also Conjecture \ref{conjecture_2nd_mom}  in Section \ref{sec:open_questions}, and the discussions following it}.

\section{Coupling From the Past - Proofs of Theorem \ref{thm:main_stability} and Proposition \ref{prop:finite_second_moment}}
\label{sec:sids}

  The key idea is to use monotonicity and the backward coupling representation. In order to implement the proof, we need some additional notation. For any $T >0$ and $n \in \mathbb{N}$ such that $n > L$,
and any $i \in \mathbb{Z}^d$, we define the random variables $x_{i;T}(0), y^{(n)}_{i;T}(0)$ and $z^{(n)}_{i:T}(0)$.
These variables represent the number of customers in queue $i$ at time $0$ in three different dynamics which
will be coupled and driven by the same arrival and departure processes -  $(\mathcal{A}_i,\mathcal{D}_i)_{i \in \mathbb{Z}^d}$. 
In all of them, the subscript $i$ refers to queue $i$ and $T$ refers to the fact that the system started empty at time $-T$.
We now describe the three different dynamics in question:

\begin{enumerate}
	\item  $x_{i;T}(0)$ denotes the number of customers in queue $i$ at time $0$
	in the original infinite dynamics as defined in Section \ref{sec:math_framework}.
	\item$y^{(n)}_{i;T}(0)$ denotes the number of customers in queue $i$ at time $0$ for
	the dynamics restricted to the set $B_n(0)$ \emph{viewed as a torus}.
	Hence $\{y^{(n)}_{i;T}(0)\}_{i \in B_n(0)}$ is the queue length of the process studied in Section \ref{sec:torus_system}.
	\item $z^{(n)}_{i;T}(0)$ denotes the number of customers at time $0$ for the dynamics restricted set $B_n$, 
	not seen as a torus. 
	Thus for all $i \in B_n(0)^{\mathsf{c}}$, we have $z^{(n)}_{i;T}(0) = 0$, by definition. 
\end{enumerate}

The following two propositions follow immediately from monotonicity.

\begin{proposition}
	For all $T > 0$, all $n > L$, and all $i \in \mathbb{Z}^d$, we have
	$x_{i;T}(0) \geq z^{(n)}_{i;T}(0)$ and $y^{(n)}_{i;T}(0) \geq z^{(n)}_{i;T}(0)$ almost surely.
	\label{prop:inequalities}
\end{proposition}

\begin{proposition}
	For all $n > L$, almost surely, the following limits exist:
	\begin{align*}
	x_{i;\infty}(0) &:= \lim_{T \rightarrow \infty} x_{i;T}(0),\\
	y^{(n)}_{i;\infty}(0) &:= \lim_{T \rightarrow \infty} y^{(n)}_{i;T}(0),\\
	z^{(n)}_{i;\infty}(0) &:= \lim_{T \rightarrow \infty} z^{(n)}_{i;T}(0).
	\end{align*}
	\label{prop:limits_as}
\end{proposition}

Note that the distribution of the random variable $y^{(n)}_{0;\infty}$ is 
the marginal on coordinate $0$ of the probability measure $\boldsymbol{\pi}^{(n)}$,
whose existence was proved in Theorem \ref{thm:finite_PR}
We also established in Corollary \ref{cor:tightness} that the sequence of probability
measures $\{\pi^{(n)}\}_{n \in \mathbb{N}}$ is tight. Moreover, in view of Lemma \ref{lem:joint_01}, it suffices to show that  queue $0$ is stable to conclude that the entire network is stable. Hence for notational brevity, we will omit the queue and time index
by adopting the following simplified notation for the rest of this section:
$x_{T} := x_{0;T}(0)$,
$y^{(n)}_{T} := y^{(n)}_{i;T}(0)$,
$z^{(n)}_{T} := z^{(n)}_{i;T}(0)$,
where $T \in [0,\infty]$.

\begin{proposition}
	Almost surely, for every  $T \geq 0$, we have $\lim_{n \rightarrow \infty} z^{(n)}_{T} = x_T$.
	\label{prop:elem_conv}
\end{proposition}
\begin{proof}
	From Corollary \ref{cor:construct_stable}, for every finite $T$, 
	there exists a random subset $X \subset \mathbb{Z}^d$ which is almost surely finite and such that
	the value of $x_T$ can be obtained by restricting the dynamics to the set $X$ in the time interval $[-T,0]$.
	Let $N$ be any integer such that $X$ is  contained in $B_n$. Then, for all $n \geq N$, $x_T = z^{(n)}_T$.
\end{proof}

\begin{lemma}
	The sequence $z^{(n)}_{\infty}$ is non-decreasing in $n$ and almost surely converges to
	a finite integer valued random variable denoted by $z^{(\infty)}_{\infty}$.
\end{lemma}
\begin{proof}
	Note that for all finite $T$, $z^{(n)}_T$ is non-decreasing in $n$. Thus for any $n > m$,
	we have $z^{(n)}_T \geq z^{(m)}_T$, for all $T$.
	Now, taking a limit in $T$ on both sides, which we know exist from Proposition \ref{prop:limits_as},
	we see that $z^{(n)}_{\infty} \geq z^{(m)}_{\infty}$.
	This establishes the fact that $z^{(n)}_{\infty}$ is an non-decreasing sequence
	and hence the almost sure limit $\lim_{n \rightarrow \infty} z^{(n)}_{\infty}:=z^{(\infty)}_{\infty}$ exists. We now show the finiteness of $z^{(\infty)}_{\infty}$. Note that for all $n$ and  $T$,
	$z^{(n)}_T \leq y^{(n)}_T$. Now, taking a limit in $T$, we see that $z^{(n)}_{\infty} \leq y^{(n)}_{\infty}$.
	The distribution of the random variable $y^{(n)}_{\infty}$ is the probability measure $\pi^{(n)}$ on $\mathbb{N}$.
	From Corollary \ref{cor:tightness}, the sequence $\{\pi_n\}$ is tight.
	Let $\tilde{\pi}^{(n)}$, $n \in \mathbb{N}$, denote the distribution of $z^{(n)}$.
	Thus the sequence $\{\tilde{\pi}^{(n)}\}_{n \in \mathbb{N}}$ is tight as well since $z^{(n)}_{\infty} \leq y^{(n)}_{\infty}$ almost surely.
	Moreover, due to monotonicity, $z^{(n)}_{\infty}$ converges almost surely to a random variable $z^{(\infty)}_{\infty}$. 
	But since the sequence $\{\tilde{\pi}^{(n)}\}_{n \in \mathbb{N}}$ is tight,
	the limiting random variable $z^{(\infty)}_{\infty}$ is almost surely finite.
\end{proof}

\begin{lemma}
	There exists a random $N \in \mathbb{N}$, such that for all $n \geq N$, there exists a random $T_n\in \mathbb{R}^+$,
	such that for all $t \geq T_n$, $z^{(\infty)}_{\infty} = z^{(n)}_{t}$.
	\label{lem:N_TN}
\end{lemma}
\begin{proof}
	
	From the previous lemma, $z^{(n)}_{\infty}$ converges almost surely to a finite limit as $n \rightarrow \infty$.
	Since the random variables $\{z^{(n)}_{\infty}\}_{n \in \mathbb{N}}$ are integer valued, there exists a random $N$
	such that $z^{(\infty)}_{\infty} = z^{(n)}_{\infty}$, $\forall n \geq N$. 
	Now, since, for each $T$ and  $n$, $z^{(n)}_T$ is integer valued, the existence of an almost surely finite limit
	$\lim_{T \rightarrow \infty} z^{(n)}_T$ implies that there exists a $T_n$, almost surely finite and such that
	$z^{(n)}_t = z^{(n)}_{\infty}$ for all $t \geq T_n$. Now, combining the two, for every $n \geq N$, we can find a $T_n$ such that $z^{(n)}_t = z^{(n)}_{\infty}$
	for all $t \geq T_n$. Since $N$ is such that for all $n \geq N$, $z^{(n)}_{\infty}= z^{(\infty)}_{\infty}$,
	the lemma is proved. 
\end{proof}

\begin{lemma}
	Let $T_N$ be the random variable defined in Lemma \ref{lem:N_TN}.
	For all $t \geq T_N$, we have $x_t = z^{(\infty)}_{\infty}$.
	\label{lem:final_stable}
\end{lemma}
\begin{proof}
	Let $m \geq N$ and $t \geq T_N$ be arbitrary. Observe that 
	$\lim_{T \rightarrow \infty}z^{(m)}_T = z^{(m)}_{\infty} =  z^{(\infty)}_{\infty}$,
	where the second equality follows from the fact that $m \geq N$.
	From Lemma \ref{lem:N_TN}, there exists an almost surely finite $T_m$ such that for all
	$t \geq T_m$, we have $z^{(m)}_t = z^{(m)}_{\infty} = z^{(\infty)}_{\infty}$.
	Let $t^{'} \geq \max(t,T_m)$. Since $t^{'} \geq T_m$, we have $z^{(m)}_{t^{'}} = z^{(\infty)}_{\infty}$.
	Basic monotonicity gives us the following two inequalities:
	\begin{align*}
	z^{(m)}_t \geq z^{(n)}_t = z^{(\infty)}_{\infty}, \\
	z^{(m)}_t \leq z^{(m)}_{t^{'}} = z^{(\infty)}_{\infty}.
	\end{align*}
	The first inequality follows from monotonicity in space and the second from monotonicity in time.
	Thus, $z_{t}^{(m)} = z^{(\infty)}_{\infty}$. But since $m \geq N$ was arbitrary, it must be the case 
	that $x_t = \lim_{m \rightarrow \infty}z^{(m)}_t = z^{(\infty)}_{\infty}$, where the first equality
	follows from Proposition \ref{prop:elem_conv}. Thus we have established that, for all $t \geq T_N$,
	we have $x_t = z^{(\infty)}_{\infty}$ and, in particular, 
	$x_{\infty} = \lim_{t \rightarrow \infty}x_t = z^{(\infty)}_{\infty}$ is an almost surely finite random variable.
	
\end{proof}

\begin{corollary}
	If $\lambda < \frac{1}{\sum_{i \in \mathbb{Z}^d }a_i}$, then the following interchange of limits holds true:
	\begin{align*}
	\lim_{t \rightarrow \infty} \lim_{n \rightarrow \infty} z^{(n)}_{t} = \lim_{n \rightarrow \infty} \lim_{t \rightarrow \infty} z^{(n)}_{t} = x_{\infty} = z^{(\infty)}_{\infty} < \infty \text{ a.s.}
	\end{align*}
	\label{cor:lim_switch}
\end{corollary}

\begin{corollary}
	If $\lambda < \frac{1}{\sum_{j \in \mathbb{Z}^d}a_j}$, then  $\mathbb{E}[x_{\infty}] \leq \frac{\lambda a_0}{1 - \lambda (\sum_{j \in \mathbb{Z}^d}a_j)} < \infty$.
	\label{cor:mean_queue_len_ub}
\end{corollary}
\begin{proof}
	From Corollary \ref{cor:lim_switch}, $x_{\infty} = \lim_{n \rightarrow \infty} z_{\infty}^{(n)}$.
	Moreover since $z_{\infty}^{(n)}$ is non-decreasing in $n$, it follows from the monotone convergence theorem
	that $\mathbb{E}[x_{\infty}] = \lim_{n \rightarrow \infty} \mathbb{E}[z_{\infty}^{(n)}]$. As
	$z_{\infty}^{(n)} \leq y_{\infty}^{(n)}$ and   
	$\sup_{n \geq L} \mathbb{E}[y_{\infty}^{(n)}] = \frac{\lambda a_0}{1 - \lambda (\sum_{j \in \mathbb{Z}^d}a_j)}$ from Lemma \ref{lem:rcl_formula_finite},
	we get  $\mathbb{E}[x_{\infty}] \leq \frac{\lambda a_0}{1 - \lambda (\sum_{j \in \mathbb{Z}^d}a_j)} < \infty$.
\end{proof}

Now to finish the proof of Theorem \ref{thm:main_stability},  we need to conclude about the mean queue length value, which we do in the following Lemma.

\begin{lemma}
	If $\lambda < \frac{1}{\sum_{j \in \mathbb{Z}^d}a_j}$, then $	\mathbb{E}[x_{0;\infty}(0)] \geq \frac{\lambda a_0}{1 - \lambda \sum_{j \in \mathbb{Z}^d}a_j}$.
	\label{lem:lb_mean_queue_length}
\end{lemma}

\begin{proof}
	We shall choose $n > L$ arbitrary and consider the stochastic process $\tilde{\mathbb{I}}^{(n)}(t)$ defined in Proposition \ref{prop:diff_eqn_edge_effects}. We shall let $\tilde{\mathbb{I}}^{(n)}(t)$ be stationary as $\lambda < \frac{1}{\sum_{j \in \mathbb{Z}^d}a_j}$. Furthermore, notice from Theorem \ref{thm:finite_PR} that the truncated process $\{z_i^{(n)}(t)\}_{i \in B_n}$ has exponential moments. Thus, we have for all $n \in \mathbb{N}$ and all $t \in \mathbb{R}$, $\mathbb{E}[\tilde{\mathbb{I}}^{(n)}(t)] < \infty$. Thus, we can equate $\frac{d}{dt}\mathbb{E}[\tilde{\mathbb{I}}^{(n)}(t)] $ to $0$ in Proposition \ref{prop:diff_eqn_edge_effects} along with the fact $|B_n| \geq |B_n^{(I)}|$, to obtain

	\begin{align}
	0 \geq -2(1 - \lambda \sum_{j \in \mathbb{Z}^d}a_j) \sum_{i \in B_n^{(I)}} \nu_i^{(n)} + 2 \lambda |B_n^{(I)}| - 2\sum_{i \in B_n \setminus B_n^{(I)}} \nu_i^{(n)}.
	\end{align}
	Re-arranging the inequality, we see that
	\begin{align}
	\frac{1}{|B_n^{(I)}|}\sum_{i \in B_n^{(I)}}\nu_i^{(n)} \geq \frac{\lambda a_0}{1 - \lambda \sum_{j \in \mathbb{Z}^d}a_j}  - \frac{\sum_{i \in B_n \setminus B_n^{(I)}} \nu_i^{(n)}}{|B_n^{(I)}|}.
		\end{align}
		Notice that $\nu_i^{(n)} \leq \mathbb{E}[x_{0;\infty}(0)]$ which in turn thanks to Corollary \ref{cor:lim_switch} is upper bounded by $\frac{\lambda a_0}{1 - \lambda \sum_{j \in \mathbb{Z}^d}a_j}$. Further-more, from elementary counting arguments we have $\lim_{n \rightarrow \infty}{|B_n \setminus B_n^{(I)}|}{|B_n^{(I)}|} = 0$. Thus, we obtain for all $n > L$, 
		\begin{align}
		\mathbb{E}[x_{0;\infty}(0)] \geq 	\frac{1}{|B_n^{(I)}|}\sum_{i \in B_n^{(I)}}\nu_i^{(n)} \geq \frac{\lambda a_0}{1 - \lambda \sum_{j \in \mathbb{Z}^d}a_j} \left(1  - \frac{|B_n \setminus B_n^{(I)}|}{|B_n^{(I)}|}\right).
		\end{align}
		Taking a limit as $n \rightarrow \infty$ concludes the proof.
		
	\end{proof}

\subsection{Proof of Proposition \ref{prop:finite_second_moment}}

From Corollaries \ref{cor:2nd_mom_finite}, and \ref{cor:lim_switch}, the conclusion of Proposition \ref{prop:finite_second_moment} follows.

\section{Proof of Proposition \ref{prop:uniq_bdd_convergence} - Uniqueness of Stationary Solution}

\label{sec:proof_bounded_converge}

To carry out the proof, we shall employ the following rate conservation principle.

\begin{lemma}
	If $\{q_i\}_{i \in \mathbb{Z}^d}$ is a stationary solution to the dynamics satisfying $\mathbb{E}[q_0^2] < \infty$, then $\mathbb{E}[q_0] = \frac{\lambda a_0}{1 - \lambda \sum_{j \in \mathbb{Z}^d}a_j}$.
	\label{lem:infinite_system_mean_form}
\end{lemma}
\begin{proof}
	Since $\mathbb{E}[q_0^2] < \infty$, then we can apply the same proof verbatim of Proposition \ref{prop:rcl_diff_eqn}, where the stochastic process $\mathbb{I}(t) := q_0(t) \sum_{i \in \mathbb{Z}^d}a_i q_i(t)$. Then the conclusion of Proposition \ref{prop:rcl_diff_eqn} and Lemma \ref{lem:rcl_formula_finite} follow.

\end{proof}

 We  now prove Proposition \ref{prop:uniq_bdd_convergence} with the aid of certain monotonicity arguments.

 \begin{proof}
 	Let $\boldsymbol{\pi}^{'}$ be a stationary measure on $(\mathbb{Z}^d)^{\mathbb{N}}$, with finite second moment for the marginals. Let this distribution be different from 
 	 $\boldsymbol{\pi}$, the distribution corresponding to the minimal stationary solution $\{x_{i;\infty}(0)\}_{i \in \mathbb{Z}^d}$. We show
 	by elementary coupling and monotonicity arguments that $\boldsymbol{\pi} = \boldsymbol{\pi}^{'}$.
 	Let $T > 0$ be arbitrary. We couple the evolutions of the two systems
 	$\{y_{i;T}(\cdot)\}_{i \in \mathbb{Z}^d}$ and $\{x_{i;T}(\cdot)\}_{i \in \mathbb{Z}^d}$ as follows:
 	Let $\{q_i\}_{i \in \mathbb{Z}^d}$ be distributed according to $\boldsymbol{\pi}^{'}$, independently of 
 	everything else. Let $\{y_{i;T}(-T)\}_{i \in \mathbb{Z}^d}$ be such that $y_{i;T}(-T) = q_i$, for all
 	$i \in \mathbb{Z}^d$  and $\{x_{i;T}(-T)\}_{i\in \mathbb{Z}^d}$ be empty, i.e., for all
 	$i \in \mathbb{Z}^d$, we have $x_{i;T}(-T) = 0$. Thus, for all $i \in \mathbb{Z}^d$, $x_{i;T}(-T) \leq y_{i;T}(-T)$.
 	Monotonicity in Lemma \ref{lem:mono1} implies that, almost surely, for all $i \in \mathbb{Z}^d$, we have $x_{i;T}(0) \leq y_{i;T}(0)$.
 	By the definition of invariance, $\{y_{i;T}(0)\}_{i \in \mathbb{Z}^d}$ is distributed as
 	$\boldsymbol{\pi}^{'}$ with $\mathbb{E}[y_{0;T}(0)]$ given in Lemma \ref{lem:infinite_system_mean_form}.
 	From Proposition \ref{lem:final_stable}, we know that as $T \rightarrow \infty$, $x_{0;T}(0)$ converges
 	almost surely to a random variable which has a finite first moment. Furthermore, from the hypothesis of the proposition, we know that the almost sure limit  $\lim_{T \rightarrow \infty}x_{0;T}(0)$  also possesses finite second moment. Thus from the dominated convergence theorem,
 	we have that $\lim_{T \rightarrow \infty} \mathbb{E}[x_{0;T}(0)] = \mathbb{E}[x_{0;\infty}(0)]$,
 	which is also the same as given in Lemma \ref{lem:infinite_system_mean_form}. Thus $\boldsymbol{\pi}^{'}$
 	coordinate-wise dominates $\boldsymbol{\pi}$. But they have the same first moment. 
 	This implies that the two probability measures are the same.
 \end{proof}

\section{Proof of Theorem \ref{thm:initial_cond_bdd_converge}}
\label{sec:proof_thm_initial_bdd}

{\color{black} We briefly summarize the main idea in the proof, before describing the details. We consider a comparison of three systems - the original system described in Section \ref{sec:math_framework} started with the empty initial condition and one wherein all queues have  $K$ customers, and the $K$-shifted system introduced in Section \ref{sec:model_estensions_K_shifted}. We are able to compare the dynamics of the three systems using the monotonicity property of the dynamics. Furthermore, using the fact that the $K$-shifted dynamics has finite second moment for its minimal stationary solution, monotonicity implies that the original system started with $K$ customers in all queues also has a finite second moment, in the limit of large time. Notice that, as time goes to infinity, the limiting law of the number of customers in each queue when all queues were started with exactly $K$ customers is translation invariant. Thus, the uniqueness of translation invariant stationary measures having finite second moment, implies the desired result}.

\begin{proof} 
	
	To prove this statement, we consider the $K$-shifted dynamics introduced in Section \ref{sec:model_estensions_K_shifted}. We know that for $\lambda < \frac{1}{\sum_{j \in \mathbb{Z}^d} a_j}$, there exists a minimal stationary solution for this dynamics with finite mean. Furthermore, from Proposition \ref{prop:K_shifted_2ndmom}, we know that for $\lambda < \frac{2(1+c)}{3\sum_{j \in \mathbb{Z}^d} a_j}$, where the value of $c$ is given in Proposition \ref{prop:finite_second_moment}, that the second moment of the minimal stationary solution is also finite.
	\\
	
	We set some notation to illustrate the proof. For each $t \leq 0$, denote by $\{\tilde{x}_i(t)\}_{i \in \mathbb{Z}^d}$ to be the stationary solution of the $K$-shifted dynamics given that this system was started with all queues having exactly $K$ customers at time minus infinity. In other words, for $s \geq t$, let $\tilde{x}_{i;s}(-t)$ be the number of customers in the $K$-shifted dynamics in queue $i \in \mathbb{Z}^d$ at time $-t$, given that it was started with all queues having exactly $K$ customers at time $-s$. Then, from proposition \ref{prop:K_shifted_stability_result}, we know that an almost surely finite limit $\lim_{s \rightarrow \infty}\tilde{x}_{i,s}(-t) := \tilde{x}_{i,\infty}(-t)$ exists for all $i \in \mathbb{Z}^d$ and $t \in \mathbb{R}$. Now, consider three coupled systems in the backward construction procedure. The first system $\{\tilde{x}_i^{(t)}(\cdot)\}_{i \in \mathbb{Z}^d}$ is started at time $-t$ with initial condition $\tilde{x}_i^{(t)}(-t) := \tilde{x}_{i,\infty}(-t)$ for all $i \in \mathbb{Z}^d$. The second system $\{\hat{x}_{i}^{(t)}(.)\}_{i \in \mathbb{Z}^d}$ is started with all queues having exactly $K$ customers, i.e., $\hat{x}_i^{(t)}(-t) = K$ for all $i \in \mathbb{Z}^d$. The third system $\{x_i^{(t)}(\cdot)\}_{i \in \mathbb{Z}^d}$ with the empty queue condition at time $-t$, i.e., 
	$x_i^{(t)}(-t) = 0$ for all $i \in \mathbb{Z}^d$. From the monotonicity in the dynamics, we clearly, have for all $t \geq 0$ and all $u \geq -t$, the inequality $x_i^{(t)}(u) \leq \hat{x}_i^{(t)}(u) \leq \tilde{x}_{i}^{(t)}(u)$ holding almost surely. Furthermore, we know that $\lim_{t \rightarrow \infty}x_i^{(t)}(0) := x_i^{(\infty)}(0)$ exists and has mean $\frac{\lambda a_0}{1 - \lambda \sum_{j \in \mathbb{Z}^d}a_j}$ and $\mathbb{E}[(x_i^{(\infty)}(0))^2] < \infty$.
	\\
	
	The key observation now is to notice that $\tilde{x}_0^{(t)}(0)$ is \emph{monotonically non-increasing} in $t$. To prove this. consider  any $t^{'} \geq t$. From monotonicity of the shifted dynamics and the original dynamics, we have  $\tilde{x}_0^{(t^{'})}(-t) \leq \tilde{x}_{0,\infty}(-t) = \tilde{x}_0^{(t)}(-t)$. Thus, from the monotonicity of the dynamics, we have that $\tilde{x}_0^{(t^{'})}(0) \leq \tilde{x}_{0}^{(t)}(0)$, thereby concluding that $\tilde{x}_0^{(t)}(0)$ is non-increasing in $t$. This ensures the existence of the almost sure limit of $\lim_{t \rightarrow \infty} \tilde{x}_{0}^{(t)}(0) := \tilde{x}_0^{(\infty)}(0)$. Furthermore since $\tilde{x}_0^{(t)}(0)$ is monotonically non-increasing in $t$ and  $\sup_{t \geq 0}\mathbb{E}[\tilde{x}_0^{(t)}(0)] \leq \mathbb{E}[\tilde{x}_0^{(0)}(0)] \leq \frac{\lambda + K}{1 - \lambda \sum_{j \in \mathbb{Z}^d a_j}} < \infty$, we have that the limit $\mathbb{E}[\tilde{x}_0^{(\infty)}(0)] < \infty$ has finite mean. Similarly, from Proposition \ref{prop:K_shifted_2ndmom}, we know that $\sup_{t \geq 0}\mathbb{E}[(\tilde{x}_0^{(t)}(0))^2] \leq \mathbb{E}[(\tilde{x}_0^{(0)}(0))^2] < \infty$. From the definition, it is clear that $\{\tilde{x}_i^{(\infty)}(0)\}_{i \in \mathbb{Z}^d}$ is a stationary solution to our dynamics as it is shift invariant in time. Furthermore, from Proposition \ref{prop:K_shifted_2ndmom}, the second moment of  $\{\tilde{x}_i^{(\infty)}(0)\}_{i \in \mathbb{Z}^d}$  is finite. Hence from the uniqueness result in Proposition \ref{prop:finite_second_moment}, it has to be the case that $\tilde{x}_0^{(\infty)}(0) = x_0^{(\infty)}(0)$. But since for all $t \geq 0$ and all $u \geq -t$, the inequality $x_i^{(t)}(u) \leq \hat{x}_i^{(t)}(u) \leq \tilde{x}_{i}^{(t)}(u)$ holds true, it has to be the case that the almost sure limit $\lim_{t \rightarrow \infty}\hat{x}_0^{(t)}(0) := \hat{x}_0^{(\infty)}(0)$ exists and further satisfies $\mathbb{E}[\hat{x}_0^{(\infty)}(0)] = \frac{\lambda a_0}{1 - \lambda\sum_{j \in \mathbb{Z}^d}a_j}$ and $\mathbb{E}[\hat{x}_0^{(\infty)}(0)^2] < \infty$. The proof is concluded by invoking the uniqueness result in Proposition \ref{prop:uniq_bdd_convergence}.

\end{proof}

\section{Large Initial Conditions from which queue lengths diverge  - Proof of Theorem \ref{thm:bad_initial_conditions}}
\label{sec:examples_diverge}
The proof of Theorem \ref{thm:bad_initial_conditions} is split into two parts - the first part of the theorem is proved in Subsection \ref{subsec:bad_state1} and the second part in Subsection \ref{subsec:bad_state2}. We omit the proof of Proposition \ref{prop:quant_diver_full}.

\subsection{Proof of Part $1$ of Theorem \ref{thm:bad_initial_conditions}}
\label{subsec:bad_state1}

\begin{proof} 
	
	We present the proof first for the simple case of $d=1$ and the interference sequence $\{a_i\}_{i \in \mathbb{Z}}$ such that $a_i = 1$ if $|i| \leq 1$ and $a_i = 0$ otherwise. This will illustrate the key idea of freezing some queue and considering its effect at the center. We then show how to generalize this argument to arbitrary $d$ and $\{a_i\}_{i \in \mathbb{Z}^d}$.
	\\
	
	Consider $d=1$ and $\{a_i\}_{i \in \mathbb{Z}}$ such that $a_i = 1$ if $|i| \leq 1$ and $a_i = 0$ otherwise. The proof for this case relies on a definition of a `frozen boundary state' system. Roughly speaking, the $n$th frozen system for some $n \in \mathbb{N}$ refers to the dynamics of queues at $\{-n,\cdots,0,\cdots n\}$, given that  queues at the  `boundary', i.e at $n+1$ and $-n-1$ are \emph{frozen} to some value $\alpha_n$. To formalize this consider the following system. Let $n \in \mathbb{N}$ be arbitrary. Consider the $(n,\infty)$ system in which the initial condition is such that $x_{i}^{(n,\infty)} = 0$ if $i \notin \{n,-n\}$ and $x_i^{(n,\infty)}(0) = \infty$ if $i \in \{n,-n\}$. Moreover, in this system, there are no arrivals to queues $j$ such that $|j| > n$. Hence we call it the $\{x_i^{(n,\infty)}(\cdot)\}_{i \in \mathbb{Z}}$ system since the arrivals are stopped for queues $j$ such that $|j| > n$ and queues at $n$ and $-n$ have an initial value of $\infty$. Now since the queues at $-n$ and $n$ are \emph{frozen} to infinity, it is an easy consequence to see that there will be no departures in queues $-n+1$ and $n-1$ as the departure rate will be $0$, and hence the queue lengths of queues $n-1$ and $-n+1$ will go to infinity almost surely at a positive rate $\lambda$. More generally, all the queues $i \in [-n+1,n-1]$ will converge to infinity almost surely at a positive rate. Thus, for every  $n \in \mathbb{N}$, there exists a $T_n > 0$ such that  
	\begin{align}
	\mathbb{P}[\{\inf_{t \geq T_n} x_0^{(n,\infty)}(t) \} \geq n] \geq 1 - 2^{-(n+2)}.
	\label{eqn:inf_inita_cond_divergence}
	\end{align}
	Furthermore, we can assume without loss of generality that $T_n \rightarrow \infty$ as $n \rightarrow \infty$.
	\\
	
	Now, consider a second system $(n,\alpha)$ for some $0 \leq \alpha < \infty$. This convention implies that the system $\{x_i^{(n,\alpha)}(\cdot)\}_{i \in \mathbb{Z}}$ is such that all arrivals into queues $j$ such that $|j| > n$ is stopped and the initial condition is that all queues except at $n$ and $-n$ have $0$ customers and queues at $n$ and $-n$ have $\alpha$ customers. The crucial observation is that for each fixed $t \geq 0$ and $n \in \mathbb{N}$, we have 
	\begin{align*}
	\lim_{\alpha \rightarrow \infty} \sup_{0\leq s \leq t}( x_i^{(n,\infty)}(t) -  x_i^{(n,\alpha)}(t) ) = 0 \text{ a.s.}
	\end{align*}
	This observation follows from the fact that, in a finite interval of time $t$, only finitely many events occur in the queues $-n,\cdots,0,\cdots ,n$. In particular, by choosing $\alpha$ sufficiently large, we can ensure that there are no departures in queues $n-1$ and $-n+1$ in the time interval $(0,t]$. This will ensure that the dynamics of queues $-n+1 < i < n-1$ in the time interval $[0,t]$, is unchanged in the system with frozen values of $\alpha$ and $\infty$. Hence, in particular, for each $n \in \mathbb{N}$ and $u_n \geq 0$ , we can find a  $\alpha_n$ such that
	\begin{align}
	\mathbb{P}[ \inf_{T_n \leq t \leq T_n + u_n} x_0^{(n,\alpha_n)}(t) \geq n] \geq 1 - 2^{-(n+1)},
	\label{eqn:init_cond_diver_un}
	\end{align}
	where $T_n$ is given in Equation (\ref{eqn:inf_inita_cond_divergence}). Furthermore, we can choose $\alpha_n$ even larger such that $\mathbb{P}[\text{Poi}(u_n+T_n) > \alpha_n] \leq 2^{-(n+1)}$. Thus, if we then consider an initial condition where queue $i$ has $2 \alpha_i$ customers, then by monotonicity and a simple union bound, we have for all $n$ sufficiently large
	\begin{align*}
	\mathbb{P}[ \inf_{T_n \leq t \leq T_n + u_n} x_0^{(n,2\alpha_n)}(t) \geq n] \geq 1 - 2^{-(n)}.
	\end{align*}

	
	To complete the proof, let $(b_n)_{n \in \mathbb{N}}$ be an arbitrary sequence of non-negative integers such that $b_n \rightarrow \infty$. Consider an initial condition such that $x_{b_i}(0) = x_{-b_i}(0) = 2\alpha_{b_i}$. For all $i \in \mathbb{N}$ such that $i \notin \{b_n: n\in \mathbb{N}\}$, we have $x_i(0) = x_{-i}(0) = 0$. In this case, from monotonicity that for all $t$ and all $n \in \mathbb{N}$, the inequality $x_0(t) \geq x_{0}^{(b_n,\alpha_{b_n})}(t)$ almost surely. In particular, from Equation (\ref{eqn:init_cond_diver_un}), $\mathbb{P}[x_0(T_{b_n}) \geq b_n] \geq 1-2^{-(b_n+1)}$ for every $n \in \mathbb{N}$.  Thus, by a standard Borel-Cantelli argument, the queue length at $0$, $x_0(t)$ converges almost surely to $+\infty$  as $t \rightarrow \infty$, as the sequence of times $(T_{b_n})_{n \in \mathbb{N}}$ is deterministic with $\lim_{n \rightarrow \infty} T_{b_n} = \infty$.
	\\
	
	Now consider arbitrary $d$ and arbitrary irreducible $\{a_i\}_{i \in \mathbb{Z}^d}$. Let $n > L := \sup \{||i||_{\infty} \in \mathbb{Z}^d : a_i > 0 \}$, be larger than the support of the interference sequence. We modify the definition of freezing where the $(n,\infty)$ system for some $n \in \mathbb{N}$ denotes a system where the arrivals into queues $i$ such that $||i||_{\infty} > n$ is suppressed and the initial condition is such that $x_i^{(n,\infty)}(0) = \infty$, if $||i||_{\infty} = n$ and $x_i^{(n,\infty)}(0) = 0$ if $||i||_{\infty} \neq n$. For such a system, we can find a sequence of times $(T_n)_{n \in \mathbb{N}}$ that satisfy Equation (\ref{eqn:inf_inita_cond_divergence}), for all $n > L$. This follows from the irreducibility of the interference sequence $\{a_i\}_{i \in \mathbb{Z}^d}$ and the fact that it is finitely supported. The reason we may have to avoid some finite $n$ is to account for certain $a_j$ being equal to $0$, and thus an infinite wall of customers may not ``influence'' the queue at $0$. Nevertheless, we can find a $T_n$ satisfying Equation (\ref{eqn:inf_inita_cond_divergence}) for all $n > L$. Thus, we can then define a $(n,\alpha)$ system  for all $n > L$, where the initial condition is such that $x_i^{(n,\alpha)}(0) = \alpha$ if $||i||_{\infty} = n$ and $x_i^{(n,\alpha)}(0) = 0$ otherwise. From monotonicity, for every non-negative sequence of $(u_n)_{n > n_0}$, there exists a non-negative sequence $(\alpha_n)_{n > n_0}$ such that Equation (\ref{eqn:init_cond_diver_un}) is satisfied for all $n > L$. The remainder of the proof follows from the discussion of the one dimensional case.
\end{proof}

\subsection{Proof of  Part $2$ of Theorem \ref{thm:bad_initial_conditions}}
\label{subsec:bad_state2}

\begin{proof}
	We will first carefully implement the proof for the case of $d=1$ and the interference sequence being $(a_i)_{i \in \mathbb{Z}}$ with $a_i = 1$ for $|i| \leq 1$ and $a_i = 0$ otherwise. The proof builds on the ideas developed in the previous proof. The key observation we make in this proof is to notice that for every non-negative sequence $(u_n)_{n \in \mathbb{N}}$, there exists a sequence $(\alpha_n)_{n \in \mathbb{N}}$ such that if the initial condition satisfies $\min (x_n(0),x_{-n}(0)) \geq \alpha_n$, then for all sufficiently large $n$, we have
	\begin{align}
	\mathbb{P}[\inf_{T_n \leq s \leq T_n + u_n} x_0(s) \geq n] \geq 1 - 2^{-n},
	\label{eqn:diver_extension}
	\end{align}
	where $T_n$ is as defined in Equation (\ref{eqn:inf_inita_cond_divergence}). To implement the proof of this theorem, let the initial conditions be such that $\{\zeta_i\}_{i \in \mathbb{Z}}$ be an i.i.d. sequence of $\mathbb{N}$ valued random variables independent of everything else. We will divide the queues into blocks denoted by sets $\mathcal{B}_k \subset \mathbb{Z}$ recursively using indices $(m_k)_{k \in \mathbb{N}}$ such that $m_0 = 0$ and $\mathcal{B}_k$ to be of the form $\mathcal{B}_k := \{-m_{k},\cdots, -m_{k-1}-1\} \cup \{m_{k-1}+1,\cdots,m_k\}$. The sequence $(m_k)_{k \in \mathbb{N}}$ is defined by $m_k := m_{k-1} + l_k$ where $l_k$ satisfies $\mathbb{P}[\text{Geom}(1/k) \geq l_k] \leq 2^{-(k+1)}$. Let the sequence of times $(T_i)_{i \in \mathbb{N}}$ be as defined in Equation (\ref{eqn:diver_extension}). For every $n \in \mathbb{N}$ such that $n \in \mathcal{B}_k$ for some $k \in \mathbb{N}$, let $u_n := max_{v \in \mathcal{B}_k} T_v - T_n$. Let $\hat{T}_k := \max_{v \in \mathcal{B}_k} T_v$. Thus, by definition, for all $k \in \mathbb{N}$ and for all $n \in \mathcal{B}_k$, we have $T_n + u_n = \hat{T}_k$. From Equation (\ref{eqn:diver_extension}), we know that for the particular $u_n$ we have constructed, there exists a $\alpha_n$ satisfying Equation (\ref{eqn:diver_extension}) for all sufficiently large $n$. Denote by $\gamma_k := \max_{0 \leq n \leq m_k} \alpha_n$, where $m_k$ was defined above.
	\\


	Given the above set-up, we shall consider a random variable $\zeta$ on $\mathbb{N}$ such that $\mathbb{P}[\zeta \geq \gamma_k] \geq 1/\sqrt{k}$ for all sufficiently large $k$. This distribution forms the initial conditions that we consider. More precisely, the initial condition $(x_i(0))_{i \in \mathbb{Z}}$ corresponds to the i.i.d. sequence $(\zeta_i)_{i \in \mathbb{Z}}$ distributed according to $\zeta$ defined above. Define the event $\mathcal{E}_k$ as 
	\begin{align*}
	\mathcal{E}_k := \cup_{n \in \mathcal{B}_k} \{ \min(\zeta_n, \zeta_{-n}) \geq \gamma_k\}.
	\end{align*}
	From the definition of $\zeta$, it is clear that $\mathbb{P}[\mathcal{E}_k] = \mathbb{P}[\text{Geom}(1/k) \leq l_k ] \geq 1 - 2^{-k}$, where the second relation follows from the definition of $l_k$.
	\\
	
	Conditional on the event $\mathcal{E}_k$, we have that $\mathbb{P}[x_0(\hat{T}_k) \geq m_{k-1} \vert \mathcal{E}_k] \geq 1 - 2^{-m_{k-1}} \geq 1 - 2^{-k}$. This follows from the fact that at least one of the coordinates $i \in \mathcal{B}_k$ is such that $\min(x_i(0),x_{-i}(0)) \geq \gamma_k$ under the event $\mathcal{E}_k$  and hence  Equation (\ref{eqn:diver_extension}) holds. The conditioning does not affect the dynamics, as the initial conditions were chosen independent of everything else. Since $T_n + u_n = \hat{T}_k$ for all $n \in \mathcal{B}_k$, the claim follows. Thus, by un-conditioning, we get that $\mathbb{P}[x_0(\hat{T}_k) \geq m_{k-1} ]\geq (1 - 2^{-k+1} ) (1-2^{-k})  $ for all sufficiently large $k$. Now, by a standard Borel-Cantelli argument, we see that the event $\{ x_0(\hat{T}_k) < m_{k-1}\}$ occurs only finitely often. In particular, this yields that $\lim_{k \rightarrow \infty}x_0(\hat{T}_k) = \infty$ almost surely since $m_{k} \rightarrow \infty$. Since the sequence $(\hat{T}_k)_{k \in \mathbb{N}}$ is fixed and deterministic with $\lim_{k \rightarrow \infty} \hat{T}_k = \infty$, $\lim_{k \rightarrow \infty}x_0(\hat{T}_k) = \infty$ almost surely implies that $\lim_{t \rightarrow \infty} x_0(t) = \infty$ almost surely.	
	\\
	
	Now, we implement the proof in the general case.  Let the initial condition be given by the i.i.d. family $\{\xi_i\}_{i \in \mathbb{Z}^d}$ such that the initial condition $x_i(0) = \xi_i$. Furthermore, the random variable $\xi$ is such that for each $k \in \mathbb{N}$, $\mathbb{P}[\xi \geq \gamma_k] \geq k^{-\frac{1}{2d (2m_k)^{d-1}}}$, where $\gamma_k$ will be a sequence to be chosen and $m_k$ is the sequence defined in the preceding paragraph. Now, we define the blocks $\mathcal{B}_k$ as before by choosing the boundary values $(m_k)_{k \in \mathbb{N}}$ from before. Note that the sequence $m_k$ is such that $m_0 = 0$ and $m_k = m_{k-1}+l_{k-1}$ with $l_{k-1}$ satisfying $\mathbb{P}[\text{Geom}(1/k) \geq l_{k-1}] \leq 2^{-k}$ for all $k \in \mathbb{N}$. The block $\mathcal{B}_k$ is defined by $\mathcal{B}_k := \{i \in \mathbb{Z}^d: m_k < ||i||_{\infty} \leq m_{k+1} \}$. Now, for all $k \in \mathbb{N}$, let $\gamma_k := \max_{n \leq m_{k+1}} \alpha_n$, and the event $\mathcal{E}_k$ be defined as 
	\begin{align*}
	\mathcal{E}_k := \cup_{i = m_{k-1}+1}^{m_k} \cap_{ l \in \mathbb{Z}^d, ||l||_{\infty} = i}\{ \xi_{l} \geq \gamma_k  \}.
	\end{align*}
	
	From the tail probability of $\xi$, for all $k \in \mathbb{N}$, we have $\mathbb{P}[\mathcal{E}_k] \geq \mathbb{P}[\text{Geom}(1/k) \geq l_k] \geq 1 - 2^{-k}$.
	\\
	
	Since the interference sequence $\{a_i\}_{i \in \mathbb{Z}^d}$ is finitely supported and irreducible, Equation (\ref{eqn:inf_inita_cond_divergence}) will be satisfied for all $n > L$. Let $k_0 := \inf \{k: L < m_k \}$. The rest of the argument follows exactly from the one-dimensional case which we reproduce again. For $j \in \mathbb{Z}^d$ such that $||j||_{\infty} > L$ and $j \in \mathcal{B}_k$, define $u_j := \max_{v \in \mathcal{B}_k}T_v - T_j$. Similarly, define $\gamma_k := \max_{j \in \mathbb{Z}^d: ||j||_{\infty} \leq m_{k+1}} \alpha_j$. Recall that the random variable $\xi$ satisfied for all $k$, $\mathbb{P}[\xi \geq \gamma_k] \geq k^{-\frac{1}{2d (2m_k)^{d-1}}}$. It is easy to verify that for all $k \geq k_0$, we have on the event $\mathcal{E}_k$, $\mathbb{P}[x_0(\hat{T}_k) \geq m_{k-1} \vert \mathcal{E}_k] \geq 1-2^{-(k)}$. This follows, since there is at least one $i \in [m_k+1,m_{k+1}]$, such that the initial condition satisfies $x_j(0) \geq \gamma_{k+1}$ for all $j \in \mathbb{Z}^d$ such that $||j||_{\infty} = i$. By definition, $\gamma_{k+1} \geq \alpha_i$, and hence from monotonicity in the dynamics, Equation (\ref{eqn:init_cond_diver_un}) is satisfied for the specific chosen $u_n$.  
	The rest of the argument follows verbatim from the one dimensional case above, since $\mathbb{P}[\mathcal{E}_k] \geq 1 - 2^{-k}$.  The proof follows from un-conditioning and a standard application of the Borel-Cantelli lemma.
	
\end{proof}

\section{Transience - Proof of Theorem \ref{thm:transience}}
\label{sec:transience_proof}

In this section, we establish a converse to the stability result in the following theorem, which holds for the dynamics on the one dimensional grid with the interference sequence satisfying certain monotonicity property, which was specified in Definition \ref{def:monotone_ai} and reproduced here for the reader.

\begin{definition}
	The interference sequence $(a_i)_{i \in \mathbb{Z}}$ for the dynamics on the one dimensional grid is said to be \emph{monotone} if for all $ i \in \mathbb{Z}$, $a_i \geq a_{i+1}$ holds true.
\end{definition}

We now state the main result in this section regarding transience.
\begin{theorem}
	For the dynamics on the one dimensional grid with monotone $(a_i)_{i \in \mathbb{Z}}$, if $\lambda > \frac{1}{\sum_{i \in \mathbb{Z}}a_i}$, then there exists a $N_0$ large enough such that for all $N \geq N_0$, the dynamics truncated to the set $[-N,\cdots,N]$ is transient.
	\label{thm:transience_quant}
\end{theorem}

\begin{remark}
	We provide a proof using the fluid approximation approach. In the special case of $a_i = 1$ for $|i| \leq 1$ and $a_i = 0$ if $|i| > 1$, one can construct a `triangular' Lyapunov function and directly establish transience by using the results of \cite{foss_transience}. We present this alternate proof for this special case in Appendix $E$.
\end{remark}

\begin{proof}
	 Consider the case when $a_1 = 0$. Then, by monotonicity of the interference sequence, this implies that for all $i \neq 0$, $a_i = 0$, in which case the theorem is true for all $N \geq 1$, as each queue is an independent $M/M/1$ queue. Thus, we assume without loss of generality that $a_1 > 0$ in the rest of this proof. Let $N > L$ be larger than the support of $\{a_i\}_{i \in \mathbb{Z}}$ which will be chosen later. We will consider the dynamics restricted to the set $[-N,\cdots,N]$, which is a finite dimensional Markov Process denoted by the process $[Y_{-N}(\cdots),\cdots,Y_N(\cdots)]$. We will study this process in the fluid limit scaling and it is denoted by  $[y_{-N}(\cdot),\cdots,y_N(\cdot)]$. We shall study the fluid limit behavior of the Markov Chain and using the standard results of \cite{meyn_transience}, we will conclude about transience. For the truncated system $[Y_{-N}(\cdots),\cdots,Y_N(\cdots)]$, the fluid scale trajectories $[y_{-N}(\cdot),\cdots,y_N(\cdot)]$ are Lipschitz continuous functions satisfying the following system of differential equations, subject to a certain initial condition, given by
	\begin{align}
	\frac{d}{dt}y_i(t) = \begin{cases} \lambda - \frac{y_i}{\sum_{j \in \mathbb{Z} }a_{i-j} y_j  }, \text{    }& y_i(t) > 0,\\
	\lambda, \text{   }& y_i(t) = 0. \end{cases}
		\label{eqn:ode_fluid_transience}
	\end{align}
	In the equation above, $i \in \{-N,\cdots,N\}$ and for all $j$ such that $|j| > N$, and all $t \geq 0$, we have $y_j(t) = 0$ and $y_j^{'}(t) = 0$. For an initial condition $y(0):= [y_{-N}(0),\cdots,y_{N}(0)]$, we denote by the set $\mathcal{S}(y(0))$ of Lipschitz functions satisfying Equations (\ref{eqn:ode_fluid_transience}) with the initial condition specified by the vector $y(0)$. A formal derivation of this as a fluid limit ODE for the dynamics is standard (for example see \cite{sbd_tit} Theorem $6$) and we defer it to Proposition \ref{prop:transience_fluid_limit} stated and proved at the end of the section. The fluid limit Equations (\ref{eqn:ode_fluid_transience}) are also such that, if at a certain time all coordinates become equal to $0$, they stay at $0$.
	 In the rest of the proof, denote by $t_0 \in (0,\infty]$ such that for all $0 \leq t < t_0$, there exists a $i \in \{-N,\cdots,N\}$ such that $y_i(t) > 0$ and for all $t \geq t_0$ and all $i \in \{-N,\cdots,N\}$, $y_i(t) = 0$. The value of $t_0$ clearly depends on the initial value $[y_{-N}(0),\cdots,y_N(0)]$. From Proposition \ref{prop:transience_fluid_limit}, the fluid limit functions are more-over differentiable almost anywhere in $(0,t_0)$ and its derivative is given by Equations (\ref{eqn:ode_fluid_transience}). Moreover, the fluid limit functions have additional smoothness properties like the existence of higher order derivatives that we do not exploit in the present paper. Instead, we define a key notion of `unimodality' satisfied by our dynamics which is essential for establishing transience.

	\begin{definition}
		The vector $[v_{-N}, \cdots v_N] \in \mathbb{R}^{2N+1}$ is said to be \emph{ strictly unimodal} if for all $0 \leq i < j \leq N$, we have $v_i > v_j$, $v_i = v_{-i}$ and $v_i > 0$. The vector is said to be \emph{unimodal} if for all $0 \leq i < j \leq N$, we have $v_i \geq v_j$, $v_i = v_{-i}$ and $v_i \geq 0$.
		\label{defn:unimodal} 
	\end{definition}

The following  propositions characterizes the behavior of the system of Equations in (\ref{eqn:ode_fluid_transience}), which will be crucial in analyzing the system.

\begin{proposition}
	Assume the initial conditions $[y_{-N}(0), \cdots, y_N(0)]$ of the system of Equations (\ref{eqn:ode_fluid_transience}) is strictly unimodal and the interference sequence $(a_i)_{i \in \mathbb{Z}}$ is monotone. If $t_1$ is the first time any two coordinates of $[y_{0}(t), \cdots, y_N(t)]$ become equal, then $t_1 = t_0$, where $t_0$ is defined above.

	\label{prop:transience_unimodal}
\end{proposition}
\begin{proposition}
	If at time $0$, certain coordinates of $[y_{-N}(0), \cdots, y_N(0)]$ are zero, then there exists a time $\varepsilon > 0$ such that, the coordinates of $[y_{-N}(t), \cdots, y_N(t)]$ are non-zero for all $0 < t \leq \varepsilon$.
	\label{prop:transience_non_zero}
\end{proposition}

\begin{proposition}
	Let at time $0$, the vector $[y_{-N}(0),\cdots,y_N(0)]$ be unimodal. Then, if \\ $\lim_{t \rightarrow \infty}y_0(t) = \infty$, then for all $i \in \{-N,\cdots,N\}$, $\lim_{t \rightarrow \infty} y_i(t) = \infty$.
	\label{prop:transience_all_infinity}
\end{proposition}

\begin{proposition}
	If the system of Equations in \ref{eqn:ode_fluid_transience} have two initial conditions \\ $[y_{-N}(0),\cdots,y_N(0)]$ and $[\tilde{y}_{_N}(0),\cdots,\tilde{y}_N(0)]$ such that for all $i \in \{-N,\cdots,N\}$, we have $y_i(0) \geq \tilde{y}_i(0)$, then for all $ t\geq 0$ and all $i \in \{-N,\cdots,N\}$, we have $y_i(t) \geq \tilde{y}_i(t)$.
	\label{prop:monotone_transience}
\end{proposition}

Before we present the proofs of these results, we will demonstrate how to use them to conclude the proof of Theorem \ref{thm:transience_quant}. From Proposition \ref{prop:transience_non_zero}, we can assume without loss of generality that all coordinates $[y_{-N}(0), \cdots, y_N(0)]$ are non-zero. Furthermore, from monotonicity in the dynamics in Proposition \ref{prop:monotone_transience}, we can suppose that $[y_{-N}(0), \cdots, y_N(0)]$ is strictly unimodal. For if it were not strictly unimodal, then there exists $[\tilde{y}_{-N}(0), \cdots, \tilde{y}_N(0)]$ that is strictly unimodal satisfying $\tilde{y}_i(0) \leq y_i(0)$ for all $i \in [-N,\cdots,N]$. 
\\

The key quantity to study is the evolution of $\mathbb{J}(t):=\sum_{i = -N}^{N} y_i(\sum_{j = -N}^N y_{i+j}a_j )$ and concluding that $\frac{d}{dt} \mathbb{I}(t) > \epsilon$ for some $\epsilon > 0$, for all sufficiently large $ t$, provided $N$ is suitably large. To aid in understanding, we first write the equations for the simple case when $a_i = 1$ for $|i| \leq 1$ and $a_i = 0$ otherwise before attacking the general monotone $(a_i)_{i \in \mathbb{Z}}$ case. For notational brevity, we skip explicitly denoting that $y_i$ and $y_i^{'}$ are functions of time $t$.
\begin{align}
\frac{d}{dt} \sum_{i=-N}^{N} y_i(y_{i-1}+y_i + y_{i+1}) &= 2 \sum_{i=-N}^{N} y_i y_i^{'} + \sum_{i=-N}^{N} y_i y_{i-1}^{'} + y_{i-1}y_i^{'} + y_{i}y_{i+1}^{'} + y_{i+1}y_i^{'}, \nonumber \\
&= 2 \sum_{i = -N}^{N} y_iy_i^{'} + 2\sum_{i=-N}^{N}y_i^{'}(y_{i-1}+y_{i+1}), \nonumber  \\
&= 2 \sum_{i=-N}^{N} y_i^{'}(y_{i-1}+y_i + y_{i+1}), \nonumber \\
& \stackrel{(a)}{=} 2 \sum_{i = -N}^{N} \lambda (y_{i-1}+y_i +y_{i+1}) - y_i, \nonumber \\
&= 2 ( 3 \lambda - 1)\sum_{i=-N}^{N} y_i - \lambda(y_{-N} + y_N), \nonumber \\
&= 4 ( 3 \lambda - 1)\sum_{i=0}^{N} y_i -2 \lambda y_N, \nonumber \\
& \stackrel{(b)}{\geq} \left(4 (3 \lambda - 1)N - 2 \lambda \right)y_N.
\label{eqn:transience_fluid_lyapunov}
\end{align}
In the calculations above, step $(a)$ follows from substituting Equation (\ref{eqn:ode_fluid_transience}) for $y_i^{'}$ and step $(b)$ follows from unimodality which gives that for all $i \in \{0,1,\cdots,N-1\}$, $y_i \geq y_N$. 
\\

Let $N$ be sufficiently large so that the coefficient of $y_N$ in Equation (\ref{eqn:transience_fluid_lyapunov}) be strictly positive. From standard results in fluid limits of Markov Process (\cite{meyn_transience}), if we establish that $\liminf_{t \rightarrow \infty} \frac{\mathbb{J}(t)}{t} > 0$ whenever $\mathbb{J}(0) > 0$, then the underlying Markov Process is transient. Let the initial condition $[y_{-N}(0),\cdots,y_N(0)]$ be arbitrary such that $\mathbb{J}(0) > 0$. From Proposition \ref{prop:transience_non_zero}, we can assume without loss of generality that all coordinates $[y_{-N}(0), \cdots, y_N(0)]$ are non-zero. Furthermore, from monotonicity in the dynamics in Proposition \ref{prop:monotone_transience}, we can suppose that $[y_{-N}(0), \cdots, y_N(0)]$ is strictly unimodal. For if it were not strictly unimodal, then there exists $[\tilde{y}_{-N}(0), \cdots, \tilde{y}_N(0)]$ that is strictly unimodal satisfying $\tilde{y}_i(0) \leq y_i(0)$ for all $i \in [-N\cdots,N]$. We will now argue that given this arbitrary strictly unimodal initial condition, we have $\liminf_{t \rightarrow \infty} \frac{\mathbb{J}(t)}{t} > 0$.  Notice that from Equation (\ref{eqn:transience_fluid_lyapunov}),  $\frac{d}{dt}\mathbb{J}(t) \geq 0$ for all $t \geq 0$. Furthermore, since $\mathbb{J}(0) > 0$, it follows that $\inf_{t \geq 0}\mathbb{J}(t) > 0$. 
\\

We will first argue that if the initial conditions $[y_{-N}(0),\cdots,y_N(0)]$ is strictly unimodal, then $t_0 = \infty$. Recall $t_0$ is the first time when all coordinates become equal to $0$.  From the definition of $t_0$, it is clear that $\mathbb{J}(t_0) = 0$. But since $\inf_{t \geq 0} \mathbb{J}(t) > 0$, it has to be the case that $t_0 = \infty$. Thus, the conclusion of Proposition \ref{prop:transience_unimodal} holds for all $t \geq 0$ when started with an arbitrary strictly unimodal initial condition, provided $\inf_{t \geq 0}\mathbb{J}(t) > 0$, which in turn holds under the conditions in Theorem \ref{thm:transience_quant}.
\\

From Proposition \ref{prop:transience_all_infinity}, if $\lim_{t \rightarrow \infty} \mathbb{J}(t) = \infty$, then for all $k \in \{-N,\cdots,N\}$, $\lim_{t \rightarrow \infty} y_k(t) = \infty$. From Equation (\ref{eqn:transience_fluid_lyapunov}), this will yield that $\liminf_{t \rightarrow \infty} \frac{\mathbb{J}(t)}{t} > 0$, thereby from \cite{meyn_transience}, establishing the truncated Markov Chain is transient. Hence it suffices to show that $\lim_{t \rightarrow \infty} \mathbb{J}(t) = \infty$. Further from Equation (\ref{eqn:transience_fluid_lyapunov}) and the Lipschitz continuity of $t \rightarrow y_N(t)$, we can see that if $\limsup_{t \rightarrow \infty} y_N(t) > 0$, then $\lim_{t \rightarrow \infty} \mathbb{J}(t) = \infty$.  To establish that $\limsup_{t \rightarrow \infty}y_N(t) > 0$, assume on the contrary that $\lim_{t \rightarrow \infty}y_N(t) = 0$. Since $\inf_{t \geq 0} \mathbb{J}(t) > 0$, we must have some $k \in \{1,\cdots,N-1\}$ such that $\limsup_{t \rightarrow \infty} y_k(t) > 0$, but $\lim_{t \rightarrow \infty} y_{k+1}(t) = 0$. Denote by $\delta := \limsup_{t \rightarrow \infty} y_k(t) > 0$. Let $0 \leq \varepsilon < \frac{\lambda \delta}{2}$ be arbitrary. Let $t^{'} \geq 0$ be such that $y_{k}(t^{'}) \geq \delta/2$ and $y_{k+1}(t_0) \leq \varepsilon$.
There exist infinitely many choices for $t^{'}$ from our assumption that $\lim_{t \rightarrow \infty}y_{k+1}(t) = 0$ and $\limsup_{t \rightarrow \infty} y_k(t) = \delta >0$. The derivative of $y_{k+1}(\cdot)$ at time $t^{'}$ satisfies $\frac{d}{dt}y_{k+1}(t) \vert_{t = t^{'}} \geq \lambda - \frac{\varepsilon}{ \delta/2} > 0$. As $\limsup_{t \rightarrow \infty} y_{k}(t) \geq \delta/2$, and $t \rightarrow y_{k+1}(t)$ is Lipschitz continuous, we have that $\limsup_{t \rightarrow \infty} y_{k+1}(t) > \varepsilon$, contradicting the assumption that $\lim_{t \rightarrow \infty} y_{k+1}(t) = 0$. The general case can be handled similarly as follows.

\begingroup
\allowdisplaybreaks

\begin{align}
\frac{d}{dt} \sum_{i=-N}^{N} y_i(\sum_{j = -N,j\neq 0}^{N} a_{j}y_{i+j}  ) &= 2 \sum_{i=-N}^{N} y_i y_i^{'} + \sum_{i=-N}^{N} \sum_{j = -N,j\neq 0}^{N} a_j (y_i y_{i+j}^{'} + y_{i+j}y_i^{'}) \nonumber \\
&= 2 \sum_{i = -N}^{N} y_iy_i^{'} + 2\sum_{i=-N}^{N}y_i^{'} \sum_{j = -N,j\neq 0}^{N} a_j y_{i+j}, \nonumber  \\
&= 2 \sum_{i=-N}^{N} y_i^{'}(\sum_{j = -N}^{N} a_j y_{i+j}), \nonumber \\
& \stackrel{(a)}{=} 2 \sum_{i = -N}^{N} \lambda (\sum_{j = -N}^{N} a_j y_{i+j}) - y_i, \nonumber \\
& \stackrel{(b)}{\geq} 2 (\lambda \sum_{j \in \mathbb{Z}} a_j -1) \sum_{i=-N}^{N}y_i - 2\sum_{j \in \mathbb{Z}}a_j \sum_{i = N-L}^{N}y_i \nonumber \\
& \stackrel{}{=}  2 \left( (\lambda \sum_{j \in \mathbb{Z}} a_j -1) \bigg\lfloor \frac{N}{L} \bigg\rfloor - 2 \sum_{j \in \mathbb{Z}}a_j  \right) \sum_{j=N-L}^{N}y_j \nonumber.
\end{align}
\endgroup

In the calculations above, step $(a)$ follows from substituting Equation (\ref{eqn:ode_fluid_transience}) for $y_i^{'}$ and step $(b)$ follows from unimodality which gives that for all $i \in \{0,1,\cdots,N-1\}$, $y_i \geq y_{i+1}$. Since $\sum_{j=N-L}^{N}y_j(t)  \geq 0$ for all $t \geq 0$ and $\sum_{j=N-L}^{N}y_j(0) > 0$, by choosing $N$ sufficiently large, we get from similar arguments as above that the Markov process $[y_{-N}(\cdot),\cdots,y_N(\cdot)]$ is transient.
\end{proof}

We now prove Proposition \ref{prop:transience_unimodal}, which was the main structural result used in the above theorem.

\begin{proof} 
	
	We prove this by contradiction. Clearly $t_1 \leq t_0$, since at $t_0$, all coordinates of $y(t)$ are equal to $0$. Assume $t_1 < t_0$. Let $k \in \{1,\cdots,N\}$ be the largest integer $j$ such that $y_{j-1}(t_1) = y_j(t_1)$. Our first claim is  that the interference in coordinates $k$ and $k-1$ at time $t_1$ satisfies 
	\begin{equation}
	\sum_{j \in \mathbb{Z}} a_{k-j} y_{j}(t_1) < \sum_{j \in \mathbb{Z}} a_{k-j-1}y_{j}(t_1).
	\label{eqn:transience_unimod}
	\end{equation}
	
	To establish Equation (\ref{eqn:transience_unimod}), we can without loss of generality, assume that $y_{-k}(t_1) = \cdots = y_k(t_1)$. Indeed, recall that for all $j \in \mathbb{N}$, we have $a_j \geq a_{j+1}$ and $a_j = a_{-j}$. Thus for any $i \in \{-k+2,\cdots,k-2\}$, we have $a_{i-k+1}y_i\geq a_{i-k}y_i$ - in other words, by assuming $y_{-k}(t_1) = \cdots = y_k(t_1)$, we only decrease  the term on the left hand side and increase the term on the right hand side  of Equation (\ref{eqn:transience_unimod}). Now if $L \leq 2k-1$, then we have $\sum_{i \leq 0}a_i y_{k-1+i} = \sum_{i \leq 0 }a_iy_{k+i}$, while, 
	\begin{eqnarray*}
	\sum_{i > 0} a_i y_{k-1+i}(t_1) = a_1 y_k(t_1) + \sum_{i \geq 2}a_i y_{k-1+i}(t_1) > a_1y_{k+1}(t_1) +  \sum_{i \geq 2} a_i y_{k+i}(t_1) = \sum_{i > 0} a_i y_{k+i}(t_1).
	\end{eqnarray*}
The first inequality follows from the fact that $a_1 > 0$ and $y_k(t_1) > y_{k+1}(t_1)$ (by definition of $k$) and $y_{k-1+i}(t_1) \geq y_{k+i}(t_1)$, since the vector $[y_{-N}(t_1),\cdots,y_N(t_1)]$ is unimodal. This establishes Equation (\ref{eqn:transience_unimod}). If, on the other hand, $L > 2k-1$, let $L = 2k-1 + M$, where $M > 0$. For $j = k-1,k$, split the interference $\sum_{i = -L}^{L}a_i y_{i+j}$ into four terms as $\sum_{i=-L}^{-2k}$, $\sum_{i=-2k+1}^{0}$, $\sum_{i=1}^{M}$ and $\sum_{i = M+1}^{L}$. We denote the $4$ sums as $S_1(j)$, $S_2(j)$, $S_3(j)$ and $S_4(j)$, $j = k-1,k$ respectively. Since $y_{-k}(t_1) = \cdots = y_k(t_1)$, we have $S_2(k-1) = S_2(k)$. Furthermore as the interference sequence $\{a_i\}_{i \in \mathbb{Z}}$ is monotone and for all $q \in \{k,k+1,\cdots,N\}$, $y_q(t_1) > y_{q+1}(t_1)$ by definition of $k$, we have $S_4(k-1) > S_4(k)$. This follows since we have assumed that $N$ is so large that $N > L$. This ensures in particular that $S_4(k-1) > 0$. The strict inequality follows from the monotonicity of the interference sequence and the definition of $k$. We  now claim that $S_1(k-1)+ S_3(k-1) \geq S_1(k) + S_3(k)$. This claim will then conclude Equation (\ref{eqn:transience_unimod}). Indeed, from the definitions, one can write $ (S_3(k-1) - S_1(k) ) - (S_3(k) - S_1(k-1)) = \sum_{i=1}^{M} (a_i - a_{i+2k-1}) (y_{k+i-1}(t_1) - y_{k+i}(t_1)) \geq 0$ from the fact that $[y_{-N}(t_1),\cdots,y_N(t_1)]$ is unimodal and the interference sequence is monotone. Thus, Equation (\ref{eqn:transience_unimod}) holds.
\\

To conclude the proof of the proposition from Equation (\ref{eqn:transience_unimod}), we proceed as follows. We know from Proposition \ref{prop:transience_fluid_limit}, that the functions $y_k(\cdot)$ and $y_{k-1}(\cdot)$ are Lipschitz continuous functions. Hence, from Equations (\ref{eqn:ode_fluid_transience}) and (\ref{eqn:transience_unimod}), there exists an $\varepsilon > 0$ such that $y_{k-1}^{'}(t) > y_k^{'}(t)$, for all $t \in [t_1 - \varepsilon , t_1]$. This contradicts $y_{k}(t_1) = y_{k-1}(t_1)$ since $y_j(t_1) = y_{j}(t_1 - \varepsilon) + \int_{u=t_1 - \varepsilon}^{t_1} y_{j}^{'}(u) du$, for $j \in \{k,k-1\}$, with $y_{k-1}(t_1 - \varepsilon) \geq y_{k}(t_1 - \varepsilon)$ by unimodality.

\end{proof}
	
We now provide a proof of Proposition \ref{prop:transience_non_zero}.
\begin{proof}
	Let $N$ be fixed and assume $y_i(0) > 0$ for  some $i \in [-N,\cdots,N]$ and $y_j(0) = 0$ for all $j \neq i$. From Proposition \ref{prop:monotone_transience}, it suffices to consider this case due monotonicity.  From Equations (\ref{eqn:ode_fluid_transience}), it is clear that there exists an $\varepsilon > 0$ such that $y_j(t) > 0$ for all $0 < t \leq \varepsilon$ and all $j \in \{-N,\cdots,N\}$ such that $a_{i-j} > 0$. This follows from the Lipschitz continuity of the functions $y_j(\cdot)$ and the right derivative of $y_j(\cdot)$ at time $0$ is equal to $\lambda > 0$. Now, to conclude the proof, we must argue that there exists a  $\varepsilon > 0$ such that $y_j(t) > 0$ for all $0 < t \leq \varepsilon$ and \emph{for  all} $j \in \{-N,\cdots,N\}$. We do so by induction as follows. Consider a $k \in \{-N,\cdots,N\}$ such that $y_k(0) = 0$, and all $j \in \{-N,\cdots,N\}$ such that $a_{k-j} > 0$ has $y_j(0) = 0$, but there exists a $j^{'} \in \{-N,\cdots,N\}$ such that $a_{k-j^{'}} > 0$ and $a_{i - j^{'}} > 0$, where $y_i(0) > 0$. Essentially, consider a coordinate $k \in \{-N,\cdots,N\}$, which is a `second hop' neighbor of coordinate $i$. Since $j^{'} \in \{-N,\cdots,N\}$ is such that $a_{j^{'}-i} > 0$, we have $\liminf_{t \downarrow 0} \frac{y_{j^{'}}}{t} := \delta > 0$. We claim that this implies that $\liminf_{t \downarrow 0} \frac{y_{k}(t)}{t} > 0$. Assume on the contrary that $\lim_{t \downarrow 0} \frac{y_{k}(t)}{t} = 0$. This implies there exists a sequence $t_1 > t_2 >\cdots$ such that $ \lim_{r \rightarrow \infty} t_r = 0$ with the property that $ \lim_{r \rightarrow \infty} \frac{y_{k}(t_r)}{t_r} = 0$, but  $ \liminf_{r \rightarrow \infty} \frac{y_{j^{'}}(t_r)}{t_r} \geq  \delta /2$. Thus, the departure rate at queue $k$ at time $t_r$ is at most $\frac{ y_{j}(t_r) }{a_{j-j^{'}} y_{j^{'}}(t_r) }$, which tends to $0$ as $r$ goes to infinity. From Equations (\ref{eqn:ode_fluid_transience}), this implies that $\frac{d}{dt} y_k(t_r)$ is converging to $\lambda$ as $r$ goes to infinity, contradicting  $\lim_{t \downarrow 0} \frac{y_{k}}{t} = 0$. Thus, it has to be the case that $\liminf_{t \downarrow 0} \frac{y_{k}(t)}{t} > 0$. Then by induction on the number of hops of a coordinate $l$ to $i$, one can conclude the proof of the Proposition.

\end{proof}

We now prove Proposition \ref{prop:transience_all_infinity}.
\begin{proof}
	Assume $\lim_{t \rightarrow \infty} y_0(t) = \infty$ and there exists a $k \in \{1,\cdots,N\}$ such that $\liminf_{t \rightarrow \infty}y_k(t) := C < \infty$, but $\lim_{t \rightarrow \infty}y_{k-1}(t) = \infty$. We will show that this is not possible, thereby completing the proof. More precisely, we will argue that for any $C \geq 0$, it must be the case that $\liminf_{t \rightarrow \infty}y_k(t) \geq C$. From the hypothesis, since $\lim_{t \rightarrow \infty} y_{k-1}(t) = \infty$, for all $M \geq 0$, there exists $t_M \geq 0$, such that for all $ t \geq t_M$, we have $y_{k-1}(t) \geq M$. Let $C \geq 0$ be arbitrary and choose $M > \max(0, Ca_1^{-1}(\lambda^{-1}-1))$. Let $t_0 \geq t_M$ be such that $y_k(t_0) = C$. If no such $t_0$ exists, then $\liminf_{t \rightarrow \infty} y_k(t) \geq C$. If such a $t_0$ exists, then from Equation (\ref{eqn:ode_fluid_transience}), the derivative of $y_k(\cdot)$ at time $t_0$ satisfies   $\frac{d}{dt}y_k(t) \vert_{t = t_0} \geq \lambda - \frac{C}{C + a_1 M} > 0$, which is strictly positive from the choice of $M$. Thus from the Lipschitz continuity of $ y_{k}(\cdot)$, we have that $\liminf_{t \rightarrow \infty} y_k(t) \geq C$. As $C$ was arbitrary, this concludes that $\lim_{t \rightarrow \infty} y_k(t) = \infty$.
\end{proof}

We now prove Proposition \ref{prop:monotone_transience}.
\begin{proof}
	This proof essentially follows from monotonicity of the stochastic dynamics stated in Proposition \ref{lem:mono1}, and the definition of fluid-limit equation. Consider a non-negative sequence $(z_k)_{k \in \mathbb{N}}$ such that $z_k \rightarrow \infty$ and a sequence of two initial conditions $(X^{(k)}(0))_{k \in \mathbb{N}}$ and $(\widetilde{X^{(k)}}(0)_{k \in \mathbb{N}}$ such that for all $k \in \mathbb{N}$, we have $X^{(k)}(0), \widetilde{X^{(k)}}(0) \in \mathbb{N}^{2N + 1}$. Furthermore, assume the afore mentioned sequences are such that $\lim_{k \rightarrow \infty} \frac{X^{(k)}(0)}{z_k} = [y_{-N}(0),\cdots,y_{N}(0)]$ and $\lim_{k \rightarrow \infty} \frac{\widetilde{X^{(k)}}(0)}{z_k} = [\tilde{y}_{-N}(0),\cdots,\tilde{y}_{N}(0)]$, the two initial conditions under consideration in the proposition. Since, we know that for all $i \in \{-N,\cdots,N\}$, the inequality $y_i(0) \geq \tilde{y}_i(0)$ holds, we can choose the sequences $(X^{(k)}(0))_{k \in \mathbb{N}}$ and $(\widetilde{X^{(k)}}(0)_{k \in \mathbb{N}}$ such that for all $k \in \mathbb{N}$ and all coordinates $i \in \{-N,\cdots,N\}$, we have $X^{(k)}_i(0) \geq \widetilde{X^{(k)}_i}(0)$. Now having considered such a sequence, the proof will follow essentially from the monotonicity of the dynamics and Proposition \ref{prop:transience_fluid_limit}. To do so, we will set some notation. For any vector $x \in \mathbb{N}^{2N+1}$, denote by the process $\mathbf{Y}^{(x)}(\cdot)$ to be the process in consideration in Theorem \ref{thm:transience_quant} with the initial condition $\mathbf{Y}^{(x)}(0) = x$. Thus, from the monotonicity of the dynamics and the choice of the sequences, we have for all $k \in \mathbb{N}$, and all $t \geq 0$, $z_k^{-1}\mathbf{Y}^{(X^{(k)}(0))}(z_k t) \geq z_k^{-1} \mathbf{Y}^{(\widetilde{X^{(k)}}(0))}(z_k t) $ almost surely. Thus, from Proposition \ref{prop:transience_fluid_limit}, we know that as $k$ goes to infinity, the processes $z_k^{-1}\mathbf{Y}^{(X^{(k)}(0))}(z_k t) \geq z_k^{-1} \mathbf{Y}^{(\widetilde{X^{(k)}}(0))}(z_k t) $ converge in probability to the fluid limit described in Equation (\ref{eqn:ode_fluid_transience}) with the appropriate initial conditions. However, since for all $k \in \mathbb{N}$, all $i \in \{-N, \cdots, N\}$ and all $ t \geq 0$, $z_k^{-1}\mathbf{Y}^{(X^{(k)}(0))}(z_k t) \geq z_k^{-1} \mathbf{Y}^{(\widetilde{X^{(k)}}(0))}(z_k t) $ holds almost surely, even the fluid limits will satisfy this inequality, i.e., for all $t \geq 0$, $y_i(t) \geq \tilde{y}_i(t)$.

\end{proof}

We now prove the fluid limit scaling of the Markov Process and establish Equations (\ref{eqn:ode_fluid_transience}) as an appropriate law of large numbers for the original stochastic dynamics. 
\begin{proposition}
	Consider a sequence of deterministic initial conditions \\ $([Y_{-N}^{(k)}(0),\cdots,Y_{N}^{(k)}(0)])_{k \in \mathbb{N}}$ for the Markov Chain $[Y_{N}(\cdots),\cdots,Y_N(\cdot)]$ and a sequence of positive integers $(z_k)_{k \in \mathbb{N}}$ with $\lim_{k \rightarrow \infty}z_k = \infty$ such that the limit $\lim_{k \rightarrow \infty} z^{-1}_k([Y_{-N}^{(k)}(0),\cdots,Y_{N}^{(k)}(0)]) := [y_{-N}(0),\cdots,y_N(0)]:= y(0)$ exists. Then for all $T > 0$ and $\delta > 0$, we have
	\begin{align*}
	\lim_{k \rightarrow \infty} \mathbb{P} \left[  \inf_{f \in \mathcal{S}(y(0))}\sup_{t \in [0,T]} | z_k^{-1}\mathbf{Y}(z_kt) - f(t)| > \delta  \right] = 0.
	\end{align*}
	\label{prop:transience_fluid_limit}
\end{proposition}
\begin{proof}
	The proof of this is quite standard (for example \cite{sbd_tit}, Theorem $6$) and is produced here for completeness. This can be proved by contradiction. Assume that there exists an $\epsilon > 0$ and a sequence $(z_k)_{k \in \mathbb{N}}$ such that $z_k \rightarrow \infty$ such that 
	\begin{align}
	\limsup_{k \in \mathbb{N}} \mathbb{P}\left[  \inf_{f \in \mathcal{S}(y(0))}\sup_{t \in [0,T]} | z_k^{-1}\mathbf{Y}(z_kt) - f(t)| > \epsilon  \right] \geq \epsilon.
	\label{eqn:transience_fluid_contra}
	\end{align}
	Without loss of generality, we can assume the above inequality to hold for all $k \in \mathbb{N}$, by appropriately choosing the sequence $(z_k)_{k \in \mathbb{N}}$. The trajectories of the process $\mathbf{Y}^{(k)}(\cdot)$ can be written as 
	\begin{equation*}
	Y_i^{(k)}(t) = Y_i^{(k)}(0) + A_i^{(k)}(\lambda t) - D_i^{(k)} \left( \int_{s=0}^{t} \frac{Y_i^{(k)}(s)}{ \sum_{j \in \mathbb{Z}} Y_j^{(k)}(s)a_{i-j}  } ds   \right),
	\end{equation*}
	where $(A_i^{(k)}(\cdot))_{i=-N}^{N}$ and $(D_{i}^{(k)}(\cdot))_{i=-N}^{N}$ are i.i.d. unit rate Poisson Point Process on $\mathbb{R}_{+}$. One can rewrite the above equations by a change of variable as 
	\begin{equation*}
	 \frac{1}{z_k}	Y_i^{(k)}(z_k t) = \frac{1}{z_k}Y_i^{(k)}(0) + \frac{1}{z_k}A_i^{(k)}(\lambda z_k t) - \frac{1}{z_k}D_i^{(k)} \left( \int_{s=0}^{z_k t} \frac{Y_i^{(k)}(s)}{ \sum_{j \in \mathbb{Z}} Y_j^{(k)}(s)a_{i-j}  } ds   \right),
	\end{equation*}
	which by a change of variables in the departure process, can be re written as 
	\begin{equation*}
	\frac{1}{z_k}	Y_i^{(k)}(z_k t) = \frac{1}{z_K}Y_i^{(k)}(0) + \frac{1}{z_k}A_i^{(k)}(\lambda z_k t) - \frac{1}{z_k}D_i^{(k)} \left( z_k \int_{l=0}^{ t} \frac{Y_i^{(k)}(z_k l)}{ \sum_{j \in \mathbb{Z}} Y_j^{(k)}(z_k l)a_{i-j}  } dl  \right).
	\end{equation*}
	We can rewrite the above equation as a sum of a deterministic part plus an error term as 
	\begin{equation*}
		\frac{1}{z_k}	Y_i^{(k)}(z_k t) = \frac{1}{z_K}Y_i^{(k)}(0) + \frac{1}{z_k} \lambda z_k t - \frac{1}{z_k} z_k \int_{l=0}^{ t} \frac{Y_i^{(k)}(z_k l)}{ \sum_{j \in \mathbb{Z}} Y_j^{(k)}(z_k l)a_{i-j}  } dl  + \delta_k(t),
	\end{equation*}
	where the stochastic process $\delta^{(k)}_i(\cdot)$ satisfies 
	\begin{equation*}
	\sup_{t \in [0,T]} | \delta^{(k)}_i(t)| \leq \frac{1}{z_k} \sup_{t \in [0,T]}| A_i^{(k)}(z_k t) - z_k t| + \frac{1}{z_k} \sup_{t \in [0,T]} | D_i^{(k)}(z_k t) - z_kt|.
	\end{equation*}
	From standard results, (for example \cite{Poisson_LDP}), we have the following large deviations for the unit rate Poisson Process :
	\begin{lemma}
		Let $\Xi$ be a unit rate Poisson Process on $\mathbb{R}_{+}$. Then for all $T > 0$ and $\lambda > 0$, it holds that 
		\begin{align*}
		\mathbb{P} \left[ \sup_{0 \leq t \leq T} | \Xi(t) - t| \geq \lambda T \right] \leq e^{-T h(\lambda)} + e^{- T h(-\lambda) },
		\end{align*}
		where the function $h(\lambda ):= (1+\lambda) \log (1+ \lambda ) - \lambda$. In the above formula, it is understood that $h(- \lambda)	= + \infty$ if $\lambda > 1$.
	\end{lemma}
The above lemma implies that there exists a subsequence of $z_k$ denoted by $z_{k(l)}$ and another sequence $\varepsilon(l)$ with $\lim_{l \rightarrow \infty} \varepsilon(l) = 0$ such that $\sum_{l \geq 1} \mathbb{P}[ |\delta_i^{(k(l))} | \geq \varepsilon(l)] < \infty$. For example, the following particular choice of $k(l)$ can be verified to satisfy the above statement :
\begin{equation*}
\begin{cases} k(1) = 1, \\ k(l) := \min{k > k(l-1): z_k \geq l}, \text{    }&l \geq 2, \\
\epsilon_l = l^{-1/4}, \text{    }& l \geq 1. \end{cases}
\end{equation*}
Without loss of generality, we can assume that the finiteness property holds for the original sequence $k \geq 1$. Thus, by Borel-Cantelli lemma, almost surely, $\lim_{k \rightarrow \infty} \sup_{t \in [0,T]} | \delta_i^{(k)}(t)| = 0$. As there are only a finitely many coordinates $i \in \{-N,\cdots,N\}$, we have almost surely,\\ $\lim_{k \rightarrow \infty}\sup_{i \in \{-N,\cdots,N\}} \sup_{t \in [0,T]} | \delta_i^{(k)}(t)| = 0$. 
\\

Now consider the random function $\omega_k(t) := \int_{s=0}^{t} \frac{Y_i^{(k)}(z_k s)}{\sum_{j \in \mathbb{Z}} Y_j^{(k)}(z_k s)a_{i+j}} ds$, which is Lipschitz for each sample path, i.e., for all $0 \leq t \leq u$ and $k \in \mathbb{N}$, we have
\begin{align*}
\omega_i^{(k)}(u) - \omega_i^{(k)}(t) =  \int_{s=t}^{u} \frac{Y_i^{(k)}(z_k s)}{\sum_{j \in \mathbb{Z}} Y_j^{(k)}(z_k s)a_{i+j}} ds \leq (t-u) \text{  } a.s.
\end{align*}
Thus, from the Arzela-Ascoli lemma, almost surely, there exists a subsequence $k(l)$ such that $\omega_i^{(k(l))}(\cdot)$ converges uniformly on $[0,T]$ to a Lipschitz continuous function $D_i(\cdot)$. This, along with the bound on $|\delta_i^{(k)}(\cdot)|$ yields that there exists a random sub-sequence $k(l)$ such that, almost surely 
\begin{align*}
\lim_{l \rightarrow \infty}\frac{1}{z_{k(l)}} Y_i^{(k(l))}(z_k t) = y_i(0) + \lambda t - D_i(t), 
\end{align*}
where the convergence happens uniformly over $[0,T]$. Since $D_i(t)$ is Lipschitz continuous on the interval $[0,T]$, it is differentiable almost-everywhere on $[0,T]$ and its derivative, whenever it exists is defined by 
\begin{align*}
\frac{d}{dt}D_i(t) := \lim_{h \downarrow 0} \lim_{l \rightarrow \infty} \int_{s=t}^{t+h} \frac{Y_{i}^{(k(l))}(z_{k(l)} s)}{\sum_{j \in \mathbb{Z}} a_{i-j} Y_j^{(k(l))} (z_{k(l)} s)    }ds &= \lim_{h \downarrow 0}  \int_{s=t}^{t+h} \frac{y_i(s)}{\sum_{j \in \mathbb{Z}}a_{i-j}y_j(s)},\\
& = \frac{y_i(t)}{\sum_{j \in \mathbb{Z}} a_{i-j}y_j(t)}.
\end{align*}
for all $t \in [0,T]$ whenever the limit in $h$ exists. As the function $D_i(\cdot)$ is Lipschitz, the above limit in $h$ will exists for $t \in [0,T]$, Lebesgue almost-everywhere.
\\

Thus, we have shown that given a sequence $(z_k)_{k \in \mathbb{N}}$, there exists a further random subsequence such that $\frac{1}{z_{k(l)}}Y_i^{(k(l))}(z_{k(l)}(t))$ converges almost surely to a Lipschitz continuous function defined by the fluid trajectories in Equation (\ref{eqn:ode_fluid_transience}). Thus, by standard results, $\frac{1}{z_k}Y_i(z_kt)$ converges in probability to a Lipschitz continuous function, uniformly on the interval $[0,T]$, thereby contradicting Equation (\ref{eqn:transience_fluid_contra}).
\end{proof}

\section{Discussion and Conclusion}

\begin{figure}
	\centering
	\includegraphics[scale=0.4]{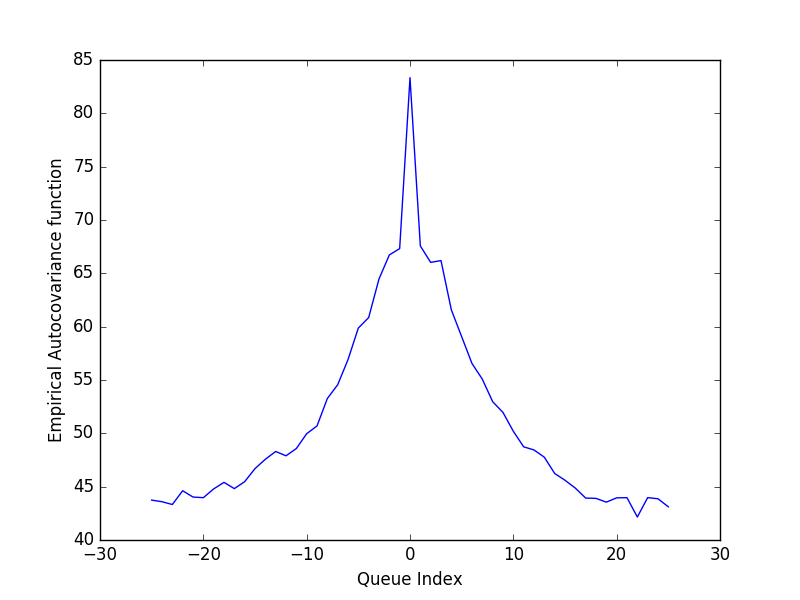}
	\caption{A  plot of the empirical covariance function of queue lengths in steady state. We consider $d=1$ and $51$ queue placed on a ring. The arrival $\lambda = 0.1419$ while $\lambda_c = 1/7$ and the interaction function is $a_i = 1$ if $|i| \leq 3$ and $0$ otherwise.}
	\label{fig:queue_corr}
\end{figure}

In this paper, we introduce a model of infinite spatial queueing system with the queues interacting with each other in a translation invariant fashion. This model is neither reversible nor admits any mean-field type approximations to analyze it. In the present paper, we analyzed this model using rate conservation and coupling arguments, which can be of interest to study other large interacting queueing systems. We establish a sufficient condition for stability which we also conjecture to be necessary. Surprisingly, we are able to compute an exact formula for the mean number of customers in steady state in any queue. Furthermore, we identify a subset of the stability region in which the stationary solution with finite second moment is unique. Interestingly however, we see that our system is sensitive to initial conditions. We construct for every $\lambda$, both a deterministic and translation invariant random initial conditions, such that the queue lengths diverges to infinity almost surely, even though the stability conditions hold. 
\\

However, our paper leaves open many intriguing questions as discussed in Section \ref{sec:open_questions}. In particular, the correlation across queues is interesting as it can be numerically simulated and is shown in Figure \ref{fig:queue_corr}. In Figure \ref{fig:queue_corr}, we are empirically estimating the function $i \rightarrow \mathbb{E}[(x_0(t) - \mu)(x_i(t) - \mu)]$, where $\mu$ is the mean queue length given in the formula in Theorem \ref{thm:main_stability}. However, we cannot simulate an infinite system and hence consider a finite system of $51$ queues placed on a ring (i.e., one dimensional torus). We use the interaction function $a_i = 1$ if $|i| \leq 3$ and $0$ otherwise. The critical arrival rate is $0.14285$ and we used a $\lambda = 0.1419$ to simulate. The mean queue length in this example is $21.18$. We estimate the function $\mathbb{E}[(x_0(t) - \mu)(x_i(t) - \mu)]$ by  collecting many independent samples approximating the steady state queue lengths $\{x_{i}^{(25)}\}_{i \in [-25,25]}$. For each collected sample, we evaluate an empirical covariance function by setting the value at $i \in [-25,25]$ to be $(x^{(25)}_{i} - \mu)(x^{(25)}_0-\mu)$, where $\mu$ is the mean queue length equal to $21.18$ in this example. We plot after averaging over many such functions computed on independent queue-length samples. From the plot, the strong positive correlations are very evident, as the function plotted is always large and positive. The figure also supports our intuition that the correlations must decay with distance as one would guess, but yields no concrete insight for the exact nature of this decay, for instance does the correlations decay polynomially or exponentially with the distance. Exploring this and other related questions in our model  is an exciting line of future work.

\section*{APPENDIX}
\begin{appendix}

		\section{Construction of the Process}
	\label{appendix_construction}
	
	In this section, we precisely describe the construction of the process alluded to in Section \ref{sec:math_framework}. To show that the dynamics is well defined, it suffices to establish that the value of the process
	at some finite time $T<\infty$  can be expressed as a deterministic function of an \emph{arbitrary}
	initial state $\{x_i(s)\}_{i \in \mathbb{Z}^d}$ for any $T>s > -\infty$ and the driving data
	$(\mathcal{A}_i,\mathcal{D}_i)_{i \in \mathbb{Z}^d}$. Roughly speaking, the queues evolve by 
	adding a customer to queue $i$ at times $A_{q}^{(i)}$ and removing a customer from a queue $i$ 
	at times  $D_{q}^{(i)}$ if $U_{q}^{(i)} \leq \frac{x_i(D_{q}^{(i)})}{\sum_{j \in \mathbb{Z}^d} a_{j-i}x_j(D_{q}^{(i)})  }$.
	In other words, we remove a customer from queue $i$ at time $D_{q}^{(i)}$ with probability
	$\frac{x_i(D_{q}^{(i)})}{\sum_{j \in \mathbb{Z}^d} a_{j-i}x_j(D_{q}^{(i)})}$ independently of everything else.
	If we had a finite collection of queues, then the above verbose description would be a sufficient
	description of the dynamics as there is a definitive `first-event' and we can sequentially order all potential events in the network in increasing order of time.
	However, the main effort in this section is to show that the dynamics described above
	in words can in fact be constructed when there are infinitely many queues.
	To show this we will need a few definitions.

	\begin{definition}
		For any $X \subset \mathbb{Z}^d$ and any $s\leq t \in \mathbb{R}$, we say that \emph{an arrival  occurs in $X$ in the interval $[s,t]$ } if $\sum_{i \in X}\mathcal{A}_i([s,t])  \geq 1$.
	\end{definition}
	
	\begin{definition}
		For any $X \subset \mathbb{Z}^d$ and any $s\leq t \in \mathbb{R}$, we say that \emph{a potential departure  occurs in $X$ in the interval $[s,t]$ } if $\sum_{i \in X} \mathcal{D}_i([s,t]) \geq 1$.
	\end{definition}
	
	\begin{definition}
		For any $X \subset \mathbb{Z}^d$ and any $s\leq t \in \mathbb{R}$, we say that \emph{a potential event occurs in $X$ in the interval $[s,t]$ } if $\sum_{i \in X}\mathcal{A}_i([s,t]) + \mathcal{D}_i([s,t]) \geq 1$ i.e., if either an arrival or a potential departure event occur.
	\end{definition}
	

	We first consider the simpler problem of constructing the dynamics if the set of queues
	were a finite set $X \subset \mathbb{Z}^d$ instead of being the entire grid. For this, let $s \leq t$ be given and $\{x_i(s)\}_{i \in X} \in \mathbb{N}^{|X|}$ be arbitrary and given. From the arrival and departure process $(\mathcal{A}_i,\mathcal{D}_i)_{i \in X}$, we can identify the set of all potential events in all of the queues in $X$ as $\{s \leq t_1 < t_2 \cdots t_n \leq t \}$. This set is finite and distinct for all finite $X \subset \mathbb{Z}^d$ and all $t \geq s$, almost surely. This is the crucial property that follows since the restricted system can be thought of being driven by a Poisson Process of intensity $(\lambda+1)|X|$, which is a finite intensity process. Thus, the atoms of this process will be distinct almost surely and will be locally finite, i.e., will contain finitely many points in any compact interval of time. Given that the set of potential events $\{s \leq t_1 < t_2 \cdots t_n \leq t \}$ is finite and distinct almost surely, we can then sequentially consider the events in chronological order of time and update the state of the queues $\{x_i(s)\}_{i \in X}$, thereby uniquely and unambiguously constructing the state $\{x_i(t)\}_{i \in X}$ at time $t$.
	\\

	To show that the dynamics is well defined, we need to show that given any initial condition $\{x_i(0)\}_{i \in \mathbb{Z}^d}$, we are able to construct the state of the system $\{x_i(t)\}_{i \in \mathbb{Z}^d}$, for all $t\geq0$. Since the dynamics is translation invariant in space and time, it suffices to show that we can unambiguously construct the state $x_0(t)$ of queue $0$, at an arbitrary time $t$. 	Before we establish this, some definitions are in order.

	\begin{definition}
		A subset $S \subseteq \mathbb{Z}^d$ is said to be  connected if for all $x,y \in S$, there exists $k \geq 1$ and $x_0:=x, x_1, \cdots, x_k := y$ such that  $x_i \in S$ for all $i \in [0,k]$, and $||x_i - x_{i-1}||_{\infty}  = 1$ for all $i \in [1,k]$.
	\end{definition}
	
	\begin{definition}
		For each $x \in \mathbb{Z}^d$ and each $L \in \mathbb{N}$, denote by $B_{\infty}(x,L)$ to be the $l_{\infty}$ ball of side-length $2 \lceil \frac{L}{2} \rceil +1$ centered around $x$. Given a set $X \subset \mathbb{Z}^d$, define its \emph{L-Thickening} to be the set $\tilde{X}_L := \cup_{z \in X} B_{\infty}(z,L)$.
	\end{definition}
	
	The following is a simple and well-known result in Boolean model percolation where the size of a connected component can be upper bounded by the total progeny of a certain branching process. We provide a short proof here for completeness.
	
	\begin{lemma}
		For every $d \geq 1$ and every $L \in \mathbb{N}$ that is finite, there exists  $p > 0$, such that if each $z \in \mathbb{Z}^d$ is declared \emph{open} with probability $p$ independent of everything else, we have almost surely, every connected subset of the random subset $\cup_{z \in \mathbb{Z}^d} \mathbf{1}(z \text{ is open } ) B_{\infty}(z,L)$ to be finite.
		\label{lem:bool_perc}
	\end{lemma}
	
	This is a classical result and much more general statements have been proven in \cite{bartek_percolation}. However, for completeness, we provide a simple proof of Lemma  \ref{lem:bool_perc} in Appendix $A.2$.
	\\

	We now use Lemma \ref{lem:bool_perc} to give a construction of our process. Given $\lambda$ and 
	$\{a_i\}_{i \in \mathbb{Z}^d}$, choose $L := \sup\{ ||i||_{\infty} : i \in \mathbb{Z}_d, a_i > 0 \}$.
	Choose  time $\hat{t} >0$ such that $\exp(-(\lambda+1)\hat{t})) \geq 1-p$, where $p$
	is defined in Lemma \ref{lem:bool_perc}. Now, we will do our construction in time steps of $\hat{t}$ units. 
	\\

	We show that we can decide on the state of queue $0$ at time $T$ in an almost surely finite number of steps.
	This will then conclude that we can do so for every queue, since the model is translation invariant.
	Divide the time interval $[0,T]$ into intervals $[0,\hat{t}], (\hat{t}, 2\hat{t}], \cdots $, i.e.,
	the interval $[0,T]$ is partitioned into finitely many blocks (i.e., $\lceil T/\hat{t} \rceil$)
	with each block being of at most $\hat{t}$. Denote by $\hat{\kappa} := T/\hat{t}$ and by
	$\kappa := \lceil T/\hat{t} \rceil$, the number of time blocks.

	\begin{definition}
		Given any $0 \leq s < t$ and any $j \in \mathbb{Z}^d$, we say $j$ is \emph{open in the time interval $[s,t]$}
		if $\mathcal{A}_j([s,t]) + \mathcal{D}_j([s,t]) \geq 1$, i.e., if there is either an arrival or a
		possible departure from queue $j$ in the time interval $[s,t]$.
	\end{definition}

	To proceed with the construction, we set some further notation. For any $r \in [1,\kappa]$,
	denote by $\mathcal{O}^{(r)}$ the set of sites of $\mathbb{Z}^d$ open in the time interval 
	$[(r-1) \hat{t}, \min( r \hat{t}, T)]$. Let $\tilde{\mathcal{O}}_{L}^{(r)}$ be its L-Thickening.
	For any $j \in \mathbb{Z}^d$, denote by $\mathcal{C}_r(j)$ the connected subset of $\tilde{\mathcal{O}}_L$ containing $j$.


	We define 
	$\mathcal{L}_{\kappa} \subseteq \mathcal{L}_{\kappa-1} \subseteq \cdots \subseteq \mathcal{L}_1 = \mathcal{L}_0 \subset \mathbb{Z}^d$,
	to be the collection of connected subsets of $\mathbb{Z}_d$ that contain the origin in a recursive fashion as follows:
	\begin{align*}
	\mathcal{L}_{\kappa} &:= \mathcal{C}_{\kappa}(0) \\
	\mathcal{L}_{i-1} &:= \cup_{j \in \mathcal{L}_i} \mathcal{C}_{i-1}(j) \text{    } \forall i \in \{\kappa, \cdots , 2\} \\
	\mathcal{L}_0 &:= \mathcal{L}_1.
	\end{align*}
	
	It is easy to check that we have $\cup_{i = 1}^{\kappa} \mathcal{L}_i$ is finite almost surely,
	since, almost surely, for all $j \in \mathbb{Z}^d$ and all $i \in \{1, \cdots , \kappa\}$, $\mathcal{C}_i(j)$ is finite.
	\\

	The following fact is now an immediate consequence of the definitions.
	\begin{proposition}
		For all $i \in \{1,\cdots ,\kappa\}$ and all $j \in \mathcal{L}_i$ and $j^{'} \in \mathcal{L}_{i}^{c}$ such that both $j$ and $j^{'}$ are open in the time interval $[(i-1)\hat{t},i\hat{t}]$, we have $d_{\infty}(j,j^{'}) > L$.
		\label{prop:non_interaction}
	\end{proposition}
	\begin{proof}
		Observe that, for any $r \in \{1,\cdots , \kappa\}$, if we have $\mathcal{C}_r(j) \neq \mathcal{C}_r(j^{'})$ and $j$ and $j^{'}$
		are open in the time interval $[(r-1) \hat{t}, \min( r \hat{t}, T)]$, then $d_{\infty}(j,j^{'}) > L$.
		This can be seen through contradiction as follows. Assume that $j$ and $j^{'}$ are open in the time interval
		$[(r-1) \hat{t}, \min( r \hat{t}, T)]$ and $d_{\infty}(j,j^{'}) \leq L-1$. This implies that there exists a
		connected path from $j$ to $j^{'}$ in the $L$-thickening of the set of open sites in the time interval
		$[(r-1) \hat{t}, \min( r \hat{t}, T)]$. This contradicts the fact that $\mathcal{C}_r(j) \neq \mathcal{C}_r(j^{'})$. 
		Since the set $\mathcal{L}_i$ is the union of a $\mathcal{C}_i(j)$ for some set of $j$, the result follows.
	\end{proof}

	The following proposition establishes that we can construct the state of queue $0$ at time $T$.

	\begin{proposition}
		For all $i \in \{0,1, \cdots, \kappa \}$, given the state of all queues in $\mathcal{L}_{i}$ at time $i\hat{t}$,
		the state of each queue in $\mathcal{L}_{i+1}$ at time $(i+1)\hat{t}$ is obtained by running the dynamics restricted to the set $X = \mathcal{L}_{i+1}$, from time  $s = i \hat{t}$ to $T = (i+1)\hat{t}$.

		\label{prop:construction}
	\end{proposition}
	
	\begin{proof}

		For any $i$, denote by $\tilde{\mathcal{L}}_i \subset \mathcal{L}_i$ the set of queues that are active
		in the time interval $[(i-1)\hat{t},i\hat{t}]$. We know from Proposition \ref{prop:non_interaction}
		that any $j \in \mathcal{L}_{i}^{c}$ that is active in the time interval $[(i-1)\hat{t},i\hat{t}]$
		is such that $d_{\infty}(j, \tilde{\mathcal{L}}_i) > L$. In words, the queues outside $\mathcal{L}_i$
		do not interact with the active queues in $\mathcal{L}_i$ during the time interval $[(i-1)\hat{t},i\hat{t}]$. 
		Thus, to know the state of queues in $\mathcal{L}_i$ in the time interval $[(i-1)\hat{t},i\hat{t}]$,
		it suffices to look at the evolution of the dynamics inside the set $\mathcal{L}_i$
		ignoring the evolutions outside this set.  Thus, the statement of the proposition follows.
	\end{proof}

	As a corollary, for any finite $T$, and any initial state $\{x_i(0)\}_{i \in \mathbb{Z}^d}$,
	we can determine $x_0(T)$ by only looking at finitely many events of the driving data
	$(\mathcal{A}_i,\mathcal{D}_i)_{i \in \mathbb{Z}^d}$.
	Since the system is translation invariant, we can do this for all $j \in \mathbb{Z}^d$. As a result of the analysis, we present the following corollary, which will be useful later on.

	\begin{corollary}
		Given any $i \in \mathbb{Z}^d$, any $s \leq T$, and any initial condition $\{x_j(s)\}_{j \in \mathbb{Z}^d}$,
		there exists a random set $X_{i;s,T} \subset \mathbb{Z}^d$ which is a deterministic function of the driving
		data $(\mathcal{A}_i,\mathcal{D}_i)_{i \in \mathbb{Z}^d}$, such that the value of $x_i(T)$ obtained by restricting the dynamics to the set $X_{i;s,T}$ in the time interval $[s,T]$.
		\label{cor:construct_stable}
	\end{corollary}
	\begin{proof}
		Setting $X_{i;s,T}$ to be equal to the set $\mathcal{L}_1$ concludes the proof. 
	\end{proof}

	\subsection{Specialization to One Dimensional Systems}

	The construction for $d=1$ is far simpler since for any finite $L$, all values of $p<1$ satisfy Lemma \ref{lem:bool_perc}. Thus given any $T$, any initial configuration $\{x_i(0)\}_{i \in \mathbb{Z}}$, given any $j \in \mathbb{Z}$, and the driving data $(\mathcal{A}_i,\mathcal{D}_i)_{i \in \mathbb{Z}}$, there almost surely exists two finite  coordinates $j_l$ and $j_r$ such that $ j_l \leq j \leq j_r$ such that there is no event in the time interval $[0,T]$ in the set of queues $\{j_l, \cdots j_l - L\}$ and in the set of queues $\{j_r, \cdots , j_r + L\}$.

	\subsection{Proof of Lemma \ref{lem:bool_perc}}

\begin{proof}
	
	Denote by the term `cluster of $j$' for some $j \in \mathbb{Z}^d$ to be the  connected component containing $j$ in the random set $\cup_{z \in \mathbb{Z}^d} \mathbf{1}(z \text{ is open } ) B_{\infty}(z,L)$. To show the lemma, it is sufficient to show that for any arbitrary $j \in \mathbb{Z}^d$, the cluster of $j$ is finite almost surely. This, along with countable additivity will establish that the cluster of all $j \in \mathbb{Z}^d$ is finite almost surely, which will prove the lemma.
	\\

	Let $j \in \mathbb{Z}^d$ be arbitrary. Consider a $p$ such that $p(L+1)^d < 1$. We can upper bound the 
	cardinality of the cluster of $j$ by the total progeny size in a certain sub-critical branching process. We sequentially construct the cluster of $j$ through the following dynamic procedure. This is a classical method to provide an upper bound on the size of the connected component of a vertex in a random graph. See for example \cite{bartek_percolation}.
	\begin{algorithm}
		\caption{Identification of the $*$-Connected Component of $j$ in $\cup_{z \in \mathbb{Z}^d} \mathbf{1}(z \text{ is open } ) B_{\infty}(z,L)$}
		\label{alg:boolean_perc}
		\begin{algorithmic}[1]
			\Procedure{Construct-Cluster}{$j$}
			\State $\mathcal{Q} \gets \{j\}$
			\State $\mathcal{U} \gets \mathbb{Z}^d \setminus \{j\}$
			\State $\mathcal{Y} \gets \emptyset$
			\Comment{We will given an upper bound to this set's cardinality}
			\While{$\mathcal{Q} \neq \emptyset$}
			\State $c \gets \text{POP}(\mathcal{Q})$ \Comment{Pop an arbitrary element from the set $\mathcal{Q}$}
			\State Set $c$ to be open with probability $p$
			\State $\mathcal{Q} \gets \mathcal{Q} \setminus \{c\}$
			\If{$c$ is declared open}
			\State $\mathcal{Q} \gets \mathcal{Q}  \cup (B_{\infty}(c,L) \cap \mathcal{U})$
			\State $\mathcal{Y} \gets \mathcal{Y} \cup B_{\infty}(c,L)$
			\State $\mathcal{U} \gets \mathcal{U} \setminus B_{\infty}(c,L)$
			\EndIf
			\State $\mathcal{Y} \gets \mathcal{Y} \cup \{c\}$
			\EndWhile
			\Return $\mathcal{Y}$
			\EndProcedure
		\end{algorithmic}
	\end{algorithm}
	
	Loosely speaking, we are performing a breadth-first exploration of the cluster of $j$.
	It is easy to see that $\mathcal{Y}$ returned by the algorithm is the cluster of $j$. 
	Moreover, it is easy to see that in each run of the while-loop, the size of $\mathcal{Y}$
	increases by at most $(L+1)^d$ with probability $p$ or remains the same with probability $(1-p)$.
	Thus, $|\mathcal{Y}|$ is stochastically dominated by the progeny size in a branching process
	where each point produces either $(L+1)^d$ off-springs with probability $p$ or no progeny with probability $1-p$.
	Since, we have that $p(L+1)^d < 1$, we know that $|\mathcal{Y}|$ is finite almost surely and
	thus, by countable additivity, all clusters of $\mathbb{Z}^d$ are finite almost surely.
	
\end{proof}

	\section{Monotonicity Proofs}
	\label{appendix_monotonicity}
	\subsection{Proof of Lemma \ref{lem:mono1}}
	\begin{proof}

		We will consider the coupling where the two systems are driven by the same arrival and departure process. Pick $\hat{t}$ as described in the construction. We will show that for all $0 \leq t\leq \hat{t}$, $x^{'}_{0}(t) \geq x_0(t)$.
		Since, the dynamics is translation invariant, this will then establish that $\{x_{i}^{'}(\hat{t})\}_{i \in \mathbb{Z}^d}$
		coordinate-wise dominates $\{x_i(\hat{t})\}_{i \in \mathbb{Z}^d}$. Since $T$ was finite, we can iterate the above argument
		in blocks of $\hat{t}$ steps and conclude the proof. 
		\\

		Denote by $\mathcal{O}$ the set of sites of $\mathbb{Z}^d$ open during the time interval $[0,\hat{t}]$ and
		by $\tilde{\mathcal{O}}_L$ its $L$-thickening. For any $j \in \mathbb{Z}^d$, denote by $\mathcal{C}(j)$ 
		the connected subset of $\tilde{\mathcal{O}}_L$ containing $j$. From the definition of $\hat{t}$, we know that,
		for all $j \in \mathbb{Z}^d$, $\mathcal{C}(j)$ is finite almost surely. Thus, we can order the events
		in $\mathcal{C}(0)$ during the time interval as $\mathcal{E}_1, \cdots, \mathcal{E}_n$ which occur at times
		$0  \leq T_1 <T_2 \cdots < T_n \leq \hat{t}$. From elementary properties $n$ is finite and $T_i < T_{i+1}$ almost surely. 
		\\

		Now, we show by induction that after the operations at all times $\{T_i\}_{i = 1}^{n}$, the ordering
		$x^{'}_j(T_i) \geq x_j(T_i)$ is maintained for all $j \in \mathcal{C}(0)$. We know that at time $0$ the 
		inequality is true. Consider the first event. If it is an arrival, then the inequality holds true  after the arrival
		since the arrivals occur in both systems. If the event $\mathcal{E}_1$ is a departure from a queue $j \in \mathcal{C}(0)$,
		then  two cases are possible. Either $x^{'}_{j}(T_{1}^{-}) \geq x_j(T_{1}^{-1}) + 1$,
		in which case the ordering $x^{'}_{j}(T_1) \geq x_j(T_1)$ is trivially true since we have at most one departure per event.
		Or, we have equality, i.e., $x^{'}_{j}(T_{1}^{-}) = x_j(T_{1}^{-})$, in which case we have the inequality of death
		probability $\frac{x_{j}^{'}(T_{1}^{-})}{\sum_{k \in \mathbb{Z}^d} a_{k-j}x_{k}^{'}(T_{1}^{-})} \leq \frac{x_{j}(T_{1}^{-})}{{\sum_{k \in \mathbb{Z}^d} a_{k-j}x_{k}(T_{1}^{-})}}$.
		We have this inequality since at time $T_{1}^{-}$, for all $k \in \mathcal{C}(0)$, we
		have $x^{'}_{k}(T_{1}^{-}) \geq x_k(T_{1}^{-})$. Since, the death probability is ordered and
		the two systems are driven by the same data, if $x^{'}_j(T_1) = x^{'}_j(T_{1}^{-}) - 1$, then it 
		must be the case that $x_j(T_1) = x_j(T_1) - 1$.
		Thus, we have that at time $T_1$, $x_{j}^{'}(T_1) \geq x_j(T_1)$ for all $j \in \mathcal{C}(0)$. Now,
		iterating the above arguments over the finitely many events, we see have the
		inequality $x_{j}^{'}(\hat{t}) \geq x_j(\hat{t})$ for all $j \in \mathcal{C}(0)$.

	\end{proof}
\subsection{Proof of Lemma \ref{lem:mono2}}
\begin{proof}
	We define two systems $\{x^{'}_j(u)\}_{j \in \mathbb{Z}^d}$ and  $\{x_j(u)\}_{j \in \mathbb{Z}^d}$ such that at time $s$,
	we have for all $j \in \mathbb{Z}^d$, $x^{'}_{j}(s) = x_j(s)$.
	We compute the  state of the queues $\{x^{'}_j(u)\}_{j \in \mathbb{Z}^d}$ for $u \geq s$ without the arrivals
	stopped in set $X$ during the time interval $[s,t]$ and evolve the system $\{x_j(u)\}_{j \in \mathbb{Z}^d}$
	with the  arrivals stopped, i.e., setting $\mathcal{A}_i([s,t]) = 0$ for all $i \in X$. Notice that at time $s$, we have $x^{'}_{k}(s) \geq x_k(s)$, for all $k \in \mathbb{Z}^d$.
	In fact, we have equality, but we represent it as an inequality to set up an induction argument.
	We first show that at time $\hat{t} + s$, we have the inequality $x^{'}_{k}(\hat{t} +s) \geq x_{k}(\hat{t} + s)$,
	for all $k \in \mathbb{Z}^d$. Now since $T$ is finite, we can iterate the above argument in blocks of time steps
	$\hat{t}$ to conclude the lemma. To prove coordinate-wise domination at time $s$ implies coordinate-wise domination
	at time $\hat{t} + s$, it suffices to show that, for any $j \in \mathbb{Z}^d$,  $x^{'}_{k}(\hat{t} + s) \geq x_k(\hat{t} + s)$,
	for all $k \in \mathcal{C}(j)$. Note that in this proof, $\mathcal{C}(j)$ is the connected component containing $j$ of
	the $L$-thickening of the set of sites open in the time-interval $[s,\hat{t}+s]$.	
	\\
	
	As above, let $j \in \mathbb{Z}^d$ be arbitrary. Denote by $\mathcal{C}(j)$ the cluster of sites that contain $j$ 
	and are open in the time interval $[s,s+\hat{t}]$. As seen before, this cluster is almost surely finite.
	Thus, there is a first event at time $T_1 \geq s$ and a last event at time $T_n \leq \hat{t}+s$ in the set $X$
	in the time interval $[s,\hat{t} +s]$. We show that the desired inequality holds through induction on the events, i.e.,
	we show that for all $i$, $\{x_{k}^{'}(T_i)\}_{k \in \mathcal{C}(j)} \geq \{x_{k}(T_i)\}_{k \in \mathcal{C}(j)}$ holds coordinate-wise. 
	\\

	If the event $\mathcal{E}_1$ is an arrival in any queue of $\mathcal{C}(j)$,
	then the inequality is trivially preserved. If the event $\mathcal{E}_1$ is a departure event from 
	queue $k \in \mathcal{C}(j)$, then there are two cases. Either $x^{'}_{k}(T_{1}^{-}) \geq x_k(T_{1}^{-}) + 1$
	or $x^{'}_{k}(T_{1}^{-}) = x_k(T_{1}^{-})$. Since, there is at most one departure per event,
	the inequality trivially holds if $x^{'}_{k}(T_{1}^{-}) \geq x_k(T_{1}^{-}) + 1$.
	If on the other hand $x^{'}_{k}(T_{1}^{-}) = x_k(T_{1}^{-})$, then the death probabilities are ordered, i.e.,
	we have $\frac{x_{k}^{'}(T_{1}^{-})}{\sum_{k \in \mathbb{Z}^d} a_{k-j}x_{k}^{'}(T_{1}^{-})} \leq \frac{x_{k}(T_{1}^{-})}{{\sum_{k \in \mathbb{Z}^d} a_{k-j}x_{k}(T_{1}^{-})}}$.
	This follows from the fact that {\color{black}at} time $T_{1}^{-}$, for all $k \in \mathcal{C}(j)$, we have $x^{'}_{k}(T_{1}^{-}) \geq x_{k}(T_{1}^{-})$.
	Thus, if there is a death in the system without stopping the arrivals, i.e., if $x_{k}^{'}(T_1) = x_{k}^{'}(T_{1}^{-}) - 1$, 
	then we will have $x_{k}(T_1) = x_{k}(T_{1}^{-}) - 1$. Hence, the inequality is preserved after the first event.
	Thus, iterating over the finitely many events, we have our desired inequality.

\end{proof}

%


\section{Proofs of the Model Extensions}

In this section, we conclude about the stability of the model extensions introduced in Section \ref{sec:model_extensions}. In particular, we will prove Propositions \ref{prop:infinite_support_existence} and \ref{prop:K_shifted_stability_result}. Proposition \ref{prop:infinite_support_existence} shows that the system that has an infinite support for the interference sequence also admits a non-trivial stability region and Proposition \ref{prop:K_shifted_stability_result} establishes the existence of a stationary solution with finite mean for the $K$-shifted system introduced in Section \ref{sec:model_extensions}.

\subsection{proof of Proposition \ref{prop:infinite_seq_stablity}}
\label{appendix:infinite_support_proof}

\begin{proof}
	The proof follows from elementary monotonicity arguments. As for the original model, we argue this using the backward construction idea. Denote by $x_{i;t}(0)$ the queue length of queue $0$ at time $0$ in the model with the infinite support interference sequence, when started empty at time $-t$. For each $K$, denote by $x_{i,t}^{(K)}(0)$ the queue length at $0$ at time $0$ in the model with $K$ truncated interference when started empty at time $-t$. The previous proposition establishes that for each $i \in \mathbb{Z}^d$ and $t \geq 0$,  $x_{i,t}(0) = \lim_{K \rightarrow \infty}x_{i,t}^{(K)}(0)$ almost surely. Moreover, from monotonicity in the dynamics, we have $\lim_{t \rightarrow \infty}x_{0,t}(0) := x_{0,\infty}(0)$ exists almost surely. Similarly for each $K$, we have the monotone limit $x_{0,\infty}^{(K)}(0) = \lim_{t \rightarrow \infty}x_{0,t}^{(K)}(0)$. From the previous results, we also have that $\mathbb{E}[x_{0,\infty}^{(K)}] = \frac{\lambda}{1 - \lambda \sum_{j \in \mathbb{Z}^d}a_j^{(K)}}$. Thus, we have $\sup_{K \in \mathbb{N}}\mathbb{E}[x_{0,\infty}^{(K)}] < \infty$ as $\lambda \sum_{j \in \mathbb{Z}^d}a_j <  1$ . This immediately yields the existence of the monotone almost sure limit $x_{0,\infty}^{(\infty)}(0) := \lim_{K \rightarrow \infty}x_{0,\infty}^{(K)}$. Furthermore, this limit satisfies $\mathbb{E}[x_{0,\infty}^{(\infty)}] = \frac{\lambda}{1 - \lambda \sum_{j \in \mathbb{Z}^d}a_j} < \infty$ as $\lim_{K \rightarrow \infty} \sum_{j \in \mathbb{Z}^d}a_{j}^{(K)} = \sum_{j \in \mathbb{Z}^d}a_j$. It remains to argue that $\lim_{t \rightarrow \infty} x_{0,t}(0) = x_{0,\infty}^{(\infty)}(0)$. 
	\\
	
	We will argue this by the simple observation that $x_{0,t}^{(K)}(0)$ is monotone in both $K$ and $t$. Thus, we have that $x_{0,\infty}^{(K)} \geq x_{0,t}^{(K)}$. Now taking a monotone limit on both sides, we obtain that $x_{0,\infty}^{(\infty)} \geq \lim_{K \rightarrow}x_{0,t}^{(K)} = x_{0,t}(0)$. Now, taking a limit on $t$, we observe that $x_{0,\infty}^{(\infty)}(0) \geq x_{0,\infty}(0)$. Now, to argue the opposite inequality, we consider $x_{0,t}(0) \geq x_{0,t}^{(K)}(0)$. Now, we take a limit on $t$ on both sides and obtain that $x_{0,\infty}(0) \geq x_{0,\infty}^{(K)}(0)$. Now taking a limit with $K$, we see that $x_{0,\infty}(0) \geq x_{0,\infty}^{(\infty)}$. Thus, it must be that $x_{0,\infty}(0) = x_{0,\infty}^{(\infty)}(0)$, which concludes the proof.
\end{proof}

\subsection{proof of Proposition \ref{prop:K_shifted_stability_result}}
\label{appendix:K_shifted_proof}

\begin{proof}
	The proof of stability follows  identical to the proof of Theorem \ref{thm:finite_PR} with the \emph{only change} being in Equation (\ref{eqn:re_norm}) to change $y_0$ to $\max(y_0,K)$, with the rest of the proof being identical. The arguments of Section \ref{sec:sids} can be repeated verbatim, as the $K$-shifted dynamics exhibits the same monotonicity when started with the initial condition of $x_i(-t) = K$ for all $i \in \mathbb{Z}^d$. Thus, the only new equation to be established for the stability program to follow is the rate conservation Equation to prove the bound in Equation (\ref{eqn:K_shifted_mean}). We use the same rate conservation equation, with the difference being in the simplification \ref{eqn:rcl_diff_eqn} with the different rate function given in Equation (\ref{eqn:K_shifted_rate_func}). The derivation is identical up to Equation (\ref{eqn:rate_conserv}) with the following steps
	\begin{align*}
	2 \lambda \mathbb{E}[\tilde{x}_0] \sum_{j \in \mathbb{Z}^d}a_j = 2 \mathbb{E}[\hat{R}_0^{(K)}(0) \sum_{j \in \mathbb{Z}^d}a_j \tilde{x}_j(0)] 
	= 2\mathbb{E}[\tilde{x}_0(0) \mathbf{1}(\tilde{x}_0(0) > K)]
	 \geq 2(\mathbb{E}[\tilde{x}_0(0)] - K)
	\end{align*}
	Thus, rearranging, yields $\mathbb{E}[\tilde{x}_0(0)] \leq \frac{\lambda + K}{1 - \lambda \sum_{j \in \mathbb{Z}^d}a_j}$.

\end{proof}

\subsection{Proof of Proposition \ref{prop:K_shifted_2ndmom}}
\label{appendix:K_shifted_proof_2ndmom}

\begin{proof}
	The proof is again identical to Propositions \ref{lem:2ndmom_auxilary} and \ref{lem:2ndmom_torus}, except with slight modifications that we indicate here.
	The program is identical to the original dynamics - namely we study the space truncated system and then write rate conservation equations. To conclude about the final infinite system, we consider a limit identical to the program carried out in Section \ref{sec:sids}. The proof of Lemma \ref{lem:2ndmom_auxilary} needs to be modified only in the conclusion as $\mathbb{E}[R_0 \sum_{i \in B_n} y_i^2 a_i] \geq 2c \mathbb{E}[R_0 y_0 \sum_{ i \in B_n} a_i y_i]$. Since $R_0(t) = \frac{y_0^{(K)}(t)}{\sum_{j \in \mathbb{Z}^d}a_j {y}_j^{(K)}(t)} \mathbf{1}({y}_0(t) > K)$. Thus, the conclusion of Proposition \ref{lem:2ndmom_auxilary} can be concluded as $\mathbb{E}[R_0 \sum_{i \in B_n} y_i^2 a_i] \geq 2c \mathbb{E}[y_0^2 \mathbf{1}(y_0 > K)]$. 
	\\
	
	Now, using this, the proof of Proposition of Lemma \ref{lem:2ndmom_torus}, the proof is identical until Equation (\ref{eqn:2ndmom_internal}) which is modified as follows.
	\begin{align*}
	0 \leq 3 \lambda \mathbb{E}[y_0^2] + C - 2c \mathbb{E}[y_0^2 \mathbf{1}(y_0 > K)] - 2 \mathbb{E}[y_0^2 \mathbf{1}(y_0 > K)],
	\end{align*}
	where $C$ absorbs all the constants independent of $y_0$ in Equation (\ref{eqn:2ndmom_internal}). Now, rearranging the above display, one arrives at the conclusion that under the conditions in Proposition \ref{prop:finite_second_moment}, the minimal stationary solution of the $K$-shifted dynamics admits a finite second moment.
\end{proof}

\section{Proof of Quantitative Statements of Bad Initial Conditions  }
\label{appendix:bad_quant_proof}

In this section, we provide a proof of Proposition \ref{prop:quant_diver_full}. In order to do so, we will need more notation and a proposition.   Let $(t_i)_{i=1}^n$ be a decreasing sequence of times defined as follows. Denote by $t_1 = n^{n2^{n+1}}$ and for all $i \in \{2,\cdots,n\}$, $t_i = \frac{\sqrt{t_{i-1}}}{ n^4}$. Thus, we see that $t_n := \frac{(t_1)^{1/2^{n-1}}}{(n^4)^{n-1}} = n^4$. For any $i \in \{1,\cdots,n\}$, let $\hat{t}_j := \sum_{i=1}^{j} t_j$. The following proposition is the main quantitative result regarding sensitivity of initial conditions.

\begin{proposition}
	Consider the system on $d=1$ and $a_i = 1$ if $|i| \leq 1$ and $a_i = 0$ otherwise. Let $\lambda \in (0, 1/3)$ and $\epsilon > 0 $ be such that $\lambda - \epsilon > 0$ and $\lambda + \epsilon < 1/3$. There exists $n_0 \in \mathbb{N}$ depending on $\lambda$ and $\epsilon$, such that for all $n \geq n_0$, if at time $0$, the initial condition satisfies, $\max(x_{n}(0),x_{-n}(0)) \geq n^{n2^{n+2}+8}$, then $\mathbb{P}[x_0(\hat{t}_n) \geq (\lambda - \epsilon)n^4] \geq 1 - n^{-4}$, where $\hat{t}_n$ is defined above.
	\label{thm:quant_diverg}
\end{proposition}

\begin{proof}
	We will implement the proof by a suitable discretization of time into slots indexed from $1$ to $n$. Slot $i$ corresponds to the time interval $[t_{i-1},t_i)$, with the convention that $t_0 = 0$. Furthermore from symmetry of the  interference sequence, we can assume that $x_n(0) = x_{-n}(0)$. We now define certain events which will help control the arrivals and departures into queues. For all $i \in \{1,2,\cdots,n-1\}$, we define the events $A_{i}^{(1)}, A_i^{(2)},D_i$ and $L_i$ as follows. The event $A_i^{(1)}$ is said to occur if the  number of arrivals into queues $n-i$ and $-n+i$ in slot $i$ is in the range $[(\lambda - \epsilon)t_i,(\lambda+\epsilon)t_i]$. The event $A_i^{(2)}$ is said to occur if the maximum number of arrivals into queues $n-i$ and $-n+i$ in time slot $i+1$ is at most $(\lambda + \epsilon)t_{i+1}$. Note that the events $A_i^{(1)}$ and $A_{i}^{(2)}$ correspond to events in different time slots. The event $D_i$ is said to occur if no departures happen in queues $n-i$ and $-n+i$ in the time slot $i$. The event $L_i$ is said to occur if no more that $(1 + \epsilon)t_{i+1}$ potential departures occur from both queues $n-i$ and $-n+i$ in time slot $i+1$. For the central queue, i.e., corresponding to index $n$, events $A_n^{(1)}$ and $D_n$ are defined as above. However, we do not define events $A_n^{(2)}$ and $L_n$, since we are only analyzing the dynamics up to time slot $n$. Similarly, for the other corner case, we only consider events $A_{0}^{(2)}$ and $L_0$ which corresponds to events in queues $n$ and $-n$ in time slot $1$.
	\\
	
	Since the system is monotone, we can further assume that arrivals into any queue $i \in \{0,1,\cdots,n\}$ only begins at the beginning of time slot $n-i$. The important fact to note is that under the event $\mathcal{E}_n : = \bigcap_{i=1}^{n-1} (A_{i}^{(1)} \cap A_{i}^{(2)} \cap D_i \cap L_i) \cap A_{0}^{(2)} \cap L_0 \cap A_n^{(1)} \cap D_n$, we have that $x_0(\hat{t}_n) \geq (\lambda -\epsilon) n^4$. This follows since in time slot $n$, there are at-least $(\lambda - \epsilon) t_n = (\lambda - \epsilon) n^4$ arrivals into queue $0$ due to event $A_n^{(1)}$ and no departures due to event $D_n$. Thus, it suffices to conclude that $\mathbb{P}[\mathcal{E}_n] \geq 1 - n^{-4}$ for all $n$ sufficiently large.
	\\
	
	For any $i \in \{0,\cdots,n\}$, the probability of events $A_{i}^{(1)}$ and $A_{i}^{(2)}$ are bounded from below by $1 - e^{-c_1 t_i}$ and $1 - e^{-c_2 t_{i+1}}$ for some $c_1,c_2 > 0$ that depend on $\lambda$ and $\epsilon$. This follows from elementary Chernoff arguments on the Poisson random variable. Observe that we do not consider $A_n^{(2)}$ and $A_0^{(1)}$ in the above computation.
	Similarly, for any $i \in \{0,1,\cdots n-1\}$, the probability of the event $L_i$ is bounded from below by $1 - e^{-c_3 t_i}$ for some $c_3 > 0$ that depends only on  $\epsilon$. Thus, it only remains to estimate the probability of events $D_i$ for $i \in \{1,\cdots,n\}$. We do this recursively, i.e., first estimate $D_1$ and then the other as follows. From the total probability law, we have that
	\begin{align*}
	\mathbb{P}[D_1^{\complement}]  \leq 1 - \mathbb{P}[D_1 \vert A_1^{1},A_0^{(2)}, L_0] \mathbb{P}[A_1^{1},A_0^{(2)}, L_0].
	\end{align*}
	Observe that $\mathbb{P}[A_1^{1},A_0^{(2)}, L_0] \geq 1- e^{-c t_1}$ for some $c > 0$ that depends only on $\lambda$ and $\epsilon$. Conditional on the events  $A_1^{1},A_0^{(2)}, L_0$, the maximum departure probability in queue $n-1$ in time slot $1$ is 
	\begin{align*}
	\frac{(\lambda + \epsilon) t_1 }{  (\lambda-\epsilon)t_1   +  (n)^{n2^{n+2}+8} - (1+\epsilon)t_1 } & \leq   \frac{(\lambda + \epsilon) ( n )^{n2^{n+1}} }{  (\lambda-\epsilon)( n )^{n2^{n+1}}   +  (n)^{n2^{n+2}+8} - (1+\epsilon)(n )^{n2^{n+1}} } \\
	& \leq \frac{(\lambda + \epsilon)}{(\lambda-\epsilon) + ( n )^{n2^{n+1} +8} - (1+ \epsilon) } \\
	&\leq \frac{ (\lambda + \epsilon)  }{  (n)^{n2^{n+1} + 8} - (1+\epsilon)   }.
	\end{align*}
	The first line follows since there can be no more than $(\lambda + \epsilon)t_1$ customers at any time in slot $1$. This is guaranteed by the event $A_1^{(1)}$. The minimum number of customers in queue $n+1$ at-any duration in time slot $1$ is $ (n)^{n2^{n+2}+8} - (1+\epsilon)t_1$ which follows from the initial condition and event $L_0$. Since the time slot is of duration $t_1$, the total number of possible departures from queue $n$ is stochastically dominated by a Poisson random variable with mean $ t_1 \frac{ (\lambda + \epsilon)  }{  (n)^{n2^{n+1} + 8} - (1+\epsilon)   } = \frac{(\lambda + \epsilon)}{n^8 - (1+\epsilon)t_1^{-1}} \stackrel{(a)}{\leq}  \frac{(\lambda + 2\epsilon)}{n^8 } \stackrel{(b)}{ \leq }\frac{1}{ n^8}$. Inequality $(a)$ holds for sufficiently large $n$ that only depends on $\lambda$ and $\epsilon$ while inequality $b$ holds  since $\lambda + 2 \epsilon < 2/3$. From the definition of the event $D_1$, it is the probability that there are no departure from queues $n-1$ and $-n+1$ in time slot $t_1$, which is at-least the probability that two independent Poisson random variables of mean $ n^{-8}$ both take value $0$. More precisely, we have
	\begin{align*}
	\mathbb{P}[D_1 \vert A_1^{1},A_0^{(2)}, L_0] \geq \mathbb{P}[X = 0]^2
	\end{align*}
	where $X$ is a Poisson random variable of mean $ n^{-8}$. From basic computations, this probability is equal to 
	\begin{align*}
	\mathbb{P}[D_1 \vert A_1^{1},A_0^{(2)}, L_0] \geq 1 - 2n^{-8},
	\end{align*}
	for all $n$ sufficiently large that do not depend on either $\lambda$ or $\epsilon$. Thus, combining the estimates, we get that 
	\begin{align*}
	\mathbb{P}[D_1^{\complement}] \leq 1 - (1 -  2n^{-8}) (1 - e^{-cn}) 
	 \leq  6 n^{-8},
	\end{align*}
	for all $n \geq n_0(\lambda,\epsilon)$ since $c$ only depends on $\lambda$ and $\epsilon$.
	\\
	
	Now, we can recursively estimate the probability of $D_i$ for $i \geq 2$. Assume we have established that for all $j \in \{1,\cdots i-1\}$, $\mathbb{P}[D_j^{\complement}] \leq (6 j^2)n^{-8}$. We now want to estimate the probability of event $D_i$. This can be written using total probability as 
	\begin{align*}
	\mathbb{P}[D_i^{\complement}] \leq 1 - \mathbb{P}[D_i \vert A_i^{1}, A_{i-1}^{(2)}, L_{i-1}, D_{i-1}] \mathbb{P}[A_i^{1}, A_{i-1}^{1}, L_{i-1}, D_{i-1}].
	\end{align*}
	Similarly as before, conditional on the events $A_i^{1}, A_{i-1}^{1}, L_{i-1}, D_{i-1}$, the maximum departure probability in slot $i$ in queue $n-i$ is at most
	\begin{align*}
	\frac{ (\lambda + \epsilon)  t_i  }{ (\lambda + \epsilon)  t_i + (\lambda - \epsilon) t_{i-1} - (1+ \epsilon) t_i  } &\leq    \frac{ (\lambda + \epsilon)  t_i  }{ (\lambda - \epsilon)  t_i + (\lambda - \epsilon) t_{i}^2  n^8 - (1+ \epsilon) t_i  } \\
	& \leq \frac{(\lambda + \epsilon) }{  (\lambda -\epsilon)t_i  n^8 - (1 - \lambda + 2 \epsilon) }.
	\end{align*}
	Thus, the total number of departures in queue $n-i$ in slot $i$ is stochastically dominated by a Poisson random variable of mean $\frac{(\lambda + \epsilon) t_i}{  (\lambda -\epsilon)t_i n^8 - (1 - \lambda+2\epsilon) } = \frac{\lambda + \epsilon}{  n^8 (\lambda - \epsilon) - (1 - \lambda+2\epsilon)t_i^{-1}   } \leq \frac{2}{n^8}$, for all $n$ sufficiently large that depends on $\lambda$ and $\epsilon$. Following similar reasoning as above, we have that $ \mathbb{P}[D_i \vert A_i^{1}, A_{i-1}^{(2)}, L_{i-1}, D_{i-1}]  \geq 1 - 5n^{-8}$, for all $i \in \{1,\cdots,n\}$ holding for all $n$ sufficiently large that do not depend on either $\lambda$ or $\epsilon$. An easy upper bound estimate for $\mathbb{P}[A_i^{1}, A_{i-1}^{1}, L_{i-1}, D_{i-1}] \leq 1 - \frac{6 (i-1)^2}{ n^8} - e^{-c t_n}$, where we use the fact that $(t_i)_{i=1}^{n}$ is decreasing in $n$. Combining these estimates, we get that     
	\begin{align*}
	\mathbb{P}[D_i^{\complement}] &\leq 1 - \left(1 - 5n^{-8}\right) \left(1 - \frac{6(i-1)^2}{ n^8} - e^{-c t_n} \right), \nonumber \\
	&\leq 5n^{-8} +  \frac{6(i-1)^2}{ n^8} + e^{-c t_n},\nonumber \\
	&=  \frac{6i^2 - 12i +11}{n^8} + e^{-ct_n}, \nonumber \\
	& = \frac{6 i^2}{n^8} + (e^{-ct_n}) + \frac{11- 12i}{n^8}, \nonumber \\
	& \leq \frac{6 i^2}{n^8} + (e^{-ct_n})  - \frac{1}{n^8} ,\nonumber \\
	&\leq  \frac{6 i^2}{n^8},
	\end{align*}
	holds for all $i \in \{2,\cdots,n\}$ when  $n \geq n_1(\lambda,\epsilon)$ is sufficiently large, since $c$ only depends on $\lambda$ and $\epsilon$. Thus, by a direct union bound, we have that $\mathbb{P}[\mathcal{E}_n] \geq 1 - \sum_{i=1}^n \frac{6 i^2}{ n^8} - n e^{-ct_n}$  which is lower bounded by $1-n^{-4}$ for all $n \geq n_2(\lambda,\epsilon)$ since $c$ only depends on $\lambda$ and $\epsilon$.

\end{proof}

With the above proposition in hand, we can conclude the proof of Proposition \ref{prop:quant_diver_full}.
\begin{proof}
	Proposition \ref{prop:quant_diver_full} gives that for all $n$ sufficiently large, if $\min(x_n(0),x_{-n}(0)) \geq (n)^{n2^{n+2}+8}$, then $\mathbb{P}[x_0(\hat{t}_n) \geq (\lambda - \epsilon) n^4  ] \geq 1 - n^{-4}$. In particular, since $\sum_{n}(j_n)^{-4} < \infty$, $\hat{t}_{j_n} \rightarrow \infty$ and $(\lambda - \epsilon)(j_n)^2 \rightarrow \infty$, the event $\{x_0(\hat{t}_{j_n}) \geq K \}$ for all $K$ does not occur for infinitely many $n$, thanks to Borel-Cantelli lemma. In particular, this yields that $\lim_{n \rightarrow \infty} x_0(\hat{t}_{j_n}) = \infty$ a.s. Since $\{\hat{t}_{j_n}\}_{n \in \mathbb{N}}$ is a deterministic sequence of times with $\hat{t}_{j_n} \rightarrow \infty$, $\lim_{n \rightarrow \infty} x_0(\hat{t}_{j_n}) = \infty$ a.s. implies that $\lim_{t \rightarrow \infty} x_0(t) = \infty$ a.s.
\end{proof}


\section{Proof of Transience through a direct Lyapunov Function construction}
\label{appendix_direct_lyapunov_transience}

\begin{proof}
	From monotonicity, it suffices to establish that the restriction of the queueing dynamics to the set $[-n,n]$
	is transient for some sufficiently large $n$. We will consider a slightly modified process 
	$\{\tilde{x}_i(t)\}_{i \in [-n,n]}$, such that $\tilde{x}_i(t) \leq x_i(t)$ for all $i \in [-n,n]$
	and all $t \geq 0$. This modified dynamics will evolve according to the same rules as our original
	dynamics, except that the arrival process is different along with a key-renormalization step. 
	We will argue that, almost surely, for all $i \in [-n,n]$, $\tilde{x}_i(t)$ converges to infinity,
	thereby implying transience for our original infinite system. To describe the modified dynamics,
	we will need several constants which we now describe.
	\\

	Given $\lambda > \frac{1}{3}$, we choose  $\delta > 0$ such that $\lambda > \frac{1}{3} + \delta$.
	We pick $h > 0$ and $\epsilon > 0$ small enough for having 
	$\mathbb{P}[\text{Exp}(\lambda) \leq h] - (1/3 + \delta)h \geq \epsilon > 0$.
	Here $\text{Exp}(\lambda)$ refers to an exponential random variable with mean $\lambda^{-1}$. 
	Such a $h$ and $\epsilon$ exist since $\lambda > 1/3 + \delta$.
	We pick $n \in \mathbb{N}$ sufficiently large so that $\frac{1}{3 - 2.01/n} < \frac{1}{3} + \delta$.
	From the arrival process $\mathcal{A}_i$, we construct a sub-process $\widehat{\mathcal{A}}_i$ 
	such that at most one arrival (the first if there is more than one) occurs in each queue and each time-slot.
	We choose $r \in \mathbb{N}$ with $r \geq 4n/\epsilon$ sufficiently large so that
	\begin{align}
	\frac{1}{r} \mathbb{E}[\min_{i \in [-n,n]} \widehat{\mathcal{A}}_i[0,rh] ] \geq (\frac{1}{3}+\delta)h + \frac{2}{3} \epsilon \label{eqn:r_arrival_bound} \\
	\frac{1}{r} \mathbb{E}[\max_{i \in [-n,n]} {\mathcal{D}}_i[[0,rh] \times [0,1/3+\delta] ]] \leq (\frac{1}{3}+\delta)h + \frac{1}{3} \epsilon \label{eqn:r_departure}
	\end{align}

	Such a choice for $r$ exists since $\frac{1}{r} \min_{i \in [-n,n]} \widehat{\mathcal{A}}_i[0,rh]$ converges
	almost surely. as $r$ goes to infinity to a value greater than or equal to $(1/3 + \delta)h + \epsilon$.
	Moreover, the convergence is also in $L^{1}$ since the family of random variables
	$\{\frac{1}{r} \min_{i \in [-n,n]} \widehat{\mathcal{A}}_i[0,rh]\}_{r \geq 1}$ are bounded from above and below by 
	$1$ and $0$, respectively. Similarly, the strong law of large numbers entails that 
	$\frac{1}{r} \max_{i \in [-n,n]} {\mathcal{D}}_i[0,rh]$ converges almost surely. to $(1/3+\delta)h$.
	This convergence also occurs in $L^{1}$ since the family of random variables $\{\frac{1}{r}  {\mathcal{D}}_i[0,rh]\}_{i \in [-n,n],r \geq 1}$ 
	is uniformly integrable (UI). Since $n$ is fixed, this immediately implies that
	$\{\frac{1}{r} \min_{i \in [-n,n]} {\mathcal{D}}_i[0,rh]\}_{i \in [-n,n],r \geq 1}$ is UI. 
	The UI of $\{\frac{1}{r}  {\mathcal{D}}_i[0,rh]\}_{i \in [-n,n],r \geq 1}$
	follows since the second moment \\ $\sup_{r \geq 1,i \in [-n,n]}\mathbb{E}\left[ \left(\frac{1}{r}  {\mathcal{D}}_i[0,rh]\right)^2 \right] < \infty$
	is uniformly bounded  in $r$ and $i$. 
	\\

	The dynamics of our lower-bound is identical to that of the original system, except that it is driven by
	the modified arrival process $\widehat{\mathcal{A}}$
	along with a modification at the end of every $r$ time slots.
	After every $r$ time slots, we {reduce} queue lengths to the largest possible
	\emph{envelope function}, which serves as a Lyapunov function. Denote by $L_n(\cdot):
	\mathbb{N} \rightarrow \mathbb{N}^{(2n+1)} : = [x,2x,\cdots, (n-1)x,nx,(n-1)x,\cdots,x] $, 
	the linear triangle function.  At the end of every $r$ slots, we will reduce the queue lengths to 
	$L_n(x)$ for the largest possible $x$.
	In other words, if the  queue length at the end of $r$ slots is given by the vector
	$Q:= [q_{-n},\cdots,q_0,\cdots,q_n]$, then we reduce the queue lengths to $L_n(x)$ for
	the largest possible $x \in \mathbb{N}$ such that no coordinate of $Q - L_n(x)$ is strictly negative.
	From the monotonicity of the dynamics, for all $i \in [-n,n]$ and $t \geq 0$, we have $\tilde{x}_i(t) \leq x_i(t)$.
	\\

	For $m \in \mathbb{Z}$, let $Y_m := \tilde{x}_{n}(mrh)$, i.e., $Y_m$ is the length of queue $n$
	\emph{just after} the reduction of customers to the level function $L_n(\cdot)$. 
	Since the dynamics is driven by Poisson point processes, it is easy to verify that $Y_m$ is
	a Markov process with respect to the filtration $\mathcal{F}_m^{Y} := \sigma(\cdots, Y_0,\cdots, Y_m)$.
	The following is a key structural lemma which will enable us to conclude about the proof of Theorem \ref{thm:transience}.
	
	\begin{lemma}
		There exists a $x_0 \in \mathbb{N}$ and $\gamma > 0$ such that for all $x \geq x_0$, we have $\mathbb{E}[Y_{1} - Y_0|Y_0 = x] \geq \gamma$. 
		\label{lem:transience_drift}
	\end{lemma}
	
	\begin{proof}
		To prove this lemma, we need to define a constant $C$ which satisfies
		$r(2n+1) \mathbb{P}[\text{Poi}((\frac{1}{3} + \delta)rh) \geq Crh \log(x)] \leq x^{-2}$ for all sufficiently large $x$.
		We call a collection of $r$ time-slots $[prh,(p+1)rh]$, for any $p \in \mathbb{Z}$ `x-good', if  $\max_{i \in [-n,n]} \mathcal{D}_i[[prh,(p+1)rh] \times [0,1/3+\delta]] \leq Cr h \log(x)$.
		In other words a collection of $r$ time-slots is `x-good' if the total number of potential departures
		in every queue that has its mark less than or equal to $1/3 + \delta$ does not exceed $Crh \log(x)$. 
		We can then use this definition to break up the drift into two terms as follows -

		\begin{align*}
		\mathbb{E}[Y_1 - Y_0|Y_0 = x]  = \mathbb{E}[(Y_1 - Y_0) \mathbf{1}_{ [0,rh] \text{ is }  x\text{-Good}  } \vert Y_0 = x] +  \mathbb{E}[(Y_1 - Y_0) (1-\mathbf{1}_{ [0,rh] \text{ is }  x\text{-Good}  }) \vert Y_0 = x].
		\end{align*}
		We can bound the drift trivially as follows - 
		\begin{align*}
		\mathbb{E}[Y_1 - Y_0|Y_0 = x] \geq \mathbb{E}[(Y_1 - Y_0) \mathbf{1}_{ [0,rh] \text{ is }  x\text{-Good}  } \vert Y_0 = x] -   x \mathbb{P}[[0,rh] \text{ is not } x\text{-Good}].
		\end{align*}
		Now using the definition of $x$-Good, we get
		\begin{align}
		\mathbb{E}[Y_1 - Y_0|Y_0 = x]	\geq   \mathbb{E}[(Y_1 - Y_0) \mathbf{1}_{ [0,rh] \text{ is }  x\text{-Good}  } \vert Y_0 = x] - \frac{1}{x}.
		\label{eqn:drift_ptve}
		\end{align}
		
		Now let $x$ be sufficiently large so that 
		\begin{align}
		\frac{1}{3 - (\frac{2}{n} + \frac{3Crh \log(x)}{x} + \frac{3r}{x})} < \frac{1}{3} + \delta.
		\label{eqn:large_x_formula}
		\end{align}
		Such a choice of $x$ is possible since $n$ is so large that $\frac{1}{3 - 2.01/n} < \frac{1}{3} + \delta$.
		The crucial fact is that if $x$ satisfies Equation (\ref{eqn:large_x_formula}), then, on the event that a slot is x-Good,
		the number of departures from any queue in the collection of $r$-slots is at most the number of
		potential departures that have a mark less than or equal to $1/3 + \delta$.
		Conditioned on the collection of $r$ time-slots being x-Good, the maximum departure probability
		in any queue in that time period is 
		\begin{align}
		\frac{nx + r}{2(n-1)x - 2Crh\log(x) + nx - Crh \log(x)},
		\end{align}
		which is less than or equal to $1/3 + \delta$, since $x$ satisfies Equation (\ref{eqn:large_x_formula}).
		\\
		
		For notational convenience, for any $p \in \mathbb{Z}$ and $x \in \mathbb{N}$,
		denote by $E^{p}_x$ the event that the time interval $[prh,(p+1)rh]$ is $x$-Good. 
		\\
		
		If the time interval $[0,rh]$ is $x$-Good, then the drift $Y_1-Y_0$ is at-least the minimum difference
		in the total number of arrivals and departures with marks at most $1/3+\delta$,
		with a factor of $1/n$ to account for the re-normalization. In other words, we have
		
		\begin{align}
		\mathbb{E}[(Y_1 - Y_0) E^{0}_x \vert Y_0 = x] \geq   \mathbb{E}\bigg[ \bigg\lfloor \frac{1}{n}\min_{i \in [-n,n]} \nonumber\left( \widehat{\mathcal{A}}_i[0,rh] - (\mathcal{D}_i[[0,rh]\times[0,1/3+\delta]])E^{0}_x \right) \bigg\rfloor \bigg| Y_0 = x \bigg], \nonumber
		\end{align}
		
		The above inequality can be further simplified using Equation (\ref{eqn:r_arrival_bound}) and
		the fact that for all $x \in \mathbb{R}$, $\lfloor x \rfloor \geq x-1$ as follows.
		
		\begin{multline}
		\mathbb{E}[(Y_1 - Y_0) E^{0}_x \vert Y_0 = x] 	{\geq} \\ \frac{1}{n} \bigg( r\left((\frac{1}{3}+\delta)h + \frac{2\epsilon}{3} \right) -  \mathbb{E}[\max_{i \in [-n,n]} \mathcal{D}_i[[0,rh]\times[0,1/3+\delta]] E^{0}_x \vert Y_0 = x ]  \bigg)     -1.
		\label{eqn:drift_ptve2}
		\end{multline}
		Now since  $\lim_{x \rightarrow \infty}\mathbf{1}_{[0,rh] \text{ is }x\text{-Good}} = 1$ almost surely,
		we have from dominated convergence and Equation (\ref{eqn:r_departure}) that 
		\begin{equation}
		\mathbb{E}[\max_{i \in [-n,n]} \mathcal{D}_i[[0,rh]\times[0,1/3+\delta]] E^{0}_x \vert Y_0 = x ] \leq r\bigg((\frac{1}{3} + \delta)h +  \frac{\epsilon}{3} + o(x) \bigg).
		\label{eqn:drift_ptve3}
		\end{equation}
		
		Combining Equations (\ref{eqn:drift_ptve},\ref{eqn:drift_ptve2}) and (\ref{eqn:drift_ptve3}),
		we get that for all $x$ that satisfy Equation (\ref{eqn:large_x_formula}), we have
		\begin{align*}
		\mathbb{E}[Y_1 - Y_0\vert Y_0 = x] \geq \frac{r}{n} \left(\frac{ \epsilon}{3} -  o(x) \right) -1.
		\end{align*}
		Since $r \geq 4n/\epsilon$, by choosing $x$ sufficiently large such that the $o(x)$ term is less
		than $\epsilon/12$, the drift will be strictly positive.

	\end{proof}
	
	We now complete the Proof of Theorem \ref{thm:transience}
	
	\begin{proof}

		It suffices to establish that the Markov chain $(Y_m)_{m \in \mathbb{N}}$ is transient. 
		This is apparent since $Y_m$ satisfies all the conditions of the main result in \cite{foss_transience},
		which we verify here for completeness. Denote by $\Delta_m := Y_{m+1} - Y_m$. If we establish the following
		three conditions on the one-step drift $\Delta_m$, then transience is concluded in view of
		Lemma \ref{lem:transience_drift} and the main result in \cite{foss_transience}.

		\begin{enumerate}
			\item For each $m \in \mathbb{N}$, $\Delta_m \leq r$ a.s.
			\item For all 
			$x,t \in \mathbb{N}$, $\mathbb{P}[|\Delta_1| \geq t \vert Y_0 = x ] \leq \mathbb{P}[\text{Poi}((2n+1)r(\lambda+1)rh) \geq t]$.
			In particular, for all $s \in \mathbb{R}$ and all $x \in \mathbb{N}$, 
			$\sup_{x \in \mathbb{N}} \mathbb{E}[{e^{s |\Delta_1|} \vert Y_0 = x}] < \infty$.
			\item For all $y \in \{0,1,\cdots,x_0-1\}$ where $x_0$ is defined in Lemma \ref{lem:transience_drift},
			$\mathbb{P}[\exists m \geq 1: Y_m \geq x_0 \vert Y_0 = y] = 1$.
		\end{enumerate}
		
		To finish the proof, we resort to the technique of \cite{foss_transience} which we reproduce for our scenario for completeness. 
		\begin{lemma}
			There exists a $s > 0$ such that  $\tilde{Y}_m = e^{-s(Y_m - x_0)^{+}}$ is a positive super-martingale.
		\end{lemma}
		\begin{proof}
			It suffices to establish that $\mathbb{E}[\tilde{Y}_1 - \tilde{Y}_0 \vert \tilde{Y}_0 = y] \leq 0$, for all $y \in \mathbb{N}$.
			If $y \leq x_0$, the inequality is immediate. Let $y > x_0$ and denote by $a:= y - x_0 > 0$.
			In this case, one can decompose the martingale difference as 
			\begin{align}
			\mathbb{E}[\tilde{Y}_1 - \tilde{Y}_0 \vert \tilde{Y}_0 = y] = \mathbb{E}[e^{-s(a + \Delta_y)^{+}} - e^{-sa}].
			\end{align}
			But since for all $s > 0$ and all $y$, $\mathbb{E}[e^{s|\Delta_y|}] < \infty$,
			the function $s \rightarrow \mathbb{E}[e^{-s(a + \Delta_y)^{+}} - e^{-sa}]$ 
			is bounded and continuous for each fixed $a$. Moreover the derivate at $s=0$ satisfies
			\begin{align*}
			\frac{d}{ds}\mathbb{E}[e^{-s(a + \Delta_y)^{+}} - e^{-sa}] \vert_{s=0} &= - \mathbb{E}[\Delta_y \mathbf{1}_{\Delta_y \geq -a}] -  a\mathbb{P}[\Delta_y \geq -a], \\
			& \stackrel{(a)}{\leq} -\mathbb{E}[\Delta_y \mathbf{1}_{\Delta_y \geq \epsilon}], \\
			&  \stackrel{(b)}{\leq} -\epsilon \mathbb{P}[\Delta_y \geq \epsilon].
			\end{align*} 
			Where $\epsilon$ in step $(a)$ is such that $\inf_{y \geq x_0} \mathbb{P}[\Delta_y \geq \epsilon] > 0$.  It is easy to observe that a $\epsilon > 0$ exists satisfying  $\inf_{y \geq x_0} \mathbb{P}[\Delta_y \geq \epsilon] > 0$ since the network is driven by a collection of Poisson Process and hence Step $(b)$ is a strict inequality. In particular, we have that 
			\begin{align}
			\inf_{a \geq 0} \frac{d}{ds}\mathbb{E}[e^{-s(a + \Delta_y)^{+}} - e^{-sa}] \leq -\epsilon \mathbb{P}[\Delta_y \geq \epsilon] < -\delta^{'},
			\end{align}
			for some $\delta^{'} > 0$. As the function $\mathbb{E}[e^{-s(a + \Delta_y)^{+}} - e^{-sa}]$ is also continuous around $0$ for all $a$, there exists a $s > 0$, such that for all $y \geq x_0$, we have $\mathbb{E}[e^{-s(a + \Delta_y)^{+}} - e^{-sa}] \leq 0$.
		\end{proof}

		Thus, by familiar martingale convergence theorems, $\tilde{Y}_m$ and hence $Y_m$ has an almost sure
		limit as $m$ goes to infinity. To conclude that $Y_m \rightarrow \infty$ almost surely,
		it suffices to establish that $\limsup_{m \rightarrow \infty} \tilde{Y}_m = 0$ almost surely which is equivalent to asserting that $\limsup_{m \rightarrow \infty} Y_m = \infty$ almost surely. It suffices to then establish that for all $b > 0$, $\mathbb{P}[Y_m > b \text{ infinitely often } \vert Y_0 = 0 ] = 1$.  This is in a sense immediate from the description of the dynamics as follows. We know that for all $y \in \mathbb{N}$,  
		$\mathbb{P}[Y_1 - Y_0 = r\vert Y_0 = y] \geq \left(\left(\frac{1}{3}+\delta\right)h + \epsilon\right)^re^{-rh} := \eta > 0$.
		Now let $b \in \mathbb{N}$ be arbitrary and let $\mathcal{E}_k$ be the event that
		$Y_{ k \lceil \frac{b}{r} \rceil } \geq b$.
		From the monotonicity of the dynamics, for all $y \in \mathbb{N}$, we have that
		$\mathbb{P}[\mathcal{E}_k\vert Y_{(k-1)\lceil \frac{b}{r} \rceil} = y] \geq \eta > 0$.
		Thus $\mathbb{P}[\mathcal{E}_{k}^{\mathsf{c}} \vert \mathcal{E}_{1}^{\mathsf{c}}, \cdots , \mathcal{E}_{k-1}^{\mathsf{c}}] \leq 1 - \eta^{'} < 1$. 
		It follows that $\mathbb{P}[ \cap_{k \geq 1} \mathcal{E}_{k}^{\mathsf{c}}] = 0$. 
		This implies that the random variable $\tau_1 := \min\{m \geq 0: Y_m \geq b \vert Y_0 = 0\}$ is a.s. finite.
		Now, from the strong Markov property, the sequence of random times
		$\tau_{k+1}:= \min\{m \geq \tau_{k}+1: Y_m \geq b\}$ are almost surely finite.
		This implies that $\mathbb{P}[Y_m \geq b \text{ infinitely often } \vert Y_0 = y] = 1$, for all $y \in \mathbb{N}$.
		Since $b$ was arbitrary, this implies that $\limsup_{m \rightarrow \infty}Y_m = \infty $ almost surely, thereby concluding the proof.
	\end{proof}

\end{proof}

	\end{appendix}

\section*{Acknowledgements}

The authors  thank several interesting discussions with Thomas Mountford on the model. In particular, the extension of the model to infinite interference support in Proposition \ref{prop:infinite_support_existence} was suggested by him. The authors also thank wonderful discussions with David Gamarnik, which prompted us to consider Proposition \ref{prop:quant_diver_full}. The authors also thank anonymous reviewers for making several suggestions that helped improve the presentation.
\\

A Sankararaman  thanks the generous hospitality of S. Foss for inviting and hosting him for two weeks, when parts of this work was carried out. S.Foss thanks the hospitality of F.Baccelli for inviting and hosting him at The University of Texas at Austin, when parts of this work were carried out.  A. Sankararaman and F. Baccelli were supported by the grant of the Simons Foundations  ($\#197982$ to The University of Texas at Austin) and by the NSF grant NSF-CCF-$1514275$. 

\bibliographystyle{plain}
\bibliography{aoap-interference-queues}

\end{document}